\documentclass[11pt, a4paper,leqno]{amsart}
\usepackage{amsmath,amsthm,amscd,amssymb,amsfonts, amsbsy}
\usepackage{latexsym}
\usepackage{txfonts}
\usepackage{exscale}
\usepackage{mathrsfs}
\usepackage[titletoc]{appendix}
\usepackage[utf8]{inputenc}

\usepackage[colorlinks,citecolor=blue,linkcolor=magenta, hypertexnames=false]{hyperref}
\usepackage[usenames,dvipsnames]{color}

\calclayout
\allowdisplaybreaks

\def\Xint#1{\mathchoice
    {\XXint\displaystyle\textstyle{#1}}%
    {\XXint\textstyle\scriptstyle{#1}}%
    {\XXint\scriptstyle\scriptscriptstyle{#1}}%
    {\XXint\scriptscriptstyle\scriptscriptstyle{#1}}%
    \!\int}
    \def\XXint#1#2#3{{\setbox0=\hbox{$#1{#2#3}{\int}$}
    \vcenter{\hbox{$#2#3$}}\kern-.5\wd0}}
    \def\dashint{\Xint-}

\renewcommand{\chi}{{\bf 1}}

\theoremstyle{plain}
\newtheorem{theorem}[equation]{Theorem}
\newtheorem{lemma}[equation]{Lemma}
\newtheorem{corollary}[equation]{Corollary}
\newtheorem{proposition}[equation]{Proposition}

\theoremstyle{definition}
\newtheorem{definition}[equation]{Definition}

\theoremstyle{remark}
\newtheorem{remark}[equation]{Remark}

\numberwithin{equation}{section}

 \newcommand{\abs}[1]{{\left\lvert{#1}\right\rvert}}
\newcommand{\norm}[1]{{\left\lVert{#1}\right\rVert}}
\newcommand{\br}[1]{{\left({#1}\right)}}
\newcommand{\bbr}[1]{{\left\{{#1}\right\}}}

\numberwithin{equation}{section}

\def \R{ \mathbb{R} }
\def \re{ \mathbb{R} }
\def \N{ \mathbb{N} }

\def \Scal{ \mathcal{S} }
\def \C { \mathcal{C} }
\def \Gcal { \mathcal{G} }
\def \M { \mathcal{M} }

\def \O { \mathcal{O} }
\def \calI { \mathcal{I} }

\def \hh{ \mathrm{H} }
\def \pp{ \mathrm{P} }
\def \Grm{ \mathrm{G} }

\def \iint{\int\!\!\!\int}
\def\div{\mathop{\rm div}}
\renewcommand{\Re}{{\rm Re}\,}

\def\esssup{\mathop\mathrm{\,ess\,sup\,}}

\DeclareMathOperator{\supp}{supp}

\begin{document}
\allowdisplaybreaks
\author{Li Chen}
\address{Li Chen
\\
Instituto de Ciencias Matem\'aticas CSIC-UAM-UC3M-UCM
\\
Consejo Superior de Investigaciones Cient{\'\i}\-fi\-cas
\\
C/ Nicol\'as Cabrera, 13-15
\\
E-28049 Madrid, Spain} \email{li.chen@icmat.es}

\author{Jos\'e Mar{\'\i}a Martell}
\address{Jos\'e Mar{\'\i}a Martell
\\
Instituto de Ciencias Matem\'aticas CSIC-UAM-UC3M-UCM
\\
Consejo Superior de Investigaciones Cient{\'\i}ficas
\\
C/ Nicol\'as Cabrera, 13-15
\\
E-28049 Madrid, Spain} \email{chema.martell@icmat.es}

\author{Cruz Prisuelos-Arribas}

\address{Cruz Prisuelos-Arribas
\\
Instituto de Ciencias Matem\'aticas CSIC-UAM-UC3M-UCM
\\
Consejo Superior de Investigaciones Cient{\'\i}ficas
\\
C/ Nicol\'as Cabrera, 13-15
\\
E-28049 Madrid, Spain} \email{cruz.prisuelos@icmat.es}

\title[Conical square functions for degenerate elliptic
operators]{Conical square functions for degenerate elliptic
operators}

\thanks{The research leading to these results has received funding from the European Research
Council under the European Union's Seventh Framework Programme (FP7/2007-2013)/ ERC
agreement no. 615112 HAPDEGMT. The authors also acknowledge financial support from the Spanish Ministry of Economy and Competitiveness, through the ``Severo Ochoa Programme for Centres of Excellence in R\&D'' (SEV-2015-0554).}

\date{December 14, 2016. \textit{Revised}: \today}


\begin{abstract}
The aim of this paper is to study the boundedness of different conical square functions that arise naturally  from second order divergence form degenerate elliptic operators. More precisely, let $L_w=-w^{-1}\,\div(w\,A\,\nabla)$ where $w\in A_2$ and $A$ is an $n\times n$ bounded, complex-valued,
uniformly elliptic matrix. D.~Cruz-Uribe and C.~Rios solved the $L^2(w)$-Kato square root problem obtaining that $\sqrt{L_w}$ is equivalent to the gradient on $L^2(w)$. The same authors in collaboration with the second named author of this paper studied the $L^p(w)$-boundedness of operators that are naturally associated with $L_w$, such as the functional calculus, Riesz transforms, or vertical square functions. The theory developed admitted also weighted estimates (i.e., estimates in $L^p(v dw)$ for $v\in A_\infty(w)$), and in particular a class of ``degeneracy'' weights $w$ was found in such a way that the classical $L^2$-Kato problem can be solved. In this paper, continuing this line of research, and also that originated in some recent results by the second and third named authors of the current paper, we study the boundedness on $L^p(w)$ and on $L^p(v dw)$, with $v\in A_\infty(w)$, of the conical square functions that one can construct using the heat or Poisson semigroup associated with $L_w$. As a consequence of our methods, we find  a class of degeneracy weights $w$ for which $L^2$-estimates for these conical square functions hold. This opens the door to the study of weighted and unweighted Hardy spaces and of boundary value problems associated with $L_w$.
\end{abstract}

\maketitle

\tableofcontents

\section{Introduction}

Associated with divergence form elliptic operators with complex bounded coefficients, we find certain operators (functional calculi, Riesz transforms, square functions\dots) that are beyond the classical Calder\'on-Zygmund theory. The study of these operators and the development of a Cal\-der\'on-Zygmund theory for them are subjects of big interest, which mainly  came  up after the solution of the Kato conjecture in  \cite{{AHLMcT}}.
A great contribution to a new theory adapted to singular ``non-integral" operators arising from elliptic operators was done  in \cite{Auscher}, where some key ingredients  exploited ideas from  \cite{BlunckKunstmann, HofmannMartell,ACDH}. The related weighted theory was considered by P. Auscher and the second named author of this paper in \cite{AuscherMartell:I, AuscherMartell:II, AuscherMartell:III}. The study of conical square functions, which played a fundamental role in the development of Hardy spaces associated with elliptic operators done by S. Hofmann, A. McIntosh, and S. Mayboroda in \cite{HofmannMayboroda, HofmannMcIntoshMayboroda},  was later taken in \cite{AuscherHofmannMartell} (in the unweighted case) and completed in \cite{MartellPrisuelos}, see also \cite{BCKYY,ZCJY}.

One can also consider degenerate elliptic operators in which case the associated matrix ceases to be uniformly elliptic and presents some controlled degeneracy in the ellipticity condition.  The case in which the degeneracy is an $A_2$ weight was pioneered by E.B.~Fabes, C.~Kenig, and R.~Serapioni in \cite{FabesKenigSerapioni} (with real symmetric matrices) and the corresponding Kato square root problem was solved by D.~Cruz-Uribe and C.~Rios in \cite{Cruz-UribeRios}. The latter amounted to obtaining that the square root of the operator in question is equivalent to the gradient in the weighted space $L^2(w)$, where $w$ is the $A_2$ weight that controls the degeneracy of the matrix of coefficients. A further step was taken in \cite{CruzMartellRios} (see also \cite{Le, YangZ}) where $L^p(w)$ estimates were established for the associated operators (functional calculi, Riesz transforms,  reverse inequalities, vertical  square functions\dots). In fact, using the Calderón-Zygmund theory for singular ``non-integral" operators developed in \cite{AuscherMartell:I} and the notion of off-diagonal estimates on balls introduced in \cite{AuscherMartell:II}, ``weighted'' estimates (i.e, inequalities in $L^p(v dw)$ with $v\in A_\infty(w)$, see Section \ref{section:prelim}) were also proved. As a consequence, it is shown in \cite{CruzMartellRios} that under some additional assumptions on $w$ (written in terms of some controlled higher integrability) one can actually solve the $L^2$-Kato square root problem, that is,  the square root of the operator and the gradient are comparable on $L^2(\R^n)$. 

In this paper we continue these lines of research and study  several  conical square functions associated with the heat or Poisson semigroup generated by a degenerate elliptic operator  (see \eqref{square-H-1}--\eqref{square-P-3} below).   It is well-known that these conical square functions are important objects in the study of Hardy spaces,  as well as in the study of boundary value problems (see, e.g., \cite{Kenig}). Here we present a theory that allows us to prove boundedness on $L^p(w)$ (again $w\in A_2$ controls the degeneracy of the ellipticity condition) --- we note that in \cite{YangZ} there is a similar result for just the conical square functions in \eqref{square-H-1} in a more restricted range.  Additionally, we obtain weighted estimates in $L^p(v dw)$ with $v\in A_\infty(w)$, which in particular lead us to establish $L^2$-estimates under some conditions on $w$. \

In order to state some of the main results we need to introduce some background (see Section \ref{section:prelim} for precise definitions). 
Fix $w\in A_2$, that is, $w$ is a non-negative locally integrable function such that
$$
[w]_{A_2}:=\sup_B \left(\dashint_B w(x) \, dx\right)\left(\dashint_B w(x)^{-1} \, dx\right)
<\infty.
$$
We will write $L^p(w)$ to denote the $L^p$-space with underlying measure $dw(x)=w(x)\,dx$. 
Let $A$ be an $n\times n$ matrix of complex and $L^\infty$-valued coefficients defined on $\R^n$. We assume that
this matrix satisfies the following uniform ellipticity (or \lq\lq
accretivity\rq\rq) condition: there exist $0<\lambda\le\Lambda<\infty$ such that
\begin{equation}
\lambda\,|\xi|^2
\le
\Re A(x)\,\xi\cdot\bar{\xi}
\quad\qquad\mbox{and}\qquad\quad
|A(x)\,\xi\cdot \bar{\zeta}|
\le
\Lambda\,|\xi|\,|\zeta|,
\label{eq:elliptic-intro}
\end{equation}
for all $\xi,\zeta\in\mathbb{C}^n$ and almost every $x\in \R^n$. We have used the notation
$\xi\cdot\zeta=\xi_1\,\zeta_1+\cdots+\xi_n\,\zeta_n$ and therefore
$\xi\cdot\bar{\zeta}$ is the usual inner product in $\mathbb{C}^n$. 
Associated with this matrix and a given weight $w\in A_{2}$, we define the second order divergence
form degenerate elliptic operator
\begin{align*}
L_w f
=
-w^{-1}\div(w\,A\,\nabla f),
\end{align*}
which is understood in the standard weak sense as a maximal-accretive operator on $L^2(w)$ with domain $\mathcal{D}(L_w)$ by means of a sesquilinear form. Equivalently, $A_w:=w\,A$ is a degenerate elliptic matrix, that is,
$$
\lambda\,w(x)\,|\xi|^2
\le
\Re A_w(x)\,\xi\cdot\bar{\xi}
\quad\qquad\mbox{and}\qquad\quad
|A_w(x)\,\xi\cdot \bar{\zeta}|
\le
\Lambda\,w(x)\,|\xi|\,|\zeta|,
$$
for all $\xi,\zeta\in\mathbb{C}^n$ and almost every $x\in \R^n$. 
In \cite{Cruz-UribeRios}  the Kato problem for these degenerate elliptic operators was solved:
$$
\|L_w^{1/2} f\|_{L^2(w)}\approx \|\nabla f\|_{L^2(w)},
$$
for every $f$ in the weighted Sobolev space $H^1(w)$, that is, the completion of $C^\infty_c(\re^n)$  (the space of infinitely differentiable functions with compact support)    with respect to the norm $\|f\|_{H^1(w)}=\|f\|_{L^2(w)}+\|\nabla f\|_{L^2(w)}$. The operator $-L_w$ generates a $C^0$-semigroup $\{e^{-t L_w}\}_{t>0}$ of contractions on $L^2(w)$ which is called the heat semigroup. Using this semigroup and the corresponding Poisson semigroup $\{e^{-t\,\sqrt{L_w}}\}_{t>0}$ (defined using the classical subordination formula) one can consider several conical square functions associated with $L_w$. Here, for the sake of conciseness, we just introduce two of them (in the body of the paper we study more general versions), one associated with the heat semigroup and another with the Poisson semigroup:
\begin{align}
\mathcal{S}_{\hh}^{L_w}f(x) & = \left(\iint_{\Gamma(x)}|t^2L_w e^{-t^2L_w}f(y)|^2  \frac{dw(y) \, dt}{tw(B(y,t))}\right)^{\frac{1}{2}}
\label{SH}
\\[4pt]
\mathcal{S}_{\pp}^{L_w}f(x)
&=
\left(\iint_{\Gamma(x)}|t^2 L_w\, e^{-t\sqrt{L_w}}f(y)|^2 \frac{dw(y) \, dt}{tw(B(y,t))}\right)^{\frac{1}{2}},
\label{SP}
\end{align}
where $\Gamma(x): =\{(y,t)\in \R^{n+1}_+: |x-y|<t\}$ denotes the cone (of aperture 1) with vertex at $x\in\R^n$ and $w(B(y,t))=\int_{B(y,t)}w(x)\,dx$. 
Taking these as the model of more general conical square functions, the goal of this paper is to find ranges of $p$'s for which $\mathcal{S}_{\hh}^{L_w}$ and/or $\mathcal{S}_{\pp}^{L_w}$ are bounded on $L^p(w)$. Also we will obtain the corresponding weighted norm inequalities, that is, estimates in $L^p(v dw)$ for some range of $p$'s and some collection of $v\in A_\infty(w)$ (see Section \ref{section:prelim} for the precise definitions). As a consequence we will also establish purely unweighted inequalities, that is, estimates in $L^p(\re^n)$ (the $L^p$-space associated with the Lebesgue measure in $\re^n$). As a sample of our results, let us present one containing some of these estimates in the unweighted space $L^2(\re^n)$ (see Corollaries \ref{cor-unweightedHeat}, \ref{cor-unweightedPoisson}, and \ref{cor-PW} for complete statements).  We note that the boundedness on $L^2(\re^n)$ of the conical square functions \eqref{SH} and \eqref{SP} in the uniformly elliptic case (i.e, when $w\equiv 1$) follows at once from the fact that the associated divergence form elliptic operator  has a bounded functional calculus on $L^2(\re^n)$. Here, in contrast, the $L^2(\re^n)$ theory for degenerate elliptic operators becomes non-trivial and our results open the door to considering, for instance, boundary value problems associated with $L_w$ with data in $L^2(\re^n)$.

\begin{theorem}\label{theor:intro}
Let $A$ be  an $n\times n$ complex- valued matrix that satisfies the uniform ellipticity condition \eqref{eq:elliptic-intro}. 

\begin{list}{$(\theenumi)$}{\usecounter{enumi}\leftmargin=1cm \labelwidth=1cm\itemsep=0.2cm\topsep=.2cm \renewcommand{\theenumi}{\alph{enumi}}}

\item Consider $L_w =-w^{-1}\div(w\,A\,\nabla)$,  a degenerate elliptic operator as above, with $w\in A_2$. 

\begin{list}{$\bullet$}{\leftmargin=.6cm \labelwidth=1cm\itemsep=0.2cm\topsep=.2cm}

\item Given $1\le r\le 2$, if $w\in A_r\cap RH_{\frac{n\,r}2+1}$ then $\mathcal{S}_{\hh}^{L_w}$ is bounded on $L^2(\re^n)$.

\item Given $1\le r\le \min\big\{2,1+\frac4n\big\}$, if $w\in A_r\cap RH_{\frac{n\,r}2+1}$ then $\mathcal{S}_{\pp}^{L_w}$ is bounded on $L^2(\re^n)$.

\end{list}

\item Consider $L_\gamma=-|x|^\gamma\div(|x|^{-\gamma}\,A\,\nabla)$ with $-n<\gamma<n$ \textup{(}hence $|x|^{-\gamma}\in A_2$\textup{)}.
\begin{list}{$\bullet$}{\leftmargin=.6cm \labelwidth=1cm\itemsep=0.2cm\topsep=.2cm}

\item If $-n<\gamma<\frac{2\,n}{n+2}$, then $\mathcal{S}_{\hh}^{L_\gamma}$ is bounded on $L^2(\re^n)$.

\item If $-\min\{4,n\}<\gamma<\frac{2\,n}{n+2}$, then $\mathcal{S}_{\pp}^{L_\gamma}$ is bounded on $L^2(\re^n)$.
\end{list}
\end{list}
\end{theorem}

\medskip

The plan of this paper is as follows. In Section \ref{section:prelim} we present some of the preliminaries needed to state our main results in Section \ref{section:main}. In Section \ref{sectionauxiliaryresults}  we first recall some earlier results concerning off-diagonal estimates for the heat semigroup in question. We then obtain some ``change of angle'' formulas that allow us to compare weighted tent-space norms for cones with different apertures. This control is done in weighted spaces $L^p(vdw)$ with $v\in A_\infty(w)$ and $w\in A_\infty$ and we obtain quantitative bounds depending on the ratio between the  apertures of the cones. We also introduce some $p$-adapted weighted Carleson condition and compare it with some weighted tent-space norms in weighted spaces. Section \ref{section:proofs} contains the proofs of the main results. In Section \ref{section:unweighted} we obtain unweighted estimates, proving in particular Theorem \ref{theor:intro} above. Finally, in the appendix, we formulate some extrapolation  results inspired by those in \cite{CruzMartellPerez} but with the weighted measure space $(\R^n,w)$ replacing $(\re^n,dx)$. The proofs are simply sketched as they follow the lines of the equivalent ones in \cite{CruzMartellPerez}.

\section{Preliminaries}\label{section:prelim}

We turn now to  introducing   some notation and set up our background. Throughout the paper $n$ will denote the dimension of the underlying space $\re^n$
and we will always assume $n\geq 2$.  We write $dx$ to denote the usual Lebesgue measure in $\R^n$ and $L^p(\re^n)$ or simply $L^p$ for $L^p(\re^n,dx)$. 

Given a ball $B$, let $r_B$ denote the radius of $B$.  We write $\lambda B$ for the concentric ball with radius $r_{\lambda B} = \lambda r_B$. Moreover, we set $C_1(B)=4B$ and, for $j\geq 2$, $C_j(B)=2^{j+1}B \setminus 2^j B$.

If we write $\Theta_1\lesssim \Theta_2$ we
mean that there exists a constant $C$ such that $\Theta_1\leq C \Theta_2$.  We write
$\Theta_1\approx \Theta_2$ if $\Theta_1\lesssim \Theta_2$ and $\Theta_2\lesssim \Theta_1$.  The
constant $C$ in these estimates may depend on the dimension $n$ and other (fixed)
parameters that should be clear from the context.  All constants,
explicit or implicit, may change at each appearance.

\subsection{Weights}

By a weight $w$ we mean a non-negative, locally integrable function.
For brevity, we will often write $dw$ for $w\,dx$. In particular, we write $w(E)=\int_E\,dw$ and $L^p(w)=L^{p}(\re^n,dw)$.  We will use the following notation for averages:    given a ball $B$ we write 
\[ \dashint_B f\,dw = \frac{1}{w(B)}\int_B f\,dw
\qquad\text{or}\qquad
\dashint_B f\,dx = \frac{1}{|B|}\int_B f\,dx, \]
and, for $j\geq 2$, we set
$$
\dashint_{C_j(B)} f\, dw= \frac{1}{w(2^{j+1}B)} \int_{C_j(B)} f\,dw.
$$

We state some definitions and basic properties of Muckenhoupt
weights.  For further details,
see~\cite{Duo, GCRF85, Grafakos2}. We say that $w\in A_p$, $1<p<\infty$, if
\[ [w]_{A_p} := \sup_B \left(\dashint_B w(x)\,dx\right) \left(\dashint_B
  w(x)^{1-p'}\,dx\right)^{p-1} < \infty. \]
Here and below the sups run over the collection of balls $B\subset\re^n$.
When $p=1$, we say $w\in A_1$ if
\[ [w]_{A_1} := \sup_B \left(\dashint_B w(x)\,dx\right)  \left(\esssup_{x\in B} w(x)^{-
1}\right)<
\infty.  \]
We say $w\in RH_s$, $1<s<\infty$ if
\[ [w]_{RH_s} := \sup_B \left(\dashint_B w(x)\,dx\right )^{-1}
\left(\dashint_B w(x)^s\,dx\right )^{1/s} < \infty, \]
and
\[ [w]_{RH_\infty} := \sup_B\left(\dashint_B w(x)\,dx\right)^{-1} \left(\esssup_{x\in B} w(x)\right)  <
\infty.  \]
Let
\[ A_\infty := \bigcup_{1\leq p <\infty} A_p  = \bigcup_{1<s\le \infty}
RH_s.  \]
The classes $A_p$, $1\leq p<\infty$, or $RH_s$, $1<s\le\infty$, may be equivalently defined using cubes in $\re^n$
(in place of balls), in which scenario $[w]_{A_p}\approx [w]_{A_p}^{\rm cubes}$ with
implicit constants depending only on $n$ and $p$.

An important property is that if $w\in RH_s$, $1<s\le\infty$, 
\begin{align}\label{pesosineqw:RHq}
\frac{w(E)}{w(B)}\leq  [w]_{RH_{s}}\br{\frac{|E|}{|B|}}^{\frac{1}{s'}}, \quad \forall\,E\subset B,
\end{align}
where $B$ is any ball in $\re^n$. 
Analogously, if
$w\in A_p$,$1\leq p<\infty$, then
\begin{align}\label{pesosineqw:Ap}
 \br{\frac{|E|}{|B|}}^{p}\le [w]_{A_{p}}\frac{w(E)}{w(B)}, \quad \forall\,E\subset B.
\end{align}
A consequence of this, is that $A_p$ weights are doubling measures: 
given $w\in A_p$, for all $\tau\ge 1$ and any ball $B$, $w(\tau B)\le [w]_{A_p} \tau^{pn} w(B).$ This property will be used throughout the paper.

%

As a consequence of this doubling property, we have that with the ordinary Euclidean distance
$|\cdot|$, $(\R^n,dw,|\cdot|)$ is a space of homogeneous type.
In this setting we can define new  classes of weights  $A_p(w)$
and $RH_s(w)$ by replacing Lebesgue measure in the definitions above with
$dw$: e.g., $v\in A_p(w)$ if
\[ [v]_{A_p(w)} = \sup_B \left(\dashint_B v(x)\,dw\right) \left(\dashint_B
  v(x)^{1-p'}\,dw\right)^{p-1} < \infty. \]
From these definitions, it follows at once  that there is a
``duality'' relationship between the weighted and unweighted
$A_p$ and $RH_s$ conditions:  $v=w^{-1} \in A_p(w)$ if and only if $w \in
RH_{p'}$ and $v=w^{-1}\in RH_s(w)$ if and only if $w\in A_{s'}$.

For every measurable set $E\subset \R^n$, we write $vw(E)=(v dw)(E)=\int_E v dw$ and $L^p(v dw)=L^p(\re^n, v(x)\,w(x)\,dx)$. In this direction, 
for every $w\in A_p$, $v\in A_q(w)$, $1\le p,q<\infty$, it follows that
\begin{align}\label{pesosineq:Ap}
\left(\frac{|E|}{|B|}\right)^{p\,q}
\le
[w]_{A_{p}}^q\left(\frac{w(E)}{w(B)}\right)^{q}
\le
[w]_{A_{p}}^q[v]_{A_{q}(w)}
\frac{vw(E)}{vw(B)},\quad \forall\,E\subset B.
\end{align}
Analogously, if $w\in RH_{p}$ and $v\in RH_{q}(w)$, $1< p,q\le\infty$, one has 
\begin{align}\label{pesosineq:RHq}
\frac{vw(E)}{vw(B)}
\leq
[v]_{RH_{q}(w)}\left(\frac{w(E)}{w(B)}\right)^{\frac{1}{q'}}\leq [v]_{RH_{q}(w)}[w]_{RH_{p}}^{\frac1{q'}}\left(\frac{|E|}{|B|}\right)^{\frac{1}{p'\,q'}},\quad \forall\,E\subset B.
\end{align}

\begin{remark}\label{remark:weightedHLM}
Consider the Hardy-Littlewood maximal function
$$
\mathcal{M} f(x):=\sup_{B\ni x} \dashint_B|f(y)|\,dy.
$$
By the classical theory of weights, $w\in A_p$, $1<p<\infty$, if and only if, $\mathcal{M}$ is bounded on $L^p(w)$.

On the other hand, given $w\in A_\infty$, we can introduce the weighted maximal operator $\mathcal{M}^w$:
\begin{align}\label{weightedHLM}
\mathcal{M}^wf(x):=\sup_{B\ni x}\dashint_B |f(y)|\,dw(y).
\end{align} 
Since $w$ is  a  doubling measure, one can also show that $v\in A_p(w)$, $1<p<\infty$, if and only if, $\mathcal{M}^w$ is bounded on $L^p(v dw)$. 
\end{remark}

We continue by introducing  some important notation.
Weights in the $A_p$ and $RH_s$ classes have a self-improving
property: if $w\in A_p$, there exists $\epsilon>0$ such that $w\in
A_{p-\epsilon}$, and similarly if $w\in RH_s$, then $w\in
RH_{s+\delta}$ for some $\delta>0$.  Hereafter, given $w\in  A_{\infty}$, let
\begin{equation}
r_w=\inf\big\{p:\ w\in A_p\big\}, \qquad s_w=\inf\big\{q:\ w\in RH_{q'}\big\}.
\label{eq:defi:rw}
\end{equation}
Note that according to our definition $s_w$ is the conjugated exponent of the one defined in \cite[Lemma 4.1]{AuscherMartell:I}.
Given $0\le p_0<q_0\le \infty$ and $w\in A_{\infty}$,  \cite[Lemma 4.1]{AuscherMartell:I} implies that
\begin{align}\label{intervalrs}
\mathcal{W}_w(p_0,q_0):=\left\{p\in (p_0, q_0): \ w\in A_{\frac{p}{p_0}}\cap RH_{\left(\frac{q_0}{p}\right)'}\right\}
=
\left(p_0r_w,\frac{q_0}{s_w}\right).
\end{align}
If $p_0=0$ and $q_0<\infty$ it is understood that the only condition that stays is $w\in RH_{\left(\frac{q_0}{p}\right)'}$. Analogously, 
if $0<p_0$ and $q_0=\infty$ the only assumption is $w\in A_{\frac{p}{p_0}}$. Finally $\mathcal{W}_w(0,\infty)=(0,\infty)$.

In the same way, for a weight $v\in A_{\infty}(w)$, with $w\in A_\infty$ we set
$$
\mathfrak{r}_v(w):=\inf\big\{r:\ v\in A_{r}(w)\big\}\quad \textrm{and}\quad
\mathfrak{s}_v(w):=\inf\big\{s:\ v\in RH_{s'}(w)\big\}.
$$
For $0\le p_0<q_0\le \infty$ and $v\in A_{\infty}(w)$,  following mutatis mutandis \cite[Lemma 4.1]{AuscherMartell:I}, we have
\begin{align}\label{intervalrsw}
\mathcal{W}_v^w(p_0,q_0):=\left\{p\in(p_0,q_0):\ v\in A_{\frac{p}{p_0}}(w)\cap RH_{\left(\frac{q_0}{p}\right)'}(w)\right\}
=
\left(p_0\mathfrak{r}_v(w),\frac{q_0}{\mathfrak{s}_v(w)}\right).
\end{align}
If $p_0=0$ and $q_0<\infty$, as before, it is understood that the only condition that stays is $v\in RH_{\left(\frac{q_0}{p}\right)'}(w)$. Analogously, if $0<p_0$ and $q_0=\infty$ the only assumption is $v\in A_{\frac{p}{p_0}}(w)$. Finally $\mathcal{W}_v^w(0,\infty)=(0,\infty)$.

\subsection{Degenerate elliptic operators}
Let $A$ be an $n\times n$ matrix of complex and
$L^\infty$-valued coefficients defined on $\R^n$. We assume that
this matrix satisfies the uniform ellipticity condition as introduced in \eqref{eq:elliptic-intro}. 
Associated with this matrix and a given weight $w\in A_{2}$ (which is fixed from now on) we define the second order divergence
form degenerate elliptic operator
\begin{align}\label{degenerateL}
L_w f
=
-w^{-1}\div(w\,A\,\nabla f),
\end{align}
which is understood in the standard weak sense as a maximal-accretive operator on $L^2(w)$ with domain $\mathcal{D}(L_w)$ by means of a sesquilinear form.
These operators
were developed in \cite{FabesJerisonKenig1, FabesJerisonKenig2,FabesKenigSerapioni,Cruz-UribeRios}  and we refer the reader there for complete
details.  Here we borrow some of their results. The operator $-L_w$ generates a $C^0$-semigroup $\{e^{-t L_w}\}_{t>0}$ of contractions on $L^2(w)$ which is called the heat semigroup. 

As in \cite{Auscher, AuscherMartell:II, CruzMartellRios}, we denote by $(p_-(L_w),p_+(L_w))$ the maximal open interval on which the heat semigroup $\{e^{-tL_w}\}_{t>0}$ is uniformly bounded on $L^p(w)$:
\begin{align}\label{p-}
p_-(L_w) &:= \inf\left\{p\in(1,\infty): \sup_{t>0} \|e^{-t^2L_w}\|_{L^p(w)\rightarrow L^p(w)}< \infty\right\},
\\[4pt]
p_+(L_w)& := \sup\left\{p\in(1,\infty) : \sup_{t>0} \|e^{-t^2L_w}\|_{L^p(w)\rightarrow L^p(w)}< \infty\right\}.
\label{p+}
\end{align}
Note that in place of the semigroup $\{e^{-t L_w}\}_{t>0}$ we are using its rescaling $\{e^{-t^2 L_w}\}_{t>0}$. We do so since all the ``heat'' square functions that we consider below are written using the latter and also because in the context of the off-diagonal estimates discussed below it will simplify some computations. 
According to \cite{CruzMartellRios}, 
\begin{equation}\label{p-p+} 
p_-(L_w) \leq (2^*_w)'<2<2^*_w\le p_+(L_w),
\end{equation}
where
$2_w^*=\frac{2\,n\,r_w}{n\,r_w-2}$ if $2<n\,r_w$ and $2_w^*=\infty$ otherwise.

Let us also introduce for every $K\in\mathbb{N}_0:=\N\cup \{0\}$, 
\begin{equation}\label{indexPoisson}
\br{p_+(L_w)}_w^{K,*}:=
\left\{
\begin{array}{ll}
\dfrac{p_+(L_w)n r_w}{n r_w-(2K+1)p_+(L_w)}, &\quad\mbox{ if}\quad(2K+1)p_+(L_w)<nr_w,
\\[10pt]
\infty, &\quad\mbox{ if}\quad(2K+1)p_+(L_w)\ge nr_w.
\end{array}
\right.
\end{equation}
When $K=0$, we write $\br{p_+(L_w)}_w^{*}:=\br{p_+(L_w)}_w^{0,*}$.

Using the heat semigroup and the classical subordination formula,  or the functional calculus for $L_w$,  we can also consider the Poisson semigroup:
\begin{align}\label{formula:SF}
e^{-t\sqrt{L_w}}=\frac1{\sqrt \pi}\int_0^{\infty}e^{-u}u^{\frac{1}{2}}e^{-\frac{t^2}{4u}L_w}\frac{du}{u}.
\end{align}

\subsection{Conical square functions}

One can define different conical square functions associated with $L_w$ as above which all have an expression of the form
$$
\mathcal{Q}^{L_w}f(x)
=\left(\iint_{\Gamma(x)}|T_t^{L_w} f(y)|^2 \frac{dw(y) \, dt}{tw(B(y,t))}\right)^{\frac{1}{2}},
\qquad
x\in\R^n,
$$
where $\Gamma(x): =\{(y,t)\in \R^{n+1}_+: |x-y|<t\}$ denotes the cone (of aperture 1) with vertex at $x\in\R^n$.  More precisely, we introduce the following conical square functions written in terms of the heat semigroup $\{e^{-t L_w}\}_{t>0}$ (hence the subscript $\hh$): for every $m\in \mathbb{N}$,
\begin{align} \label{square-H-1}
\mathcal{S}_{m,\hh}^{L_w}f(x) & = \left(\iint_{\Gamma(x)}|(t^2L_w)^{m} e^{-t^2L_w}f(y)|^2  \frac{dw(y) \, dt}{tw(B(y,t))}\right)^{\frac{1}{2}},
\end{align}
and, for every $m\in \mathbb{N}_0:=\mathbb{N}\cup\{0\}$,
\begin{align}
\mathrm{G}_{m,\hh}^{L_w}f(x)& =\left(\iint_{\Gamma(x)}|t\nabla_y(t^2L_w)^m e^{-t^2L_w}f(y)|^2 \frac{dw(y) \, dt}{tw(B(y,t))}\right)^{\frac{1}{2}},
\label{square-H-2}\\[4pt]
\mathcal{G}_{m,\hh}^{L_w}f(x)&
=
\left(\iint_{\Gamma(x)}|t\nabla_{y,t}(t^2L_w)^m e^{-t^2L_w}f(y)|^2 \frac{dw(y) \, dt}{tw(B(y,t))}\right)^{\frac{1}{2}}.
\label{square-H-3}
\end{align}

In the same manner, let us consider weighted conical square functions associated with the Poisson semigroup $\{e^{-t \sqrt{L_w}}\}_{t>0}$ (hence the subscript $\pp$):  given $K\in \mathbb{N}$,
\begin{align}
\mathcal{S}_{K,\pp}^{L_w}f(x)
&=
\left(\iint_{\Gamma(x)}|(t\sqrt{L_w}\,)^{2K} e^{-t\sqrt{L_w}}f(y)|^2 \frac{dw(y) \, dt}{tw(B(y,t))}\right)^{\frac{1}{2}},
\label{square-P-1}
\end{align}
and for every $K\in \mathbb{N}_0$,
\begin{align}
\mathrm{G}_{K,\pp}^{L_w}f(x)
&=\left(\iint_{\Gamma(x)}|t\nabla_y (t\sqrt{L_w}\,)^{2K} e^{-t\sqrt{L_w}}f(y)|^2 \frac{dw(y) \, dt}{tw(B(y,t))}\right)^{\frac{1}{2}},
\label{square-P-2}
\\[4pt]
\mathcal{G}_{K,\pp}^{L_w}f(x)
&=
\left(\iint_{\Gamma(x)}|t\nabla_{y,t}(t\sqrt{L_w}\,)^{2K} e^{-t\sqrt{L_w}}f(y)|^2 \frac{dw(y) \, dt}{tw(B(y,t))}\right)^{\frac{1}{2}}.
\label{square-P-3}
\end{align}
Corresponding to the cases $m=0$ or $K=0$ we simply write $\mathrm{G}_{\hh}^{L_w}f:=\mathrm{G}_{0,\hh}^{L_w}f$,
$\mathcal{G}_{\hh}^{L_w}f:=\mathcal{G}_{0,\hh}^{L_w}f$,  $\mathrm{G}_{\pp}^{L_w}f:=\mathrm{G}_{0,\pp}^{L_w}f$, and
$\mathcal{G}_{\pp}^{L_w}f:=\mathcal{G}_{0,\pp}^{L_w}f$. Besides, we set $\Scal_{\hh}^{L_w}f:=\Scal_{1,\hh}^{L_w}f$ and $\Scal_{\pp}^{L_w}f:=\Scal_{1,\pp}^{L_w}f$.

Let us observe that in all the above conical square functions the apertures of the cones are taken to be $1$. One could define conical square functions with any given aperture, but these are equivalent in $L^p(w)$ or in $L^p(vdw)$ for every $0<p<\infty$ and $v\in A_\infty(w)$ by the 
change of angle formulas obtained in Proposition \ref{prop:alpha}. 

Notice also that when comparing the conical square functions associated with the heat and Poisson semigroups the parameter $m$ is in correspondence with $K$ (and not with $2K$) since we can rewrite $(t\sqrt{L_w}\,)^{2K}$ as $(t^2 L_w)^{K}$. This is also reflected in the fact that, for instance, $\mathcal{S}_{K,\pp}^{L_w}f$ is controlled (in norm) by $\mathcal{S}_{K,\hh}^{L_w}f$, cf.~Theorem \ref{theor:control-SF-Poisson} part $(b)$. One could define conical square functions for the Poisson semigroup with $(t\sqrt{L_w}\,)^{2K+1}$ in front, which in terms of the heat semigroup, would mean to put $(t^2 L_w)^{m+\frac12}$. The corresponding square functions would also fit into the theory developed in this paper, with appropriate changes. One of the difficulties that will appear is that $(t^2 L_w)^{m+\frac12}e^{-t^2L_w}$ satisfies off-diagonal estimates with polynomial decay and in that scenario one would get restrictions in the range of boundedness or comparison. This will not be pursued in the present paper.

\medskip

\section{Main results}\label{section:main}

We will obtain weighted norm inequalities and boundedness for the square functions  presented in \eqref{square-H-1}-\eqref{square-P-3} in weighted measure spaces. The word  ``weighted" refers to two different concepts here, so we explain them better.
First, note  that the square functions that we consider are associated  with a degenerate elliptic operator, $L_w$, defined as in \eqref{degenerateL}. Thus, the natural underlying measure space is the ``weighted'' space $(\re^n,w)$. For this reason, the square functions introduced above incorporate $w$ in their definition. In this way, an $L^p(w)$ estimate for any of these square functions can be written as a norm of a function in $\re_+^{n+1}$ in the corresponding tent space whose underlying measure is $dw\,dt/t$. Our goal is to obtain estimates in $L^p(w)$ for some range of $p$'s and also to obtain ``weighted'' estimates, that is, estimates in $L^p(v dw)$ with $v\in A_\infty(w)$.

Our first two results establish the boundedness of the conical square functions associated with the heat and Poisson semigroup: 

\begin{theorem}\label{thm:SF-Heat}
Let $L_w$ be a degenerate elliptic operator with $w\in A_2$ and let $v\in A_{\infty}(w)$.

\begin{list}{$(\theenumi)$}{\usecounter{enumi}\leftmargin=1cm \labelwidth=1cm\itemsep=0.2cm\topsep=.2cm \renewcommand{\theenumi}{\alph{enumi}}}

\item For every $m\in\mathbb{N}$, $\Scal_{m,\hh}^{L_w}$ is bounded on  $L^p(vdw)$ for all $p\in \mathcal{W}_v^w(p_-(L_w),\infty)$.

\item For every $m\in\mathbb{N}_0$, $\Grm_{m, \hh}^{L_w}$, and $\Gcal_{m, \hh}^{L_w}$ are bounded on $L^p(vdw)$ for all $p\in \mathcal{W}_v^w(p_-(L_w),\infty)$.
\end{list}
Equivalently, all the previous square functions are bounded on $L^p(vdw)$ for every $p_-(L_w)<p<\infty$ and every  $v\in A_{\frac{p}{p_-(L_w)}}(w)$. In particular, letting $v\equiv 1$,  all these square functions are bounded on $L^p(w)$ for every $p_-(L_w)<p<\infty$.
\end{theorem}

\begin{theorem}\label{thm:SF-Poisson}
Let $L_w$ be a degenerate elliptic operator with $w\in A_2$ and let $v\in A_{\infty}(w)$.

\begin{list}{$(\theenumi)$}{\usecounter{enumi}\leftmargin=1cm \labelwidth=1cm\itemsep=0.2cm\topsep=.2cm \renewcommand{\theenumi}{\alph{enumi}}}

\item Given $K\in\mathbb{N}$, $\Scal_{K,\pp}^{L_w}$ is bounded on $L^p(vdw)$ for all $p\in \mathcal{W}_v^w(p_-(L_w),\br{p_+(L_w)}_w^{K,*})$.

\item Given $K\in\mathbb{N}_0$,  $\Gcal_{K,\pp}^{L_w}$ and $\Grm_{K,\pp}^{L_w}$ are bounded on $L^p(vdw)$ for all $p\in \mathcal{W}_v^w(p_-(L_w),\br{p_+(L_w)}_w^{K,*})$.
\end{list}
In particular, letting $v\equiv1$, $\Scal_{K,\pp}^{L_w}$ for $K\in\mathbb{N}$, and $\Gcal_{K,\pp}^{L_w}$ and $\Grm_{K,\pp}^{L_w}$ for $K\in\mathbb{N}_0$, are bounded on $L^p(w)$ for every $p_-(L_w)<p<\br{p_+(L_w)}_w^{K,*}$.
\end{theorem}

These two results will be proved with the help of some estimates, interesting  in  their own right, which establish that all the previous square functions can be controlled (in the $L^p(v dw)$-norm) by either $\Scal_{\hh}^{L_w}$ or $\Gcal_{\hh}^{L_w}$. Hence matters reduce to proving the boundedness of  these two operators.

In the following two results we  compare  the square functions associated with the heat and Poisson semigroups, respectively.

\begin{theorem}\label{theor:control-SF-Heat}
Let $L_w$ be a degenerate elliptic operator with $w\in A_2$ and take an arbitrary $f\in L^2(w)$. 
\begin{list}{$(\theenumi)$}{\usecounter{enumi}\leftmargin=1cm \labelwidth=1cm\itemsep=0.2cm\topsep=.2cm \renewcommand{\theenumi}{\alph{enumi}}}

\item $\mathcal{S}_{\hh}^{L_w}f(x)\leq\frac{1}{2} \mathcal{G}^{L_w}_{\hh}f(x)$ and $\Grm_{m,\hh}^{L_w}f(x)\le \Gcal_{m, \hh}^{L_w}f(x)$, for every $x\in\mathbb{R}^n $ and  for all $m\in\mathbb{N}_0$.

\item Given $m\in\mathbb{N}$,  $\displaystyle \|\Gcal_{m,\hh}^{L_w}f\|_{L^p(vdw)}\lesssim \|\Scal_{m,\hh}^{L_w}f\|_{L^p(vdw)}$, for all $v\in A_{\infty}(w)$ and $0<p<\infty$.

\item Given $m\in\mathbb{N}$,  $\displaystyle \|\Scal_{m+1,\hh}^{L_w}f\|_{L^p(vdw)}\lesssim \|\Scal_{m,\hh}^{L_w}f\|_{L^p(vdw)}$, for all $v\in A_{\infty}(w)$ and $0<p<\infty$.
\end{list}

As a consequence, for every $m\in\mathbb{N}$, and for all $v\in A_{\infty}(w)$ and $0<p<\infty$ there holds
\begin{equation}\label{eq:all-SH}
\|\Scal_{m,\hh}^{L_w}f\|_{L^p(vdw)}
+
\|G_{m,\hh}^{L_w}f\|_{L^p(vdw)}
+
\|\Gcal_{m,\hh}^{L_w}f\|_{L^p(vdw)}
\lesssim
\|\Scal_{\hh}^{L_w}f\|_{L^p(vdw)}.
\end{equation}

\end{theorem}

\medskip

\begin{theorem}\label{theor:control-SF-Poisson} 
Let $L_w$ be a degenerate elliptic operator with $w\in A_2$ and take an arbitrary $f\in L^2(w)$. 
\begin{list}{$(\theenumi)$}{\usecounter{enumi}\leftmargin=1cm \labelwidth=1cm\itemsep=0.2cm\topsep=.2cm \renewcommand{\theenumi}{\alph{enumi}}}

\item $\Grm_{K,\pp}^{L_w}f(x)\le \Gcal_{K, \pp}^{L_w}f(x)$, for every $x\in\mathbb{R}^n $ and  for all $K\in\mathbb{N}_0$.

\item Given $K\in\mathbb{N}$, $\displaystyle \|\Scal_{K,\pp}^{L_w}f\|_{L^p(vdw)}\lesssim \|\Scal_{K,\hh}^{L_w}f\|_{L^p(vdw)}$, for all $v\in A_{\infty}(w)$ and $p\in\mathcal{W}_v^w(0,\br{p_+(L_w)}_w^{K,*})$.

\item $\displaystyle \|\Gcal_{\pp}^{L_w}f\|_{L^p(vdw)}\lesssim \|\Gcal_{\hh}^{L_w}f\|_{L^p(vdw)}$, for all $v\in A_{\infty}(w)$ and $p\in\mathcal{W}_v^w(0,(p_+(L_w))^{*}_w)$.

 \item Given $K\in\mathbb{N}$, $\displaystyle \|\Gcal_{K,\pp}^{L_w}f\|_{L^p(vdw)}\lesssim \|\Scal_{K,\hh}^{L_w}f\|_{L^p(vdw)}$, for all $v\in A_{\infty}(w)$ and $p\in\mathcal{W}_v^w(0,\br{p_+(L_w)}_w^{K,*})$. 
\end{list}

 As a consequence, for every $K\in\mathbb{N}$, and for all $v\in A_{\infty}(w)$ and $p\in\mathcal{W}_v^w(0,\br{p_+(L_w)}_w^{K,*})$ there holds
\begin{equation}\label{eq:all-SP}
\|\Scal_{K,\pp}^{L_w}f\|_{L^p(vdw)}
+
\|G_{K,\pp}^{L_w}f\|_{L^p(vdw)}
+
\|\Gcal_{K,\pp}^{L_w}f\|_{L^p(vdw)}
\lesssim
\|\Scal_{\hh}^{L_w}f\|_{L^p(vdw)}.
\end{equation}

\end{theorem}

\medskip

The proofs of Theorems \ref{thm:SF-Heat}--\ref{theor:control-SF-Poisson}  are in Section \ref{section:main}. Let us note that the non-degenerate versions (i.e., the case when $w\equiv 1$) were established in \cite{MartellPrisuelos} (see also \cite{AuscherHofmannMartell}) and some of the ideas of this paper are borrowed  from there.

\section{Auxiliary results}\label{sectionauxiliaryresults}

\subsection{Off-diagonal estimates}\label{subsectionoff-diagonal}
We recall here the concept of weighted off-diagonal estimates on balls. For more definitions of weighted off-diagonal estimates and a careful study of their properties, we refer to \cite{AuscherMartell:II}. 

\begin{definition}
Let $\{T_t\}_{t>0}$ be a family of linear operators and let $1\leq p\leq q\leq \infty$. Given $w\in A_\infty$, we say that $\{T_t\}_{t>0}$ satisfies $L^p(w)-L^q(w)$ off-diagonal estimates on balls, which will be denoted by
$T_t\in \mathcal O(L^p(w)-L^q(w))$, if there exist $\theta_1,\theta_2,c>0$ such that for any $t>0$ and for any ball $B$ with radius $r_B$, 
\begin{equation}\label{off:BtoB}
\br{\dashint_B \abs{T_t(f \chi_B)}^q\, dw}^{1/q} \lesssim \Upsilon\br{\frac{r_B}{\sqrt t}}^{\theta_2} \br{\dashint_B |f|^p\, dw}^{1/p},
\end{equation}
and for $j\geq 2$,
\begin{equation}\label{off:CtoB}
\br{\dashint_B \abs{T_t(f \chi_{C_j(B)})}^q\, dw}^{1/q} \lesssim 2^{j\theta_1} \Upsilon\br{\frac{2^j r_B}{\sqrt t}}^{\theta_2} 
e^{-\frac{c4^j r_B^2}{t}} \br{\dashint_{C_j(B)} |f|^p\, dw}^{1/p},
\end{equation}
and
\begin{equation}\label{off:BtoC}
\br{\dashint_{C_j(B)} \abs{T_t(f \chi_B)}^q\, dw}^{1/q} \lesssim 2^{j\theta_1} \Upsilon\br{\frac{2^j r_B}{\sqrt t}}^{\theta_2} e^{-\frac{c4^j r_B^2}{t}} \br{\dashint_B |f|^p\,dw}^{1/p},
\end{equation}
where $\Upsilon(s):=\max\{s,s^{-1}\}$, for $s>0$.
\end{definition}

Recently, the second named author of this paper, together with D. Cruz-Uribe and C. Rios, has obtained in \cite{CruzMartellRios} 
some new results about these types of estimates for the heat semigroup associated with $L_w$. Here we just state some properties that will be needed later.

\begin{lemma}[{\cite[Lemma 7.5]{CruzMartellRios}}]\label{off-diag}
Given $w\in A_{\infty}$ and a family of sublinear operators $\{T_t\}_{t>0}$ such that $T_t \in \mathcal O(L^p(w)-L^q(w))$, with $1\leq p<q\leq \infty$, there exist $\alpha, \beta>0$ such that for any $t>0$ and for any ball $B$ with radius $r_B$, 
\begin{equation}\label{off-AlphaBeta}
\br{\dashint_B \abs{T_t(f \chi_B)}^q\, dw}^{1/q} \lesssim \max\left\{\br{\frac{r_B}{\sqrt t}}^{\alpha}, \br{\frac{r_B}{\sqrt t}}^{\beta}\right\} \br{\dashint_B |f|^p\,dw}^{1/p}.
\end{equation}
\end{lemma}

\begin{lemma}[{\cite[Proposition 3.1, Corollary 3.4, Proposition 7.1, and Section 8]{CruzMartellRios}}]\label{off-diag-sg}
Let $L_w$ be a degenerate elliptic operator with $w\in A_2$. 
\begin{list}{$(\theenumi)$}{\usecounter{enumi}\leftmargin=1cm \labelwidth=1cm\itemsep=0.2cm\topsep=.2cm \renewcommand{\theenumi}{\alph{enumi}}}

\item If $p_-(L_w)<p\leq q<p_+(L_w)$ \textup{(}cf. \eqref{p-}, \eqref{p+}\textup{)}, then $e^{-t L_w}$ and $(tL_w)^m e^{-t L_w}$, for every $m\in\N$, belong to $\mathcal O(L^p(w)-L^q(w))$. 

\item There exists an interval $\mathcal K(L_w)$ such that if $p,q \in \mathcal K(L_w)$, $p\leq q$, then $\sqrt t\nabla e^{-t L_w}\in \mathcal O(L^p(w)-L^q(w))$. Moreover, denoting by $q_-(L_w)$ and $q_+(L_w)$ the left and right endpoints of $\mathcal K(L_w)$, then $q_-(L_w)=p_-(L_w)$, $2<q_+(L_w)\leq p_+(L_w)$.

\end{list}

\end{lemma}

\subsection{Change of angles for weighted conical square functions}\label{subsectionweightv}

In this section we present a result that will allow us to change the aperture of the cone in different square functions. We indeed work in the setting of tent spaces and we put the emphasis on quantifying the bound that is obtained by the change of aperture. These change of angle formulas were first established for the Lebesgue measure in \cite{CoifmanMeyerStein} and with an optimal version in \cite{Auscherangles}. The weighted case was considered in \cite{MartellPrisuelos} (see also \cite{Lerner}). Here, as opposed to what was done in \cite{MartellPrisuelos}, the underlying measure is $dw$, as can be seen from  the conical square functions \eqref{square-H-1}--\eqref{square-P-3}.

To set the stage, we denote by $\R_+^{n+1}$ the upper-half space, that is, the set of points $(y,t)\in \R^n\times \R_+$. Given $\alpha>0$ and $x\in \mathbb{R}^n$ we define the cone of aperture  $\alpha$ with vertex at $x$ by
$$
\Gamma^{\alpha}(x):=\{(y,t)\in \R_+^{n+1} : |x-y|<\alpha t\}.
$$
For any closed set $E$ in $\mathbb{R}^n$, let $
\mathcal{R}^{\alpha}(E):=\bigcup_{x\in E}\Gamma^{\alpha}(x).$
We also define the operator $\mathcal{A}_w^{\alpha}$, $\alpha>0$, $w\in A_\infty$:
\begin{align}\label{AF}
\mathcal{A}_w^{\alpha}F(x):=\left(\iint_{\Gamma^{\alpha}(x)}|F(y,t)|^2 \ \frac{dw(y) \, dt}{tw(B(y,t))}\right)^{\frac{1}{2}}.
\end{align}
When $\alpha=1$ we simplify the above notation by writing $\Gamma(x)$, $\mathcal{R}(E)$, and $\mathcal{A}_w$.

In the following proposition we present the promised change of angle formulas which allow us to compare the $L^p(vdw)$-norms of the operators $\mathcal{A}_w^{\alpha}$ for different values of $\alpha$.
\begin{proposition}[Change of angles]\label{prop:alpha}
Let $0< \alpha\leq \beta<\infty$.
\begin{list}{$(\theenumi)$}{\usecounter{enumi}\leftmargin=1cm
\labelwidth=1cm\itemsep=0.2cm\topsep=.2cm
\renewcommand{\theenumi}{\roman{enumi}}}

\item  For every $w\in A_{\widetilde{r}}$ and  $v\in A_r(w)$, $1\leq r,\widetilde{r}<\infty$,
there holds
\begin{align}\label{change-alph-1}
\norm{\mathcal{A}_w^{\beta}F}_{L^p(vdw)}\
\leq 
C \br{\frac{\beta}{\alpha}}^{\frac{n\,\widetilde{r}\,r}{p}}
 \norm{\mathcal{A}^{\alpha}_w F}_{L^p(vdw)} \quad \textrm{for all} \quad 0<p\leq 2r,
\end{align}
 where   $C\ge 1$ depends on $n$, $p$, $r$, $\widetilde{r}$, $[w]_{A_{\widetilde{r}}}$, and  $[v]_{A_r(w)}$, but it is independent of  $\alpha$ and $\beta$.

\item  For every $w\in RH_{\widetilde{s}'}$ and $v\in RH_{s'}(w)$, $1\leq s,\widetilde{s}<\infty$, there holds
\begin{align}\label{change-alph-2}
\norm{\mathcal{A}^{\alpha}_w F}_{L^p(vdw)}
\leq  
C\br{\frac{\alpha}{\beta}}^{\frac{n}{s\,\widetilde{s}\,p}} \norm{\mathcal{A}^{\beta}_w F}_{L^p(vdw)}
\quad \text{for all} \quad \frac{2}{s}\leq p<\infty.
\end{align}
 where   $C\ge 1$ depends on $n$, $p$, $s$, $\widetilde{s}$, $[w]_{RH_{\widetilde{s}'}}$, and  $[v]_{RH_{s'}(w)}$, but it is independent of  $\alpha$ and $\beta$.

\end{list}
\end{proposition}

\begin{proof}
We start proving part ($i$). Fix $w\in A_{\widetilde{r}}$, $1\leq \widetilde{r}<\infty$. We first consider the case 
$p=2$ and $1\le r<\infty$, then we shall extrapolate to obtain \eqref{change-alph-1} for $1< r<\infty$ and $0<p\leq 2r$. Finally we prove the case $r=1$ and $0<p<2$. In all these cases we may assume that
$\|\mathcal{A}_w^{\alpha}F\|_{L^p(vdw)}<\infty$. Otherwise, there is nothing to prove.

For $p=2$ and $v\in A_{r_0}(w)$, $1\leq r_0<\infty$, applying \eqref{pesosineq:Ap} and  Fubini's theorem, we obtain
\begin{align}\label{Change-p=2:Ar}
\|\mathcal{A}^{\beta}_wF\|_{L^2(vdw)}&=\left(\int_{\mathbb{R}^n}\int_0^{\infty}\int_{|x-y|<\beta t}|F(y,t)|^2 \frac{dw(y) \, dt}{tw(B(y,t))}v(x)dw(x) \right)^{\frac{1}{2}}
\\
\nonumber & =
\left(\int_{\mathbb{R}^n}\int_0^{\infty}|F(y,t)|^2 vw(B(y,\beta t)) \frac{dw(y) \, dt}{tw(B(y,t))}\right)^{\frac{1}{2}}\nonumber
\\
\nonumber & \le C \left(\frac{\beta}{\alpha}\right)^{\frac{n\,r_0\,\widetilde{r}}{2}}
\left(\int_{\mathbb{R}^n}\int_0^{\infty}|F(y,t)|^2 vw(B(y,\alpha t)) \frac{dw(y) \, dt}{tw(B(y,t))}\right)^{\frac{1}{2}}
\\
\nonumber &=
C\left(\frac{\beta}{\alpha}\right)^{\frac{n\,r_0\,\widetilde{r}}{2}}\left(\int_{\mathbb{R}^n}\int_0^{\infty}\int_{|x-y|< \alpha t}|F(y,t)|^2 \frac{dw(y) \, dt}{tw(B(y,t))}v(x) dw(x) \right)^{\frac{1}{2}}
\\
\nonumber &=C\left(\frac{\beta}{\alpha}\right)^{\frac{n\,r_0\,\widetilde{r}}{2}} \norm{\mathcal{A}_w^{\alpha}F}_{L^2(vdw)},
\end{align}
 where $C$ is independent of $\alpha$ and $\beta$. 

Next we extrapolate from this inequality to the case $1<r<\infty$ and $0<p \le 2r$.
Take an arbitrary $1\leq r_0<\infty$. Then, \eqref{Change-p=2:Ar} implies, for all $v\in A_{r_0}(w)$,
$$
\int_{\mathbb{R}^n}\br{\br{\mathcal{A}_w^{\beta}F(x)}^{\frac{2}{r_0}}}^{r_0}v(x)dw(x) 
\lesssim 
\int_{\mathbb{R}^n}\br{\br{\frac{\beta}{\alpha}}^{n\,\widetilde{r}}\br{\mathcal{A}_w^\alpha F(x)}^{\frac{2}{r_0}}}^{r_0}v(x)dw(x).
 $$
Now, using  Theorem \ref{theor:extrapol}, part ($a$), we obtain that, for all $1<r<\infty$ and $v\in A_r(w)$,
$$
\int_{\mathbb{R}^n}\br{\mathcal{A}_w^{\beta}F(x)}^{\frac{2r}{r_0}}v(x)dw(x) 
\lesssim \left(\frac{\beta}{\alpha}\right)^{nr\,\widetilde{r}}
\int_{\mathbb{R}^n} \br{\mathcal{A}_w^{\alpha}F(x)}^{\frac{2r}{r_0}}v(x)dw(x).
$$
Since $1\leq r_0<\infty$ is arbitrary, we conclude \eqref{change-alph-1} for all $1<r<\infty$, $v\in A_{r}(w)$, and $0<p\leq 2r$.  Note that the implicit constant is independent of $\alpha$ and $\beta$.

\medskip
Finally we show the case $v\in A_{1}(w)$ and $0<p<2$. As in the proof of \cite[Section 3, Proposition 4]{CoifmanMeyerStein} and \cite[Proposition 3.2]{MartellPrisuelos}, we consider, for all $\lambda>0$, and for $0<\gamma<1$ to be chosen later,
$$
O_{\lambda}:=\{x\in \R^{n}:\mathcal{A}_w^{\alpha}F(x)>\lambda\},
\quad
E_{\lambda}:=\R^n\setminus O_{\lambda},
\ \
\textrm{and}
\ \ 
E_{\lambda}^*:=\bbr{x\in \R^n:\frac{|B(x,r)\cap E_{\lambda}|}{|B(x,r)|}\geq \gamma, \,\forall r>0}.
$$
Note that 
$$
O_{\lambda}^*:=\R^n\setminus E_{\lambda}^*=\{x\in \R^{n}:\mathcal{M}(\chi_{O_{\lambda}})(x)>1-\gamma\},
$$
where $\mathcal{M}$ is the Hardy-Littlewood maximal operator.  Clearly $O_{\lambda}^*$ is open and so is $O_{\lambda}$ (see for instance \cite[Proposition 3.2]{MartellPrisuelos}). Besides, $O_\lambda\subset O_{\lambda}^*$, and
\begin{align}\label{maximal-v}
vw(O_{\lambda}^*)\leq C\frac{1}{(1-\gamma)^{\widetilde{r}}}vw(O_{\lambda})<\infty.
\end{align}
Here the first inequality follows from the fact that $\mathcal{M}:L^{\widetilde{r}}(vdw)\rightarrow L^{\widetilde{r},\infty}(vdw)$, since 
$w\in A_{\widetilde{r}}$ and $v\in A_{1}(w)$ easily imply that $v w\in A_{\widetilde{r}}$. The last inequality is due to the assumption $\|\mathcal{A}_w^{\alpha}F\|_{L^p(vdw)}<\infty$.

Note now that, for all $(y,t)\in \mathcal{R}^{\beta}(E_{\lambda}^*)$, there exists $\bar{x}\in E_{\lambda}^*$ such that
$|\bar{x}-y|<\beta t$. We claim that
\begin{equation}\label{bound}
|B(y,\beta t)|(\gamma-c_{\beta,\alpha})=|B(\bar{x},\beta t)|(\gamma-c_{\beta,\alpha})\leq |E_{\lambda}\cap B(y,\alpha t)|, 
\end{equation}
where $c_{\beta,\alpha}:=1-\frac{\alpha^n}{2^n\beta^n}$. 
Indeed, for $z:=y-\frac{\alpha t}{2}\frac{y-\bar{x}}{|y-\bar{x}|}$, we have that $B(z,\alpha t/2)\subset B(\bar{x}, \beta t)\cap B(y,\alpha t)$. Then, since $\bar{x}\in E_{\lambda}^*$,
\begin{multline*}
\gamma|B(\bar{x},\beta t)|\leq |E_{\lambda}\cap B(\bar{x},\beta t)|
\leq |E_{\lambda}\cap B(y,\alpha t)|+|B(\bar{x},\beta t)\setminus  B(y,\alpha t)|
\\
\leq 
|E_{\lambda}\cap B(y,\alpha t)|+|B(\bar{x},\beta t)\setminus  B(z,\alpha t/2)|\leq |E_{\lambda}\cap B(y,\alpha t)|+|B(\bar{x},\beta t)|\left(1-\frac{\alpha^n}{2^n\beta^n}\right),
\end{multline*}
and this proves \eqref{bound}.

Next, recalling that $v w\in A_{\widetilde{r}}$,
\eqref{pesosineq:Ap} and \eqref{bound} then imply
$$
\frac{vw(E_{\lambda}\cap B(y, \alpha t))}{vw(B(y,\beta t))}
\geq C \left(\frac{|E_{\lambda}\cap B(y,\alpha t)|}{|B(y,\beta t)|}\right)^{\widetilde{r}}
\geq C(\gamma-c_{\beta,\alpha})^{\widetilde{r}}=C  \br{\frac{\alpha}{\beta}}^{n\,\widetilde{r}},
$$
after choosing $\gamma:=1-\frac{\alpha^n}{2^{n+1}\beta^n}$.
Hence,
\begin{align}\label{ineq}
\int_{E_{\lambda}^*} \mathcal{A}_w^{\beta}F(x)^2 v(x)dw(x) 
&= \int_{E_{\lambda}^*}\int_{0}^{\infty}\int_{\mathbb{R}^n} |F(y,t)|^2 \chi_{B(0,1)}\left(\frac{x-y}{\beta t}\right) \frac{dw(y) \, dt}{tw(B(y,t))} v(x)dw(x)
\\
\nonumber &\leq \iint_{\mathcal{R}^{\beta}(E_{\lambda}^*)} |F(y,t)|^2vw(B(y,\beta t))\frac{dw(y) \, dt}{tw(B(y,t))}
\\
\nonumber &\leq C \left(\frac{\beta}{\alpha}\right)^{n\,\widetilde{r}} \iint_{\mathcal{R}^{\beta}(E_{\lambda}^*)} |F(y,t)|^2vw(E_{\lambda}\cap B(y,\alpha t))\frac{dw(y) \, dt}{tw(B(y,t))}
\\
\nonumber &
= C\left(\frac{\beta}{\alpha}\right)^{n\,\widetilde{r}} \iint_{\mathcal{R}^{\beta}(E_{\lambda}^*)} |F(y,t)|^2 \int_{B(y,\alpha t)\cap E_{\lambda}} v(x)dw(x)\, \frac{dw(y) \, dt}{tw(B(y,t))}
\\
\nonumber &
\leq C \left(\frac{\beta}{\alpha}\right)^{n\,\widetilde{r}} \int_{E_{\lambda}} \mathcal{A}_w^{\alpha}F(x)^2 v(x)dw(x).
\end{align}
 Therefore, from our choice of $\gamma$, by \eqref{maximal-v} and \eqref{ineq}, and applying Chebychev's inequality, we have that
\begin{align*}
vw\br{\bbr{x\in \R^n: \mathcal{A}_w^{\beta}F(x)> \lambda}} &\leq vw\br{\bbr{x\in O_{\lambda}^*:\mathcal{A}_w^{\beta}F(x)> \lambda}}+vw\br{\bbr{x\in E_{\lambda}^*: \mathcal{A}_w^{\beta}F(x)> \lambda}}
\\
&
\leq vw\br{O_{\lambda}^*} + \frac{1}{\lambda^2}\int_{E_{\lambda}^*}\mathcal{A}_w^{\beta}F(x)^{2} v(x) dw(x) 
\\
&
\lesssim \br{\frac{\beta}{\alpha}}^{n\,\widetilde{r}}  vw(O_{\lambda}) + \br{\frac{\beta}{\alpha}}^{n\,\widetilde{r}}  \frac{1}{\lambda^2}\int_{E_{\lambda}}\mathcal{A}_w^{\alpha}F(x)^2 v(x) dw(x)
.
\end{align*}

Using the above estimate, it follows that for $0<p<2$,
\begin{align*}
&\|\mathcal{A}_w^{\beta}F\|_{L^p(vdw)}^p
=
\int_{0}^{\infty}p\,\lambda^{p} \,vw\br{\bbr{x\in \R^n: \mathcal{A}_w^{\beta}F(x)> \lambda}}\, \frac{d\lambda}{\lambda}
\\
&
\quad\lesssim
\br{\frac{\beta}{\alpha}}^{n\,\widetilde{r}} \br{\int_{0}^{\infty}p\,\lambda^{p} \, vw\br{O_{\lambda}}\, \frac{d\lambda}{\lambda}
+
\int_{0}^{\infty}p\lambda^{p-2}\int_{E_{\lambda}} \mathcal{A}_w^{\alpha}F(x)^2v(x) dw(x)\, \frac{d\lambda}{\lambda}}
\\
&
\quad\lesssim
\br{\frac{\beta}{\alpha}}^{n\,\widetilde{r}} 
\br{
\norm{\mathcal{A}_w^{\alpha}F}_{L^p(vdw)}^p
+
\int_{\mathbb{R}^n} \mathcal{A}_w^{\alpha}F(x)^2\int_{\mathcal{A}_w^{\alpha}F(x)}^{\infty}p\lambda^{p-2} \, \frac{d\lambda}{\lambda} v(x) dw(x)}\\
&
\quad=
C\,\br{\frac{\beta}{\alpha}}^{n\,\widetilde{r}} \norm{\mathcal{A}_w^{\alpha}F}_{L^p(vdw)}^p,
\end{align*}
 where $C$ is independent of $\alpha$ and $\beta$.  This completes 
 the proof of $(i)$. 

\medskip
We next prove part ($ii$). Fix $w\in RH_{\tilde{s}'}$, $1\leq \tilde{s}<\infty$. As in the proof of part $(i)$, we split the proof into three steps.  We first prove \eqref{change-alph-2} for $p=2$ and $1\leq s<\infty$, then by extrapolation we will show it for
$2/s\leq p<\infty$ and $1<s<\infty$, and finally we will deal with the case $s=1$ and $2<p<\infty$.

We start by taking $p=2$ and $v\in RH_{s_0'}(w)$ with $1\leq s_0<\infty$. Proceeding as in \eqref{Change-p=2:Ar} but using \eqref{pesosineq:RHq}  instead of \eqref{pesosineq:Ap}, we obtain
\begin{multline}\label{Change-p=2:RHs}
\|\mathcal{A}_w^{\alpha}F\|_{L^2(vdw)}
=
\left(\int_0^{\infty}\int_{\mathbb{R}^n}|F(y,t)|^2 vw(B(y,\alpha t))  \frac{dw(y) \, dt}{tw(B(y,t))}\right)^{\frac{1}{2}}
\\
\lesssim 
\br{\frac{\alpha}{\beta}}^{\frac{n}{2s_0\widetilde{s}}}
\left(\int_0^{\infty}\int_{\mathbb{R}^n}|F(y,t)|^2 vw(B(y,\beta t)) \frac{dw(y) \, dt}{tw(B(y,t))}\right)^{\frac{1}{2}}
=
\br{\frac{\alpha}{\beta}}^{\frac{n}{2s_0\,\widetilde{s}}} \norm{\mathcal{A}_w^{\beta} F}_{L^2(vdw)},
\end{multline}
 where the implicit constant is independent of $\alpha$ and $\beta$.  
Let us extrapolate from this inequality. Take an arbitrary $1\le s_0<\infty$ and notice that \eqref{Change-p=2:RHs} immediately yields that, for every $v\in RH_{s_0'}(w)$,
$$
\int_{\R^n} (\mathcal{A}_w^{\alpha}F(x)^{2\,s_0})^{\frac{1}{s_0}}\,v(x)dw(x)
\lesssim
\int_{\R^n} \left(\left(\frac{\alpha}{\beta}\right)^{\frac{n}{\widetilde{s}}}\mathcal{A}_w^\beta F(x)^{2\,s_0}\right)^{\frac{1}{s_0}}\,v(x)dw(x).
$$
Next, we apply Theorem \ref{theor:extrapol}, part ($b$), to conclude that, for every $1<s<\infty$ and for every $v\in RH_{s'}(w)$,
$$ 
\int_{\R^n} \mathcal{A}_w^{\alpha}F(x)^{\frac{2\,s_0}{s}}\,v(x)dw(x)
\le
C\,\left(\frac{\alpha}{\beta}\right)^{\frac{n}{s\,\widetilde{s}}}\,\int_{\R^n} \mathcal{A}_w^{\beta}  F(x)^{\frac{2\,s_0}{s}}\,v(x)dw(x),
$$
where $C$ does not depend on $\alpha$ or $\beta$. From this, using that $1\le s_0<\infty$ is arbitrary we conclude \eqref{change-alph-2} for all $1<s<\infty$ and  $2/s\leq p<\infty$.

Finally, we show \eqref{change-alph-2} for all $2<p<\infty$ and $v\in RH_{\infty}(w)$ (i.e., $s=1$).
Without loss of generality, we may assume that $\frac{\beta}{\alpha}>32$ (for $1\le \frac{\beta}{\alpha}\le 32$ we just use the fact that $\mathcal{A}_w^{\alpha}F\le \mathcal{A}_w^\beta F$). Let us also assume that $\|\mathcal{A}_w^\beta F\|_{L^p(vdw)}<\infty$ (otherwise there is nothing to prove). Besides, since $v\in RH_{\infty}(w)$ there exists $r>1$  such that $r\ge p/2$ and $v\in A_{r}(w)$. Then we can apply part $(i)$  and obtain that
\begin{equation}\label{A-A}
\|\mathcal{A}_w^{6\sqrt{n}\beta}F\|_{L^p(vdw)}
\le
C
\left(\frac{6\sqrt{n}\beta}{\beta}\right)^{\frac{n\,r\,\widetilde{r}}{p}}
\|\mathcal{A}_w^{\beta}F\|_{L^p(vdw)}
=
C\|\mathcal{A}_w^{\beta}F\|_{L^p(vdw)}<\infty,
\end{equation}
where $C$ does not depend on $\beta$.

After these observations, for every $\lambda>0$, consider the set
$$
O_{\lambda}:=\bbr{x\in \mathbb{R}^n:\mathcal{A}_w^{6\sqrt{n}\beta}F(x)>\lambda}.
$$
We claim that
\begin{align}\label{AF-Awe}
vw\br{\bbr{x\in \mathbb{R}^n : \mathcal{A}_w^{\alpha}F(x)>2\lambda}}
\lesssim
\frac{1}{\lambda^2}\br{\frac{\alpha}{\beta}}^{\frac n{\widetilde{s}}} \int_{O_{\lambda}}\abs{\mathcal{A}_w^{6\sqrt{n}\beta}F(x)}^2v(x)dw(x).
\end{align}

Assuming \eqref{AF-Awe} momentarily and applying \eqref{A-A}, we obtain \eqref{change-alph-2} for $2<p<\infty$. Indeed,
\begin{align*}
\|\mathcal{A}_w^{\alpha}F\|_{L^p(vdw)}^p
&=
2^p\int_{0}^{\infty}p\,\lambda^{p} \, vw\left(\left\{x\in \R^n: \mathcal{A}_w^{\alpha}F(x)> 2\lambda\right\}\right)\, \frac{d\lambda}{\lambda}
\\
&\lesssim
\left(\frac{\alpha}{\beta}\right)^{\frac{n}{\widetilde{s}}}\int_{0}^{\infty}\lambda^{p-2}
\int_{O_{\lambda}}\mathcal{A}_w^{6\sqrt{n}\beta}F(x)^2v(x)dw(x)\frac{d\lambda}{\lambda}
\\
&\lesssim \left(\frac{\alpha}{\beta}\right)^{\frac{n}{\widetilde{s}}}\int_{\mathbb{R}^n}\mathcal{A}_w^{6\sqrt{n}\beta}F(x)^2
\int_{0}^{\mathcal{A}_w^{6\sqrt{n}\beta}F(x)}\lambda^{p-2}
\frac{d\lambda}{\lambda}v(x)dw(x)
\\
&\lesssim \left(\frac{\alpha}{\beta}\right)^{\frac{n}{\widetilde{s}}} \norm{\mathcal{A}_w^{6\sqrt{n}\beta}F}_{L^p(vdw)}^p
\\
&
\lesssim \left(\frac{\alpha}{\beta}\right)^{\frac{n}{\widetilde{s}}} \norm{\mathcal{A}_w^{\beta}F}_{L^p(vdw)}^p,
\end{align*}
 where the implicit constants are independent of $\alpha$ and $\beta$.

It remains to show \eqref{AF-Awe}. We may  assume that  $O_\lambda\neq\emptyset$, otherwise both sides in \eqref{AF-Awe} would vanish since $\mathcal{A}_w^{\alpha}F\leq \mathcal{A}_w^{6\sqrt{n}\beta}F$. Using similar arguments as in the proof of \cite[Proposition 3.2, part ($i$)]{MartellPrisuelos}, we clearly have that $O_\lambda$ is open. Also, \eqref{A-A} and Chebychev's inequality give us $vw(O_\lambda)<\infty$, which in turn yields that $O_\lambda\subsetneq \R^n$ (since $vw$ is a doubling measure and hence $vw(\R^n)=\infty$).  Then, we can take a Whitney decomposition of $O_{\lambda}$ (see for example \cite[Chapter VI]{St70}): there exists a family of  closed cubes  $\{Q_j\}_{j\in \N}$ with disjoint interiors so that
\begin{equation}\label{Whitney}
O_{\lambda}
=
\bigcup_{j\in \N}Q_j,  \qquad \textrm{diam}(Q_j)\leq d(Q_j,\R^n\setminus O_{\lambda})\leq 4 \textrm{diam}(Q_j),
\quad\textrm{and}\quad
\sum_{j\in\N} \chi_{Q_j^*}\le 12^n\,\chi_{O_\lambda},
\end{equation}
where $Q_j^*:=\frac98Q_j$  and $d(Q_j,\R^n\setminus O_{\lambda})$ denotes the Euclidean distance between the sets $Q_j$ and $\R^n\setminus O_{\lambda}$.

On the other hand, since $\mathcal{A}_w^{\alpha}F\leq \mathcal{A}_w^{6\sqrt{n}\beta}F$, we have that
\begin{align}\label{rhwhi}
vw\br{\bbr{x\in \mathbb{R}^n : \mathcal{A}_w^{\alpha}F(x)>2\lambda}}
&= vw\br{\bbr{x\in O_\lambda: \mathcal{A}_w^{\alpha}F(x)>2\lambda}}
\\ \nonumber &=\sum_{j\in \N}
vw\br{\bbr{x\in Q_j : \mathcal{A}_w^{\alpha}F(x)>2\lambda}}.
\end{align}
Fix $j\in\N$ and, for every $x\in Q_j$, write
\begin{align*}
\mathcal{A}_w^{\alpha}F(x)
&\le
\left(\int_{\frac{\ell(Q_j)}{\beta}}^{\infty}\int_{B(x,\alpha t)}|F(y,t)|^2\frac{dw(y) \, dt}{tw(B(y,t))}\right)^{\frac{1}{2}}
+
\left(\int_{0}^{\frac{\ell(Q_j)}{\beta}}\int_{B(x,\alpha t)}|F(y,t)|^2\frac{dw(y) \, dt}{tw(B(y,t))}\right)^{\frac{1}{2}}
\\ &=: G_j(x)+H_j(x).
\end{align*}
Pick $x_j\in \mathbb{R}^n\setminus O_{\lambda}$ such that $d(x_j,Q_j)\leq 4 \textrm{diam}(Q_j)$. Notice that for every $x\in Q_{j}$ and $t\ge \ell(Q_j)/\beta$ we have that
$B(x,\alpha t)\subset B(x_j,6\sqrt{n}\beta t)$.
Then,
\begin{align*}
G_j(x)^2
&=\int_{\frac{\ell(Q_j)}{\beta}}^{\infty}\int_{B(x,\alpha t)}|F(y,t)|^2\frac{dw(y) \, dt}{tw(B(y,t))}
\leq
\int_{\frac{\ell(Q_j)}{\beta}}^{\infty}\int_{B(x_j,6\sqrt{n}\beta t)}|F(y,t)|^2\frac{dw(y) \, dt}{tw(B(y,t))}
\\ &\leq \mathcal{A}_w^{6\sqrt{n}\beta}F(x_j)^2\leq \lambda^2,
\end{align*}
where the last inequality holds since $x_j\in \mathbb{R}^n\setminus O_{\lambda}$. Using this, Chebychev's inequality, and \eqref{pesosineq:RHq} for $w \in RH_{\widetilde{s}}$ and $v\in RH_{\infty}(w)$, we have
\begin{align*}
&vw\br{\bbr{x\in Q_j : \mathcal{A}_w^{\alpha}F(x)>2\lambda}}
\le
vw\br{\bbr{x\in Q_j : H_j(x)>\lambda}}
\\
&\qquad\qquad
\leq
\frac{1}{\lambda^2}\int_{Q_j} H_j(x)^2v(x)dw(x)
\\
&\qquad\qquad\leq
\frac{1}{\lambda^2} \iint_{\mathcal{R}^{\alpha}(Q_j)}\chi_{(0,\,\beta^{-1}\ell(Q_j))}(t)|F(y,t)|^2 vw(B(y,\alpha t))\frac{dw(y) \, dt}{tw(B(y,t))}
\\
&\qquad\qquad\lesssim
\frac{(\alpha/\beta)^{\frac{n}{\widetilde{s}}}}{\lambda^2}\iint_{\mathcal{R}^{\alpha}(Q_j)}\chi_{(0,\beta^{-1}\ell(Q_j))}(t)|F(y,t)|^2 vw(B(y,32^{-1}\beta t))\frac{dw(y) \, dt}{tw(B(y,t))}
\\
&\qquad\qquad\leq
\frac{(\alpha/\beta)^{\frac{n}{\widetilde{s}}}}{\lambda^2}\int_{Q_j^*}\int_0^{\infty}\int_{B(x,32^{-1}\beta t)}|F(y,t)|^2 \frac{dw(y) \, dt}{tw(B(y,t))}v(x)dw(x)
\\
&\qquad\qquad\leq
\frac{(\alpha/\beta)^{\frac{n}{\widetilde{s}}}}{\lambda^2}\int_{Q_j^*}\mathcal{A}_w^{6\sqrt{n}\beta}F(x)^2v(x)dw(x).
\end{align*}
Then, by  \eqref{rhwhi} and the bounded overlap of the family $\{Q_j^*\}_{j\in\N}$, 
\begin{align*}
vw\br{\bbr{x\in \mathbb{R}^n : \mathcal{A}_w^{\alpha}F(x)>2\lambda}}
&\lesssim
\frac{1}{\lambda^2} \br{\frac{\alpha}{\beta}}^{\frac{n}{\widetilde{s}}}\sum_{j\in \N}\int_{Q_j^*} \abs{\mathcal{A}_w^{6\sqrt{n}\beta}F(x)}^2v(x)dw(x)
\\
&\lesssim
\frac{1}{\lambda^2} \br{\frac{\alpha}{\beta}}^{\frac{n}{\widetilde{s}}} \int_{O_{\lambda}} \abs{\mathcal{A}_w^{6\sqrt{n}\beta}F(x)}^2 v(x)dw(x),
\end{align*}
 where the implicit constants are independent of $\alpha$ and $\beta$.   This completes the proof of \eqref{AF-Awe}.
\end{proof}

\bigskip

\subsection{Carleson measure condition}
Given $0<p<\infty$, we now introduce a new maximal operator (see \cite{MartellPrisuelos} for the case $w\equiv 1$)
\begin{align}\label{CF}
\mathcal{C}_{w,p}F(x_0)=\sup_{B\ni x_0}\left(\frac{1}{w(B)}\int_{B}\left(\int_{0}^{r_B}\int_{B(x,t)}|F(y,t)|^2 \frac{dw(y) \, dt}{tw(B(y,t))}\right)^{\frac{p}{2}} \, dw(x) \right)^{\frac{1}{p}},
\end{align}
where the supremum is taken over all balls $B\subset \mathbb{R}^n$ containing $x_0$, and $r_B$ denotes the radius of $B$. 
We also consider  
$$
\mathcal{C}_w F(x_0)=\sup_{B\ni x_0}\left(\frac{1}{w(B)}\int_0^{r_B}\int_{B}|F(y,t)|^2 \frac{dw(y) \, dt}{t} \right)^{\frac{1}{2}},
$$
which is a weighted version of the one introduced in \cite{CoifmanMeyerStein} to study duality in tent spaces.

We first observe that for $p=2$, 
\begin{align}\label{Cw-Cw,2:equiv}
\mathcal{C}_w F \approx \mathcal{C}_{w,2}F.
\end{align}
Indeed, by \eqref{pesosineqw:Ap} and  Fubini's theorem, 
\begin{align*}
\mathcal{C}_{w,2}F(x_0)
&=
\sup_{B\ni x_0}\left(\frac{1}{w(B)}\int_{B}\int_{0}^{r_B}\int_{B(x,t)}|F(y,t)|^2 \frac{dw(y) \, dt}{tw(B(y,t))} dw(x)\right)^{\frac{1}{2}}
\\
&\leq \sup_{B\ni x_0}\left(\frac{1}{w(B)}\int_{2B}\int_{0}^{r_B}|F(y,t)|^2\int_{B(y,t)}  dw(x) \frac{dw(y) \, dt}{tw(B(y,t))} \right)^{\frac{1}{2}}
\\
&\lesssim \sup_{B\ni x_0}\left(\frac{1}{w(2B)}\int_0^{2r_B}\int_{2B}|F(y,t)|^2 \frac{dw(y) \, dt}{t} \right)^{\frac{1}{2}}=\mathcal{C}_w F(x_0).
\end{align*}
As for the reverse inequality, there holds
\begin{align*}
\mathcal{C}_w F(x_0)
&=\sup_{B\ni x_0}\left(\frac{1}{w(B)}\int_0^{r_B}\int_{B}|F(y,t)|^2 \frac{dw(y) \, dt}{t} \right)^{\frac{1}{2}}
\\ &= 
\sup_{B\ni x_0}\left(\frac{1}{w(B)}\int_{0}^{r_B}\int_{B}|F(y,t)|^2\int_{B(y,t)} dw(x) \frac{dw(y) \, dt}{tw(B(y,t))} \right)^{\frac{1}{2}}
\\ &\lesssim 
\sup_{B\ni x_0}\left(\frac{1}{w(2B)}\int_{2B}\int_{0}^{2r_B}\int_{B(x,t)}|F(y,t)|^2\frac{dw(y) \, dt}{tw(B(y,t))} dw(x) \right)^{\frac{1}{2}}
=
\mathcal{C}_{w,2}F(x_0). 
\end{align*}

\medskip

The following proposition relates the norm of $\mathcal{C}_{w,p_0}f$ with that of $\mathcal{A}_w f$. This  will be crucial in the proof of Theorem \ref{thm:SF-Heat}.  When $w\equiv 1$ this was proved in \cite[Proposition 3.34]{MartellPrisuelos} for a general $p_0$ (see also 
\cite[Theorem 3]{CoifmanMeyerStein} for the case $p_0=2$ and $w,v\equiv 1$). 
\begin{proposition}\label{prop:maximal}
Let $w\in A_{\infty}$.
\begin{list}{$(\theenumi)$}{\usecounter{enumi}\leftmargin=1cm \labelwidth=1cm\itemsep=0.2cm\topsep=.2cm \renewcommand{\theenumi}{\alph{enumi}}}
\item If\, $0<p_0,p<\infty$, $v \in A_{\infty}(w)$, and $F\in L^2_{\rm loc}(\R_{+}^{n+1},dwdt)$ then
\begin{align*}
\|\mathcal{A}_wF\|_{L^p(vdw)}\lesssim \|\mathcal{C}_{w,p_0}F\|_{L^p(vdw)}.
\end{align*}

\item If\, $0< p_0< p<\infty$ and $v\in A_{\frac{p}{p_0}}(w)$ then
\begin{align*}
\|\mathcal{C}_{w,p_0}F\|_{L^p(vdw)}\lesssim \|\mathcal{A}_wF\|_{L^p(vdw)}.
\end{align*}
\end{list}
\end{proposition}

\begin{proof}

We start proving part ($a$). Fix $w\in A_\infty$ and let $1\leq r<\infty$ be so that $w\in A_r$.   The proof is divided into two steps. 

{\em Step $1$}: We first consider a function $F\in L^2(\R_{+}^{n+1},dwdt)$ such that, for some $N>0$, 
$\supp F\subset K_N:=\{(y,t)\in \R^{n+1}_+:y\in B(0,N),N^{-1}<t<N\}$.
Notice that for $y\in B(0,N)$ and $t\ge N^{-1}$, 
$$
w(B(0,N)) \le w(B(y,2N)) \le w(B(y,2N^2t)) \le [w]_{A_r} (2N^2)^{nr} w(B(y,t)).
$$
Then,
\begin{align}\label{normfinite}
\|\mathcal{A}_{w}F\|_{L^p(vdw)}&=
\br{\int_{B(0,2N)}\br{\int_{N^{-1}}^{N}\int_{B(x,t)\cap B(0,N)}|F(y,t)|^2\frac{dw(y)\,dt}{tw(B(y,t))}}^{\frac{p}{2}}v(x)dw(x)}^{\frac{1}{p}}
\\ \nonumber
&\lesssim N^{\frac{1}{2}}
\br{ [w]_{A_r} (2N^2)^{nr}}^{1/2} w(B(0,N))^{-\frac12} vw(B(0,2N))^{\frac1p}\norm{F}_{L^2(\R_{+}^{n+1},\,dwdt)}
<\infty.
\end{align}

We claim that it is enough to prove that there exist $\alpha>1$ and $c,\,c_{w,v}>0$ such that for all $0<\gamma< 1$ and $0<\lambda<\infty$, we have
\begin{align}\label{good-lambda:A-Cp:w}
vw\br{\bbr{x\in \mathbb{R}^n: \mathcal{A}_wF(x)>2\lambda,\mathcal{C}_{w,p_0}F(x)\leq \gamma\lambda}}
\leq 
c\gamma^{c_{w,v}}vw\br{\bbr{x\in \mathbb{R}^n: \mathcal{A}_w^{\alpha}F(x)>\lambda}}.
\end{align}
Assuming this momentarily, it follows that
\begin{align*}
&
vw\br{\bbr{x\in \mathbb{R}^n: \mathcal{A}_wF(x)>2\lambda}}
\\
&
\qquad \leq vw\br{\bbr{x\in \mathbb{R}^n: \mathcal{A}_wF(x)>2\lambda,\mathcal{C}_{w,p_0}F(x)\leq \gamma\lambda}}
+ vw\br{\bbr{x\in \mathbb{R}^n: \mathcal{C}_{w,p_0}F(x)> \gamma\lambda}}
\\
&
\qquad\leq c\gamma^{c_{w,v}}vw\br{\bbr{x\in \mathbb{R}^n: \mathcal{A}_w^{\alpha}F(x)>\lambda}} + vw\br{\bbr{x\in \mathbb{R}^n: \mathcal{C}_{w,p_0}F(x)> \gamma\lambda}}.
\end{align*}
This easily gives
\begin{align*}
\norm{\mathcal{A}_wF}_{L^p(vdw)}^p \leq C\gamma^{c_{w,v}}\norm{\mathcal{A}_w^{\alpha}F}_{L^p(vdw)}+C_{\gamma,p} \norm{\mathcal{C}_{w,p_0}F}_{L^p(vdw)}^p.
\end{align*}
Note that, from Proposition \ref{prop:alpha} we know that $\|\mathcal{A}_w^{\alpha}F\|_{L^p(vdw)}\leq c(\alpha,p)\|\mathcal{A}_wF\|_{L^p(vdw)}$. Then, choosing $\gamma$ small enough so that $C\gamma^{c_{w,v}}c(\alpha,p)<1$ and from the fact that $\norm{\mathcal{A}_wF}_{L^p(vdw)}<\infty$, we conclude that
\begin{align*}
\norm{\mathcal{A}_wF}_{L^p(vdw)}\lesssim \norm{\mathcal{C}_{w,p_0}F}_{L^p(vdw)}.
\end{align*}

Therefore, in order to complete the proof it just remains to show  \eqref{good-lambda:A-Cp:w}. We argue as in \cite{CoifmanMeyerStein}. Consider $O_\lambda:=\{x\in \mathbb{R}^n:\mathcal{A}_w^{\alpha}F(x)>\lambda\}$, and note that \eqref{change-alph-1} and \eqref{normfinite} yield that $vw(O_\lambda)<\infty$, for all $\lambda>0$ and as before $O_\lambda \subsetneq \R^n$. Without loss of generality we can also suppose that $O_\lambda\neq\emptyset$ (otherwise both terms in \eqref{good-lambda:A-Cp:w} vanish, since $\mathcal{A}_w^{\alpha}F\geq \mathcal{A}_wF$ for $\alpha>1$, and then the  claim  is trivial). Note finally that $O_\lambda$ is open (see for instance \cite[Proposition 3.2]{MartellPrisuelos}). We can then take a Whitney decomposition  of $O_\lambda$ (cf. \cite[Chapter VI]{St70}): there exists a family of  closed cubes  $\{Q_j\}_{j\in \N}$ with disjoint interiors satisfying \eqref{Whitney}.
In particular, for each $j\in\N$ we can pick $x_j\in \R^n\setminus O_\lambda$ such that $d(x_j,Q_j)\leq 4 \textrm{diam}(Q_j)$. Furthermore, note again that $\mathcal{A}_w^{\alpha}F\geq \mathcal{A}_wF$ for $\alpha>1$. Then
\begin{multline*}
vw\br{\bbr{x\in \R^n: \mathcal{A}_wF(x)>2\lambda,\mathcal{C}_{w,p_0}F(x)\leq \gamma\lambda}}
\\
=
vw\br{\bbr{x\in O_\lambda: \mathcal{A}_wF(x)>2\lambda,\mathcal{C}_{w,p_0}F(x)\leq \gamma\lambda}}
\\
=
\sum_{j\in\N} vw\br{\bbr{x\in Q_j: \mathcal{A}_wF(x)>2\lambda,\mathcal{C}_{w,p_0}F(x)\leq \gamma\lambda}}.
\end{multline*}
Thus, to show \eqref{good-lambda:A-Cp:w}, it is enough to prove, for each $j\in \N$,
\begin{align*}
vw\br{\bbr{x\in Q_j: \mathcal{A}_wF(x)>2\lambda, \mathcal{C}_{w,p_0} F(x)\leq \gamma\lambda}}
\leq c\gamma^{c_{w,v}}vw(Q_j).
\end{align*}
Finally note that since $v \in A_{\infty}(w)$, (cf. \eqref{pesosineq:RHq}), the above inequality follows at once if we show, for each $j\in \N$,
\begin{align}\label{good-lambda:A-Cp:dx}
w\br{\bbr{x\in Q_j: \mathcal{A}_wF(x)>2\lambda,\mathcal{C}_{w,p_0}F(x)\leq \gamma\lambda}}
\leq c\gamma^{p_0}w(Q_j).
\end{align}

Then, let us fix $j\in \N$ and obtain \eqref{good-lambda:A-Cp:dx}. There is nothing to prove if the set on its left-hand side is empty. Thus, we assume that there exists $\bar{x}_j\in \bbr{x\in Q_j: \mathcal{A}_wF(x)>2\lambda,\mathcal{C}_{w,p_0}F(x)\leq \gamma\lambda}$. Let $B_j$ be the ball such that $Q_j\subset \overline{B_j}$ with radius $r_{B_j}=\textrm{diam}(Q_j)/2$. Then,
$d(x_j,Q_j)\leq 8r_{B_j}$ and $Q_j\subset \overline{B(x_j, 10r_{B_j})}$.

We now write
\begin{align*}
F(x,t)
=
F_{1,j}(x,t)+F_{2,j}(x,t)
:=
F(x,t)\,\chi_{[r_{B_j},\infty)}(t)+
F(x,t)\,\chi_{(0,r_{B_j})}(t).
\end{align*}
Then, $\mathcal{A}_wF(x)\leq \mathcal{A}_wF_{1,j}(x)+\mathcal{A}_wF_{2,j}(x)$. Now, for $t>r_{B_j}$, $x\in Q_j$ , and $\alpha\ge 11$, we have that $B(x,t)\subset B(x_j,\alpha t)$. Hence,
\begin{multline}\label{AF1}
\mathcal{A}_wF_{1,j}(x)^2
=
\int_{r_{B_j}}^{\infty}\int_{|x-y|<t}|F(y,t)|^2 \frac{dw(y) \, dt}{tw(B(y,t))}
\\
\leq
\int_{0}^{\infty}\int_{|x_j-y|<\alpha t}|F(y,t)|^2 \frac{dw(y) \, dt}{tw(B(y,t))}
=
\mathcal{A}_w^{\alpha}F(x_j)^2 \leq \lambda^2,
\end{multline}
where the last inequality holds since $x_j\in \R^n\setminus O_\lambda$.
On the other hand, by our choice of $\bar{x}_j\in Q_j\subset \overline{B_j}\subset 2B_j$,  it follows that
\begin{multline}\label{AF2} 
\frac{1}{w({2B_j} )}\int_{{2B_j}}|\mathcal{A}_wF_{2,j}(x)|^{p_0} \, dw(x)
\\
=
\frac{1}{w( 2B_j )}\int_{{2B_j}}\left(\int_0^{{2} r_{B_j}}\int_{B(x,t)} |F(y,t)|^2 \frac{dw(y) \, dt}{tw(B(y,t))}\right)^{\frac{p_0}{2}}dw(x)
\le
\mathcal{C}_{w,p_0}F(\bar{x}_j)^{p_0}
\leq
(\gamma\lambda)^{p_0}.
\end{multline}
Using \eqref{AF1}, Chebychev's  inequality, and \eqref{AF2} we conclude \eqref{good-lambda:A-Cp:dx}:
\begin{multline*}
w(\{x\in Q_j: \mathcal{A}_wF(x)>2\lambda, \mathcal{C}_{w,p_0}F(x)
\leq 
\gamma \lambda\})\leq w(\{x\in Q_j: \mathcal{A}_wF_{2,j}(x)>\lambda\})
\\
\le
\frac1{\lambda^{p_0}}\int_{Q_j}|\mathcal{A}_wF_{2,j}(x)|^{p_0} \, dw(x)
\leq \gamma^{p_0}w( 2B_j)
\leq c\gamma^{p_0}w(Q_j).
\end{multline*}
This completes the proof of Step $1$.

\medskip

{\em Step $2$}: Take $F\in L^2_{\rm loc}(\R_{+}^{n+1},dwdt)$ and define, for every $N>1$, $F_N:=F\chi_{K_N}$. Then, since $F_N\in L^2(\R_{+}^{n+1},dwdt)$ and $\supp F_N\subset K_N$, we can apply Step $1$ and obtain that
$$
\|\mathcal{A}_wF_N\|_{L^p(vdw)}\lesssim \|\mathcal{C}_{w,p_0}F_N\|_{L^p(vdw)}\leq \|\mathcal{C}_{w,p_0}F\|_{L^p(vdw)},
$$
where the implicit constant is uniform in $N$.
Finally since $|F_N|\nearrow |F|$ in $\R^{n+1}_+$ (that is, $|F_N|$ is an increasing sequence which converges  to $|F|$), the monotone convergence theorem yields the desired estimate. This finishes the proof of $(a)$.

\medskip

We next turn to the proof of $(b)$.
For every $x_0\in \mathbb{R}^n$ and any ball $B\subset \mathbb{R}^n$ such that $x_0\in B$, we have
\begin{align*}
\left(\dashint_B\left(\int_0^{r_B}\int_{B(x,t)}|F(y,t)|^2\frac{dw(y) \, dt}{tw(B(y,t))}\right)^{\frac{p_0}{2}}dw(x)\right)^{\frac{1}{p_0}}
\leq \left(\dashint_B|\mathcal{A}_wF(x)|^{p_0} dw(x)\right)^{\frac{1}{p_0}}
\leq
\mathcal{M}_{p_0}^w(\mathcal{A}_wF)(x_0),
\end{align*}
where for any function $h$, $\mathcal{M}_{p_0}^w h(x):= \mathcal{M}^w\big(|h|^{p_0}\big)(x)^{\frac1{p_0}}$ and $\M^w$ is defined in \eqref{weightedHLM}. Taking the supremum over all balls containing $x_0$, we conclude that $\mathcal{C}_{w,p_0}F(x_0)\leq \mathcal{M}_{p_0}^w(\mathcal{A}_wF)(x_0)$.
Moreover, since $p>p_0$ and $v\in A_{\frac{p}{p_0}}(w)$, it follows from Remark \ref{remark:weightedHLM} that $\mathcal{M}_{p_0}^w$ is bounded on $L^{p}(vdw)$. Thus we conclude that
$$
\|\mathcal{C}_{w,p_0}F\|_{L^{p}(vdw)}\leq \|\mathcal{M}_{p_0}^w(\mathcal{A}_wF)\|_{L^{p}(vdw)}\lesssim  \|\mathcal{A}_wF\|_{L^{p}(vdw)}.
$$
This completes the proof.
\end{proof}

\bigskip

\section{Proofs of the main results}\label{section:proofs}

In this section we prove Theorems \ref{thm:SF-Heat}--\ref{theor:control-SF-Poisson}. To this aim we first establish in Section \ref{firstproof} the boundedness of  $\Gcal_{\hh}^{L_w}$ and $\Scal_{\hh}^{L_w}$ on $L^p(w)$. In Section \ref{section:proof:Heat} we study the boundedness of $\Gcal_{\hh}^{L_w}$  on $L^p(vdw)$  with the help of the operator $\mathcal{C}_{w,p}$ introduced in \eqref{CF}. This and Theorem \ref{theor:control-SF-Heat} easily yield the desired estimates in Theorem \ref{thm:SF-Heat}.  In Section \ref{section:proof:SP-bounded} we see that Theorem \ref{thm:SF-Poisson} follows at once from Theorems  \ref{thm:SF-Heat} and \ref{theor:control-SF-Poisson}. Finally the proofs of Theorems \ref{theor:control-SF-Heat} and \ref{theor:control-SF-Poisson} are given in  Sections \ref{section:proof:control:Heat} and \ref{section:proof:control:Poisson},  respectively. We note that we improve some of the results in \cite{BCKYY,YangZ} by considering wider families of  conical square functions, allowing estimates on $L^p(vdw)$ (in place of on $L^p(w)$) and also enlarging considerably the ranges of the estimates. For instance,  \cite{BCKYY,YangZ} establish that $\Scal_{\hh}^{L_w}$  is bounded on $L^p(w)$ for $p\in(\frac{2n}{n+1},\frac{2n}{n-1})$. Here we obtain that it is bounded on $L^p(vdw)$ for all $p_-(L_w)<p<\infty$ and every  $v\in A_{\frac{p}{p_-(L_w)}}(w)$. Note that, in particular, we can take $v\equiv 1$ and in that case we get boundedness in the range $(p_-(L_w),\infty)$ which is clearly bigger as $p_-(L_w)< \frac{2n}{n+1}<\frac{2n}{n-1}<\infty$. We finally observe that it was shown in \cite{AuscherHofmannMartell} that the ranges for the boundedness of some conical square functions associated with the heat-semigroup are sharp, hence our ranges in that case are also sharp.

In what follows, unless otherwise specified, $L_w$ is a degenerate elliptic operator as in \eqref{degenerateL} with fixed $w\in A_2$.  In the context of Theorems \ref{thm:SF-Heat} and \ref{thm:SF-Poisson}, the considered conical square functions are sublinear operators a priori defined in $L^2(w)$. When we say that any of them is  bounded on $L^p(vdw)$ we mean that it satisfies estimates on $L^p(vdw)$ for any function in $L^\infty_c(\re^n)$  (or in $L^2(w)\cap L^p(vdw)$) where $L^\infty_c(\re^n)$ stands for the space of essentially bounded functions with compact support. It is standard to see that since $L^\infty_c(\re^n)$ in dense in $L^p(vdw)$, one can uniquely extend the conical square function to a bounded operator on $L^p(vdw)$. We will skip this standard argument below.


\subsection{Boundedness of $\Gcal_{\hh}^{L_w}$ and $\Scal_{\hh}^{L_w}$ on $L^p(w)$}\label{firstproof}


We recall that $\Gcal_{\hh}^{L_w}=\Gcal_{0,\hh}^{L_w}$ (cf. \eqref{square-H-3}) and $\Scal_{\hh}^{L_w} =\Scal_{1,\hh}^{L_w}$ (cf. \eqref{square-H-1}).

\begin{proposition}\label{thm:boundednessGH}
The conical square functions $\Gcal_{\hh}^{L_w}$ and $\Scal_{\hh}^{L_w}$ are bounded on $L^p(w)$ for all $p_-(L_w)<p<\infty$.
\end{proposition}
\begin{proof}
Note first that it is trivial to see that $\mathcal{S}_{\hh}^{L_w}f\leq\frac{1}{2} \mathcal{G}^{L_w}_{\hh}f$, hence it is enough to establish the $L^p(w)$ boundedness for $\Gcal_{\hh}^{L_w}$.  We split the proof into three cases: $p=2$, $2<p<\infty$, and $p_-(L_w)<p<2$.

{\em Case $1$: $p=2$}. 
Recall that the vertical square function
$$
\mathrm{g}_{\hh}^{L_w}f(y):=
\Bigg(\int_0^{\infty}\left|t\nabla_{y,t}e^{-t^2L_w}f(y)\right|^2\frac{dt}{t}\Bigg)^{\frac{1}{2}}
$$
is bounded on $L^2(w)$ (see \cite{CruzMartellRios}). Then, applying  Fubini's theorem, it follows that,  for every $f\in L^2(w)$, 
\begin{multline*}
\|\Gcal_{\hh}^{L_w}f\|_{L^2(w)}=
\Bigg(\int_{\R^n}\int_0^{\infty}\int_{B(x,t)}\left|t\nabla_{y,t}e^{-t^2L_w}f(y)\right|^2\frac{dw(y)\,dt}{tw(B(y,t))}dw(x)\Bigg)^{\frac{1}{2}}
\\
=\Bigg(\int_{\R^n}\int_0^{\infty}\left|t\nabla_{y,t}e^{-t^2L_w}f(y)\right|^2\int_{B(y,t)}dw(x)\frac{dw(y)\,dt}{tw(B(y,t))}\Bigg)^{\frac{1}{2}}
=\norm{\mathrm{g}_{\hh}^{L_w}f}_{L^2(w)}
\lesssim
\|f\|_{L^2(w)}.
\end{multline*}

\medskip
{\em Case $2$: $p>2$.} Fix $f\in  L^\infty_c(\R^n)$ and let $F(y,t):=t\nabla_{y,t}e^{-t^2L_w}f(y)$ so that
$\mathcal{A}_wF=\Gcal_{\hh}^{L_w}f$ (recall the definition of $\mathcal{A}_w$ in \eqref{AF}). Moreover, for every $x_0\in \R^n$, denote
$$
\widetilde{\mathcal{C}}_{w}f(x_0):=\mathcal{C}_{w}F(x_0)=\sup_{B\ni x_0}\Bigg(\frac{1}{w(B)}\int_0^{r_B}\int_B|t\nabla_{y,t}e^{-t^2L_w}f(y)|^2\frac{dw(y)\,dt}{t}\Bigg)^{\frac{1}{2}}.
$$
Then, Proposition \ref{prop:maximal}, part ($a$),  for $v\equiv 1$, and \eqref{Cw-Cw,2:equiv} imply that
$$
\norm{\Gcal_{\hh}^{L_w} f}_{L^p(w)} = \norm{\mathcal{A}_wF}_{L^p(w)} \lesssim \norm{\C_{w,2}F}_{L^p(w)}
\approx \norm{{\C}_{w} F}_{L^p(w)} =
\Vert\widetilde{\mathcal{C}}_{w}f\Vert_{L^p(w)}.
$$
Consequently, it suffices to prove that 
\begin{align}\label{C-M:w}
\|\widetilde{\mathcal{C}}_{w}f\|_{L^p(w)}\lesssim \|\mathcal{M}_{2}^wf\|_{L^p(w)},
\end{align}
where $\mathcal{M}_2^wf:=(\mathcal{M}^w|f|^2)^{\frac{1}{2}}$, since $\mathcal{M}_{2}^w$ is bounded on $L^p(w)$ for $p>2$ (see Remark \ref{remark:weightedHLM}).

In order to prove \eqref{C-M:w}, take $B$ a ball in $\R^n$ with radius $r_B$ and split $f$ into its local and its global part:
$f=f_{\rm loc}+f_{\rm glob}:=f\chi_{4B}+f\chi_{\R^{n}\setminus 4B}$.
Then, applying the boundedness of $g_{\hh}^{L_w}$ on $L^2(w)$, we obtain
\begin{multline}\label{localp>2}
\br{\frac{1}{w(B)}\int_{B}\int_0^{r_B}|t\nabla_{y,t}e^{-t^2L_w}f_{\rm loc}(y)|^2\frac{dt\,dw(y)}{t}}^{1/2}
\leq 
\br{\frac{1}{w(B)} \int_{\R^n} |g_{\hh}^{L_w} f_{\rm loc}(y)|^2 dw(y)}^{1/2}
\\
\lesssim
\br{\frac{1}{w(B)}\int_{\R^n}|f_{\rm loc}(y)|^2dw(y)}^{1/2}
\lesssim
\br{\dashint_{4B}|f(y)|^2dw(y)}^{1/2}
\lesssim
\inf_{x\in B} \mathcal M_2^wf(x).
\end{multline}
As for the global part, since by Lemma \ref{off-diag-sg} we have that 
$\sqrt{t}\nabla_{y}e^{-tL_w},\ tL_we^{-tL_w} \in\mathcal{O}(L^2(w)-L^2(w))$, then
\begin{align}\label{globalp>2}
\Bigg(\frac{1}{w(B)}&\int_{B}\int_0^{r_B}|t\nabla_{y,t}e^{-t^2L_w}f_{\rm glob}(y)|^2\frac{dt\,dw(y)}{t}\Bigg)^{1/2}
\\
\nonumber
&\lesssim 
\sum_{j\geq 2} \br{\int_0^{r_B}\dashint_{B}|t\nabla_{y,t}e^{-t^2L_w}(f\chi_{C_j(B)})(y)|^2\,dw(y)\frac{dt}{t}}^{1/2}
\\
\nonumber
&\lesssim 
\sum_{j\geq 2}2^{j\theta_1} \br{\int_0^{r_B}\Upsilon\br{\frac{2^j r_B}{t}}^{2\theta_2}e^{-\frac{c4^jr_B^2}{t^2}}
\dashint_{2^{j+1}B}|f(y)|^2dw(y)\frac{dt}{t}}^{1/2}
\\
\nonumber
&\lesssim 
\sum_{j\geq 2}2^{j\theta_1} \br{\int_0^{r_B}\br{\frac{2^j r_B}{t}}^{2\theta_2}e^{-\frac{c4^jr_B^2}{t^2}}\frac{dt}{t}}^{1/2}
\inf_{x\in B} \mathcal M_2^wf(x)
\\
\nonumber
&\lesssim \inf_{x\in B} \mathcal M_2^wf(x).
\end{align}
Hence by \eqref{localp>2} and \eqref{globalp>2}, we conclude that, for every $x_0\in \R^n$ and every ball $B\subset \R^n$ such that $x_0\in B$,
\begin{align*}
\br{\frac{1}{w(B)}\int_0^{r_B}\int_B|t\nabla_{y,t}e^{-t^2L_w}f(y)|^2\frac{dw(y)\,dt}{t}}^{1/2}
\lesssim 
\mathcal M_2^wf(x_0).
\end{align*}
Now taking the supremum over all balls $B\ni x_0$, we have that
$\widetilde{C}_wf(x_0)\lesssim \mathcal M_2^wf(x_0)$ for all $x_{0}\in \R^n$ and
hence \eqref{C-M:w} follows.
\medskip

{\em Case $3$: $p_-(L_w)<p<2$.} 
Since $\Gcal_{\hh}^{L_w}$ is bounded on $L^q(w)$ for $2\le q<\infty$, by  the  Marcinkiewicz interpolation theorem, it is enough to prove that $\Gcal_{\hh}^{L_w}$ maps continuously $L^{p}(w)$ into $L^{p,\infty}(w)$ for all $p_-(L_w)<p<2$. That is, we wish to show that for all $\lambda>0$ and $f\in  L^\infty_c(\re^n)$, 
\begin{equation}\label{weak p}
w\br{\bbr{x\in \R^n: \Gcal_{\hh}^{L_w}f(x)>\lambda}} \lesssim \frac{1}{\lambda^p} \int_{\R^n}|f(x)|^pdw(x).
\end{equation}
To this end, fix $p_-(L_w)<p<2$, $f\in  L^\infty_c(\re^n)$, and $\lambda>0$.  Here we need to adapt the argument in \cite[p. 5480]{AuscherHofmannMartell} and proceed as in \cite[p. 61]{Auscher} (notice that, as it was already observed in the latter, \cite[Theorem 1.1]{Auscher} does not apply due to the nature of the conical square function). For starters we need  a 
Calder{\'o}n-Zygmund decomposition for $f$ adapted to the $L^p(w)$ norm. This is quite standard, but we need some extra features, hence we sketch the argument. First, using the notation in \eqref{weightedHLM}  consider the level set $E_\lambda=\{x\in\re^n: \mathcal{M}^w(|f|^p)(x)^{1/p}>\lambda\}$. Since this is clearly open, we can dyadically divide the standard Whitney cubes, as constructed in \cite[Chapter
VI]{St70}, to find a pairwise disjoint family of dyadic cubes $\{Q_i\}_{i\in\mathbb{N}}$ such that $E_\lambda=\cup_{i\in\mathbb{N}} Q_i$; the family $\{4B_i\}_{i\in\mathbb{N}}$ has bounded overlap, where $B_i=B(x_{Q_i}, \sqrt{n}\ell(Q_i))$; and $B_i^*\cap (\re^n\setminus E_\lambda)\neq\emptyset$ where $B_i^*=C_n B_i$ with $C_n>8$ a large enough constant depending only $n$. We can then write $f=g+\sum_{i=1}^{\infty}b_i$ with 
\begin{equation}\label{decomp}
g:=
f\chi_{\re^n\setminus E_\lambda}+\sum_{i=1}^\infty \left(\dashint_{Q_i} f(x)\,dw(x)\right)\chi_{Q_i}
\qquad\text{and}\qquad
b_i:=\left(f-\dashint_{Q_i} f(x)\,dw(x)\right)\chi_{Q_i}.
\end{equation}
We claim that the following properties hold:
\begin{equation}\label{g}
\|g\|_{L^{\infty}(w)}\leq C\lambda,
\end{equation}
\begin{equation}\label{b}
\supp b_i\subset Q_i\subset B_i\quad \textrm{and}\quad \int_{B_i}| b_i(x)|^pdw(x)\leq C\lambda^pw(B_i),
\end{equation}
\begin{equation}\label{sum}
\sum_{i=1}^{\infty}w(B_i)\leq \frac{C}{\lambda^p}\int_{\R^n}|f(x)|^pdw(x),
\end{equation}
\begin{equation}\label{overlap}
\sum_{i=1}^{\infty}\chi_{4B_i}\leq N,
\end{equation}
where $C$ and $N$ depend only on the dimension, $p$ and $w$. To show \eqref{g} we combine Lebesgue's differentiation theorem for the doubling measure $w$ and the fact that 
\begin{equation}\label{eq:qwf5fgt}
\left|\dashint_{Q_i} f(x)\,dw(x)\right|
\lesssim
\left(\dashint_{B_i^*} |f(x)|^p\,dw(x)\right)^{\frac1p}
\le \inf_{B_i^*} \mathcal{M}^w(|f|^p)^{\frac1p}
\le
\lambda,
\end{equation}
where the last estimate uses that $B_i^*\cap (\re^n\setminus E_\lambda)\neq\emptyset$. Analogously, 
$$
\left(\dashint_{B_i} |b_i(x)|^p\,dw(x)\right)^{\frac1p}
\lesssim
\left(\dashint_{B_i} |f(x)|^p\,dw(x)\right)^{\frac1p} 
\lesssim
\left(\dashint_{B_i^*} |f(x)|^p\,dw(x)\right)^{\frac1p}
\le
\lambda,
$$
and this gives \eqref{b}. To obtain \eqref{sum} we use that $w$ is a doubling measure (cf.~\eqref{pesosineqw:Ap}), that the cubes $\{Q_i\}_i$ are pairwise disjoint, and the weak-type $(1,1)$ inequality for $\mathcal{M}^w$:
$$
\sum_{i=1}^{\infty}w(B_i)
\lesssim
\sum_{i=1}^{\infty}w(Q_i)
=
w(E_\lambda)
\lesssim
\frac{1}{\lambda^p}\int_{\R^n}|f(x)|^pdw(x).
$$
Finally \eqref{overlap} follows from the construction of the Whitney cubes. 

Next, in order to justify the following computations we show that $g$ and $b_i$ belong to $L^2(w)$. Note first that by \eqref{eq:qwf5fgt} and the weak-type $(1,1)$ inequality for $\mathcal{M}^w$ we obtain
$$
\|g\|_{L^2(w)}
\le
\|f\|_{L^2(w)}+ \lambda w(E_\lambda)^{\frac12}
\le
\|f\|_{L^2(w)}+ \lambda^{1-\frac{p}2}\|f\|_{L^p(w)}^{\frac{p}2}<\infty,
$$
since $f\in L^\infty_c(\re^n)$. Analogously,
$$
\|b_i\|_{L^2(w)}
\le
\|f\|_{L^2(w)}+\lambda w(Q_i)^{\frac12}<\infty.
$$

To continue, for each $i\in \N$, denote by $r_i$ the radius of $B_i$ and  consider the operator $
A_{r_i}=I-(I-e^{-r_i^2L_w})^M,
$
where $M\in \N$ will be determined later.
Then,
\begin{align*}
\Gcal_{\hh}^{L_w}f
\leq
\Gcal_{\hh}^{L_w}g+\Gcal_{\hh}^{L_w}\br{\sum_{i=1}^{\infty}A_{r_i}b_i}+\Gcal_{\hh}^{L_w}\br{\sum_{i=1}^{\infty}\br{I-e^{-r_i^2L_w}}^M b_i}.
\end{align*}
Therefore, for all $\lambda>0$,
\begin{align*}
w\br{\bbr{x\in \R^n: \Gcal_{\hh}^{L_w}f(x)>\lambda}}
&\leq 
w\br{\bbr{x\in \R^n: \Gcal_{\hh}^{L_w}g(x)>\frac{\lambda}{3}}}
\\
&\qquad+
w\br{\bbr{x\in \R^n: \Gcal_{\hh}^{L_w}\br{\sum_{i=1}^{\infty}A_{r_i}b_i}(x)>\frac{\lambda}{3}}}
\\
&\qquad+
w\br{\bbr{x\in \R^n: \Gcal_{\hh}^{L_w}\br{\sum_{i=1}^{\infty}\br{I-e^{-r_i^2L_w}}^Mb_i}(x)>\frac{\lambda}{3}}}
\\
&=:I+II+III.
\end{align*}

We estimate each term in turn. Using the $L^2(w)$ boundedness of $\Gcal_{\hh}^{L_w}$ and properties \eqref{decomp}$-$\eqref{sum}, it follows that
$$
I \lesssim \frac{1}{\lambda^2} \int_{\R^n}|g(x)|^2 dw(x) 
\lesssim \frac{1}{\lambda^p}\int_{\R^n}|g(x)|^p dw(x) \lesssim\frac{1}{\lambda^p}
\int_{\R^n}|f(x)|^pdw(x).
$$
To estimate  term $II$, we take $0\le \psi \in L^2(w)$ with $\norm{\psi}_{L^2(w)}=1$, and obtain that
$$
\int_{\R^n} \Bigg|\sum_{i=1}^{\infty} A_{r_i}b_i\Bigg| \psi dw \le 
\sum_{i=1}^{\infty} \sum_{j=1}^{\infty} w(2^{j+1}B_i) \br{\dashint_{C_j(B_i)} |A_{r_i}b_i|^2 dw}^{\frac12} 
\br{\dashint_{2^{j+1}B_i} |\psi|^2 dw}^{\frac12}.
$$
Besides, observe that $A_{r_i}=\sum_{k=1}^{M} c_{k,M} e^{-r_i^2 k L_w}$ satisfies $\O(L^p(w)-L^2(w))$ by Lemma \ref{off-diag-sg}. This and properties \eqref{b}$-$\eqref{overlap} yield
\begin{align*}
\int_{\R^n} \Bigg|\sum_{i=1}^{\infty} A_{r_i}b_i\Bigg| \psi dw 
&\lesssim 
\sum_{i=1}^{\infty} \sum_{j=1}^{\infty} 2^{j\theta_1} \Upsilon\br{2^j}^{\theta_2} e^{-c4^j} w(2^{j+1}B_i) \Bigg(\dashint_{B_i}|b_i|^pdw\Bigg)^{\frac{1}{p}}  \inf_{B_i} \mathcal{M}_2^w\psi
\\&\lesssim 
\lambda \sum_{i=1}^{\infty} \sum_{j=1}^{\infty} e^{-c4^j} \int_{B_i} \mathcal{M}_2^w \psi(y) dw(y)
\\
&
\lesssim 
\lambda \int_{\cup_i B_i} \mathcal{M}_2^w \psi(y) dw(y)
\\&\lesssim
\lambda w(\cup_i B_i)^{1/2} \norm{\psi^2}_{L^{1}(w)}^{\frac12} 
\lesssim \lambda^{1-p/2} \norm{f}_{L^p(w)}^{\frac p2},
\end{align*}
where in the last step above we have used Kolmogorov's inequality along with the fact that $\mathcal{M}^w$ is of weak type $(1,1)$ (with respect to $w$).
Taking the supremum over all $\psi$ as above and using again that $\Gcal_{\hh}^{L_w}$ is bounded on $L^2(w)$
we obtain
$$
II\lesssim 
\frac{1}{\lambda^2}\norm{\sum_{i=1}^{\infty} A_{r_i}b_i}_{L^2(w)}^2
\lesssim 
\frac{1}{\lambda^p} \norm{f}_{L^p(w)}^p.
$$
Next we estimate $III$.
Applying Chebychev's inequality and \eqref{sum}, we have that
\begin{align}\label{III:1}
III
&\lesssim
w\br{\bigcup_{i=1}^{\infty}5B_i}
+
w\br{\bbr{x\in \R^n\setminus \bigcup_{i=1}^{\infty}5B_i: 
\Gcal_{\hh}^{L_w}\br{\sum_{i=1}^{\infty}\br{I-e^{-r_i^2L_w}}^M b_i}(x)>\frac{\lambda}{3}}}
\\\nonumber
&\lesssim
\sum_{i=1}^{\infty}w(B_i)+
\frac{1}{\lambda^2}\int_{\R^n\setminus \bigcup_{i=1}^{\infty}5B_i}
\abs{\Gcal_{\hh}^{L_w}\br{\sum_{i=1}^{\infty}\br{I-e^{-r_i^2L_w}}^Mb_i}(x)}^2dw(x)
\\\nonumber
&\lesssim
\frac{1}{\lambda^p}\int_{\R^n}|f(x)|^pdw(x)+\frac{1}{\lambda^2}\int_{\R^n\setminus \bigcup_{i=1}^{\infty}5B_i}
\abs{\Gcal_{\hh}^{L_w} \br{\sum_{i=1}^{\infty}\br{I-e^{-r_i^2L_w}}^Mb_i}(x)}^2dw(x).
\end{align}
It remains to estimate the last integral above. Denote $h_i=(I-e^{-r_i^2L_w})^M b_i$. Then, applying  Fubini's theorem,
\begin{align}\label{III:2}
&\int_{\R^n\setminus \bigcup_{i=1}^{\infty}5B_i} \abs{\Gcal_{\hh}^{L_w}\br{\sum_{i=1}^{\infty} h_i}(x)}^2dw(x)
\\ \nonumber
&
\qquad
\lesssim
\int_0^{\infty}\int_{\R^n}\abs{\sum_{i=1}^{\infty} \chi_{4B_i}(y) t\nabla_{y,t}e^{-t^2L_w}h_i(y)}^2 
w\left(B(y,t)\setminus \bigcup_{i=1}^{\infty}5B_i\right)\,\frac{dw(y)\, dt}{tw(B(y,t))}
\\
\nonumber
&
\qquad\qquad
+
\int_0^{\infty}\int_{\R^n}\abs{\sum_{i=1}^{\infty}\chi_{\R^n\setminus 4B_i}(y) t\nabla_{y,t}e^{-t^2L_w}h_i(y)}^2 
w\left(B(y,t)\setminus \bigcup_{i=1}^{\infty}5B_i\right)\,\frac{dw(y)\, dt}{tw(B(y,t))}
\\
\nonumber
&
\qquad
=:I_{\rm loc}+I_{\rm glob}.
\end{align}

Recall that the collection $\{4B_i\}_{i\in \N}$ has finite overlap. Besides, given $y\in 4B_i$, if there exists $x\in B(y,t)\setminus \bigcup_i 5B_i$, then $t>r_i$. Hence
\begin{align*}
I_{\rm loc} 
&\lesssim 
\sum_{i=1}^{\infty}\int_{r_i}^{\infty} \int_{4B_i} \abs{t\nabla_{y,t} e^{-t^2L_w}h_i(y)}^2 w\left(B(y,t)\setminus \bigcup_{i=1}^{\infty}5B_i\right)\,
\frac{dw(y)\, dt}{tw(B(y,t))}
\\
&\lesssim 
\sum_{i=1}^{\infty} \int_{r_i}^{\infty}\int_{4B_i}\bigg|t\nabla_{y,t}e^{-t^2L_w}h_i(y)\bigg|^2 \frac{dw(y)\, dt}{t}.
\end{align*} 
Since $\sqrt{t}\,\nabla_ye^{-t\,L_w}$, $t\,L_w\,e^{-t\,L_w}$ belong to $\mathcal{O}(L^p(w)-L^2(w))$, by Lemma \ref{off-diag} , we can find $\gamma_1 \le \gamma_2$ so that \eqref{off-AlphaBeta} holds simultaneously for both (with $\alpha=\gamma_1$ and $\beta=\gamma_2$). This, the $L^p(w)$ boundedness of the heat semigroup  and \eqref{b} allow us to obtain that there is $\gamma_1>0$ such that for every $t>r_i$, 
\begin{multline*}
\br{\dashint_{4B_i}\abs{t\nabla_{y,t}e^{-t^2L_w}(h_i\chi_{4B_i})(y)}^2 dw(y)}^{1/2}
\lesssim 
\br{\frac{r_i}{t}}^{\gamma_1}\br{\dashint_{4B_i}|h_i(y)|^p dw(y)}^{\frac{1}{p}}
\\ 
 \lesssim 
\br{\frac{r_i}{t}}^{\gamma_1}\br{\dashint_{B_i}|b_i(y)|^p dw(y)}^{\frac{1}{p}}
\lesssim 
\br{\frac{r_i}{t}}^{\gamma_1} \lambda.
\end{multline*} 
Using a similar argument and expanding $(I-e^{-r_i^2L_w})^M$ it follows that
\begin{align*}
\Bigg(\dashint_{4B_i}\Big|t\nabla_{y,t}e^{-t^2L_w}&(h_i\chi_{\R^n\setminus 4B_i})(y)\Big|^2 dw(y)\Bigg)^{\frac{1}{2}}
\\
\lesssim& 
\sum_{j\geq 2}\, \left(\frac{w(2^{j+1}B_i)}{w(4B_i)} \dashint_{2^{j+1}B_i}\abs{t\nabla_{y,t}e^{-t^2L_w}(h_i\chi_{C_j(B_i)})(y)}^2 dw(y)\right)^{\frac{1}{2}}
\\
 \lesssim & 
\sum_{j\geq 2}2^{nj}\,  \left(\dashint_{2^{j+1}B_i}\abs{t\nabla_{y,t}e^{-t^2L_w}((h_i\chi_{C_j(B_i)})\chi_{2^{j+1}B_i})(y)}^2 dw(y)\right)^{\frac{1}{2}}
\\
\lesssim&
\sum_{j\geq 2} 2^{j(n+\gamma_2)}\br{\frac{r_i}{t}}^{\gamma_1}\br{\dashint_{C_j(B_i)}|h_i(y)|^p dw(y)}^{\frac{1}{p}}
\\
\lesssim&
\sum_{j\geq 2} 2^{j(n+\gamma_2)} \br{\frac{r_i}{t}}^{\gamma_1} \sum_{k=1}^{M}\br{\dashint_{C_j(B_i)}|e^{-kr_i^2L_w}b_i(y)|^p dw(y)}^{\frac{1}{p}}
\\
\lesssim&
\sum_{j\geq 2} 2^{j(n+\gamma_2 +\theta_1+\theta_2)} \br{\frac{r_i}{t}}^{\gamma_1}e^{-c4^j}
 \br{\dashint_{B_i} |b_i(y)|^p dw(y)}^{\frac{1}{p}}
\\
\lesssim&
\br{\frac{r_i}{t}}^{\gamma_1} \lambda,
\end{align*}
 where in the fourth inequality we have used that  the term $k=0$ vanishes since $b_i$ is  supported in $B_i$ and the integral takes place in $C_j(B_i)$ with $j\ge 2$. 
Therefore, 
\begin{equation}\label{Iloc}
I_{\rm loc}
\lesssim 
\lambda^2 \sum_{i=1}^{\infty}w(4B_i) \int_{r_i}^{\infty}   \br{\frac{ r_i}{t}}^{2\gamma_1}\frac{dt}{t}
\lesssim 
\lambda^2 \sum_{i=1}^{\infty}
 w(B_i)
\lesssim
\lambda^{2-p} \int_{\R^n}|f(x)|^p dw(x).
\end{equation}

We turn now to  estimating  $I_{\rm glob}$. We claim that for every $i\ge 1$ and $j\ge 2$, if $M$ is chosen large enough,
\begin{equation}\label{Iij}
I_{ij} 
:=
\br{\int_{0}^{\infty}\dashint_{ C_j(B_i)} \big|t\nabla_{y,t}e^{-t^2L_w}
h_i(y)\big|^2 dw(y)\, \frac{dt}{t}}^{\frac12} 
\lesssim 2^{-j(2M-\theta_1)}\lambda.
\end{equation}
Assuming this momentarily,  take $0\leq \Psi \in L^2(\R_{+}^{n+1}, \frac{dwdt}{t})$ with $\norm{\Psi}_{L^2(\R_{+}^{n+1}, \frac{dwdt}{t})}=1$
and write $\widetilde \Psi(y):=\int_{0}^{\infty} \Psi(y,t)^2 \frac{dt}{t}$. Then, taking $M>n+\theta_1/2$, using Kolmogorov's inequality and that $\mathcal{M}^w$ is of weak type-$(1,1)$ we conclude that
\begin{align*}
&
\int_0^{\infty} \int_{\R^n} \bigg|\sum_{i=1}^{\infty}\chi_{\R^n\setminus 4B_i}(y) t\nabla_{y,t}e^{-t^2L_w}h_i(y)\bigg| \Psi(y,t) \frac{dw(y)\, dt}{t}
\\
 &
\qquad
\leq
\sum_{i=1}^{\infty} \sum_{j\geq 2} w(2^{j+1}B_i) I_{ij}
 \br{\int_{0}^{\infty}\dashint_{ C_j(B_i)} \Psi(y,t)^2 \frac{dw(y)\, dt}{t}}^{\frac12}
\\
 &
\qquad\leq
\lambda \sum_{j\geq 1} 2^{j(2n-2M+\theta_1)} \sum_{i=1}^{\infty} w(B_i) 
\inf_{x \in B_i} \br{\mathcal M^{w}\widetilde \Psi(x)}^{\frac12}
\\
 &
\qquad\lesssim
\lambda \int_{\cup_i B_i} \br{\mathcal M^{w}\widetilde \Psi(x)}^{\frac12} dw(x)
\\
 &
\qquad\lesssim
\lambda w\Big(\bigcup_i B_i\Big)^{\frac12}
\lesssim
\lambda^{1-\frac{p}{2}} \norm{f}_{L^p(w)}^{\frac{p}{2}},
\end{align*}
where last estimate follows from \eqref{sum}. Taking the sup over all functions $\Psi$ as above yields
$$
\br{I_{\rm glob}}^{\frac12} 
\leq
\br{\int_0^{\infty} \int_{\R^n}\bigg|\sum_{i=1}^{\infty}\chi_{\R^n\setminus 4B_i}(y) t\nabla_{y,t}e^{-t^2L_w}h_i(y)\bigg|^2 \frac{dw(y)\, dt}{t}}^{\frac12}
\lesssim
\lambda^{1-\frac{p}{2}} \norm{f}_{L^p(w)}^{\frac{p}{2}}.
$$
Plugging this and \eqref{Iloc} into \eqref{III:2} and the latter into \eqref{III:1}, we see that $III\lesssim \lambda^{-p}\,\|f\|_{L^p(w)}^p$. This eventually finishes the proof of \eqref{weak p}.

To complete the proof of Case 3 we need to show \eqref{Iij}. Note first that
$$
\br{I_{ij}}^2 
\lesssim
\int_{0}^{\infty}\dashint_{ C_j(B_i)} \abs{t\nabla_{y}e^{-t^2L_w}h_i(y)}^2 dw(y) \frac{dt}{t}
+\int_{0}^{\infty}\dashint_{ C_j(B_i)} \abs{t^2 L_w e^{-t^2L_w}h_i(y)}^2 dw(y)\frac{dt}{t}
=:I_{ij,1}+I_{ij,2}. 
$$
We estimate the first term $I_{ij,1}$ by using functional calculus, where the notation is taken from \cite{Auscher} and \cite[Section 7]{AuscherMartell:III}. 
As usual write $\vartheta\in[0,\pi/2)$ for the sup of $|{\rm arg}(\langle Lf,f\rangle_{L^2(w)})|$ over all $f$ in the domain of $L_w$.
Let $0<\vartheta <\theta<\nu<\mu<\pi /2$ and note that, for a fixed $t>0$, $\phi(z,t):=e^{-t^2 z}(1-e^{-r_i^2 z})^M$ is holomorphic in the open sector $\Sigma_\mu=\{z\in\mathbb{C}\setminus\{0\}:|{\rm arg} (z)|<\mu\}$ and satisfies $|\phi(z,t)|\lesssim |z|^M\,(1+|z|)^{-2M}$ (with implicit constant depending on $t>0$, $r_i$,  $\mu$,  and $M$) for every $z\in\Sigma_\mu$. Hence we can write
\[
\phi(L_w,t )=\int_{\Gamma } e^{-zL_w }\eta (z,t)dz,
\qquad \text{where} \quad 
\eta(z,t) = \int_{\gamma} e^{\zeta z} \phi(\zeta,t ) d\zeta.
\]
Here $\Gamma=\partial\Sigma_{\frac\pi2-\theta}$ with positive orientation (although orientation is irrelevant for our computations) and 
$\gamma=\re_+e^{i\,{\rm sign}({\rm Im} (z))\,\nu}$. It is not difficult to see that for every $z\in \Gamma$,
\begin{equation}\label{eq:eta}
|\eta(z,t)| \lesssim \frac{r_i^{2M}}{(|z|+t^2)^{M+1}}.
\end{equation}
 where the implicit constant is independent of $t$ and $r_i$.
This, the fact that $\sqrt{z}\nabla_y e^{-zL_w}\in\mathcal O (L^p(w)-L^2(w))$ (see \cite[Corollary 7.4]{CruzMartellRios}), and \eqref{b} imply
\begin{align*}
\br{\dashint_{C_j(B_i)} \abs{t\nabla_y\phi(L_w,t)b_i}^2 dw}^{\frac12}
&\lesssim 
\int_{\Gamma} \br{\dashint_{C_j(B_i)} \abs{\sqrt{z}\nabla_y e^{- z L_w }b_i}^2 dw}^{\frac12}
\frac{t r_i^{2M}}{|z|^{\frac{1}{2}}(|z|+t^2)^{M+1}} |dz|
\\ &\lesssim 
2^{j\theta_1} \br{\dashint_{B_i} |b_i|^p dw}^{\frac1p}\int_0^\infty \Upsilon\br{\frac{2^j r_i}{\sqrt s}}^{\theta_2} e^{-\frac{c4^j r_i^2}{s}} t\,s^{1/2} \frac{r_i^{2M}}{(s+t^2)^{M+1}}\frac{ds}{s}
\\ &\lesssim
2^{j\theta_1}\lambda\int_0^\infty \Upsilon\br{\frac{2^j r_i}{\sqrt s}}^{\theta_2} e^{-\frac{c4^j r_i^2}{s}} t\,s^{1/2} \frac{r_i^{2M}}{(s+t^2)^{M+1}}\frac{ds}{s}.
\end{align*}
Hence, after changing  the  variables $s$ and $t$ into $\frac{4^j r_i^2}{s^2}$ and $2^j r_i t$  respectively,
\begin{multline*}
I_{ij,1}
\lesssim
2^{2j\theta_1}\lambda^2 \int_0^\infty \br{\int_0^\infty \Upsilon\br{\frac{2^j r_i}{\sqrt s}}^{\theta_2} e^{-\frac{c4^j r_i^2}{s}} t\,s^{1/2} \frac{r_i^{2M}}{(s+t^2)^{M+1}}\frac{ds}{s}}^2 \frac{dt}{t}
\\\approx
2^{-2j(2M-\theta_1)} \lambda^2 \int_0^\infty \br{\int_0^\infty \Upsilon\br{s}^{\theta_2} e^{-cs^2} \frac{t}{s} \frac{1}{\big(\frac{1}{s^2}+t^2\big)^{M+1}}\frac{ds}{s}}^2 \frac{dt}{t}
=:
2^{-2j(2M-\theta_1)} \lambda^2 \int_0^\infty \Theta(t)^2 \frac{dt}{t}.
\end{multline*}
Choosing $M$ so that $2M+1-\theta_2>0$, if $t\ge 1$,
$$
\Theta(t)
\le 
\int_0^{\frac1t} \frac 1{s^{\theta_2}} \frac{t}{s} s^{2(M+1)}\frac{ds}{s} + 
\int_{\frac1t}^{1} \frac 1{s^{\theta_2}} \frac{t}{s} \frac{1}{t^{2M+2}} \frac{ds}{s} 
+\int_{1}^{\infty} s^{\theta_2} e^{-cs^2} \frac{t}{s} \frac{1}{t^{2M+2}} \frac{ds}{s}
\lesssim  \frac{1}{t^{2M-\theta_2}}.
$$
Similarly, for $0<t<1$, 
$$
\Theta(t) \le 
\int_0^{1} \frac 1{s^{\theta_2}} \frac{t}{s} s^{2(M+1)} \frac{ds}{s} + 
\int_{1}^{\infty} s^{\theta_2} e^{-cs^2}  \frac{t}{s} s^{2M+2} \frac{ds}{s}
\lesssim t.
$$
Consequently, if $2M-\theta_2>0$,
\begin{align*}
I_{ij,1}
&\lesssim
2^{-2j(2M-\theta_1)} \lambda^2 \br{\int_0^1 t^2 \frac{dt}{t}
+ \int_1^\infty \frac{1}{t^{4M-2\theta_2}} \frac{dt}{t}} \lesssim 2^{-2j(2M-\theta_1)}  \lambda^2.
\end{align*}
This and an analogous estimate for $I_{ij,2}$ complete the proof of \eqref{Iij}.
 In fact, the formal argument for  $I_{ij,2}$ is the same as the one for $I_{ij,1}$, but taking $\phi(z,t):=t^2z e^{-t^2 z}(1-e^{-r_i^2 z})^M$. Consequently, we have that $|\eta(z,t)|\lesssim \frac{t^2r_i^{2M}}{(|z|+t^2)^{M+2}}$ in place of \eqref{eq:eta}, and we use $e^{-zL_w}\in\mathcal{O}(L^p(w)-L^2(w))$ instead of $\sqrt{z}\nabla_y e^{-zL_w}\in\mathcal O (L^p(w)-L^2(w))$. We leave the details to the interested reader. 

\end{proof}

\bigskip

\subsection{Proof of Theorem \ref{thm:SF-Heat}}\label{section:proof:Heat}


 Assuming momentarily that $\Gcal_{\hh}^{L_w}$ is bounded on $L^p(vdw)$. Then, \eqref{eq:all-SH} and Theorem \ref{theor:control-SF-Heat} part $(a)$ take care of the boundedness of $\Scal_{m,\hh}^{L_w}$, $G_{m,\hh}^{L_w}$, 
$\Gcal_{m,\hh}^{L_w}$ when $m\ge 1$, and also that of $G_{\hh}^{L_w}$. Thus we only need to consider the boundedness of $\Gcal_{\hh}^{L_w}$,  which will follow from Proposition \ref{prop:maximal}, part $(a)$ and the following auxiliary result: 

\begin{proposition}\label{prop:Cpo-Mp0-S-heat}
If we set
$$
\widetilde{\mathcal{C}}_{p_0}^{L_w}f(x)
:=\sup_{B\ni x}\br{\frac{1}{w(B)}\int_{B}\br{\int_{0}^{r_B}\int_{B(x,t)} \abs{t\nabla_{y,t}e^{-t^2L_w}f(y)}^2 \frac{dw(y) \, dt}{tw(B(y,t))}}^{\frac{p_0}{2}} \,dw(x)}^{\frac{1}{p_0}},
$$
then, for every $p_-(L_w)<p_0\le 2$ and $f\in L^\infty_c(\re^n)$ there holds
\begin{equation}\label{Cpo-Mp0-S-heat}
\widetilde{\mathcal{C}}_{p_0}^{L_w}f(x)\lesssim \mathcal{M}_{p_0}^wf(x), \qquad \forall \,x\in \mathbb{R}^n.
\end{equation}
Eventually, $\Gcal_{\hh}^{L_w}f$ is bounded on $L^p(vdw)$ for every $v\in A_\infty(w)$ and $p\in \mathcal{W}_v^w(p_-(L_w),\infty)$.
\end{proposition}

Assuming this result momentarily, Theorem \ref{thm:SF-Heat} follows immediately in view of Theorem \ref{theor:control-SF-Heat}.

\medskip
\begin{proof}[Proof of Proposition \ref{prop:Cpo-Mp0-S-heat}]
Fix $p_-(L_w)<p_0\le 2$ and  $x_0\in\R^n$. Take an arbitrary ball $B\ni x_0$ with radius $r_B$ and split  $f\in L^\infty_c(\re^n)$  into its local and its global part:
$f=f_{\rm loc}+f_{\rm glob}:=f\chi_{8B}+f\chi_{\R^n\setminus 8B}$.

In order to estimate $f_{\rm loc}$, note that by Proposition \ref{thm:boundednessGH}, $\Gcal_{\hh}^{L_w}$ is bounded on $L^{p_0}(w)$. Then
\begin{multline*}
\br{\frac{1}{w(B)}\int_{B}\br{\int_{0}^{r_B}\int_{B(x,t)} \abs{t\nabla_{y,t}e^{-t^2L_w}f_{\rm loc}(y)}^2 \frac{dw(y) \, dt}{tw(B(y,t))}}^{\frac{p_0}{2}} \,dw(x)}^{\frac{1}{p_0}}
\\ \le
\br{\frac{1}{w(B)}\int_{\mathbb{R}^n} \Gcal_{\hh}^{L_w} f_{\rm loc}(x)^{p_0}\,dw(x)}^{\frac{1}{p_0}}
\lesssim 
\br{\frac{1}{w(B)}\int_{\mathbb{R}^n}|f_{\rm loc}(x)|^{p_0} \,dw(x)}^{\frac{1}{p_0}}
\\ \lesssim
\br{\frac{1}{w(8B)}\int_{8B}|f(x)|^{p_0} \,dw(x)}^{\frac{1}{p_0}}
\lesssim 
\mathcal{M}_{p_0}^wf(x_0).
\end{multline*}
As for $f_{\rm glob}$, note first that by Lemma \ref{off-diag-sg}, 
$\{\sqrt{t}\nabla_{y}e^{-tL_w}\}_{t>0},\ \{tL_w e^{-tL_w}\}_{t>0}\in \mathcal{O}(L^{p_0}(w)\rightarrow L^2(w))$. Use this, H\"older's inequality for $2/p_0$ and argue as in the proof of \eqref{Cw-Cw,2:equiv}.
 Then, 
\begin{align*}
\Bigg(\frac{1}{w(B)}\int_{B}&\br{\int_{0}^{r_B}\int_{B(x,t)} \abs{t\nabla_{y,t}e^{-t^2L_w}f_{\rm glob}(y)}^2 \frac{dw(y) \, dt}{tw(B(y,t))}}^{\frac{p_0}{2}} \,dw(x)\bigg)^{\frac{1}{p_0}}
\\ 
&\lesssim 
\sum_{j\ge 2} \br{\int_{0}^{r_B}\dashint_{2B} \abs{t\nabla_{y,t}e^{-t^2L_w}(f\,\chi_{C_j(2B)})(y)}^2\frac{dw(y) \, dt}{t}}^{\frac{1}{2}}
\\ 
&\lesssim 
\sum_{j\geq 2}  \br{\int_{0}^{r_B}\br{\dashint_{2^{j+2}B}|f(y)|^{p_0} \, dw(y)}^{\frac{2}{p_0}}2^{2 \,j \theta_1}
 \Upsilon\br{\frac{2^{j+1}r_B}{t}}^{2\,\theta_2}
e^{-c\frac{4^jr_B^2}{t^2}}\frac{dt}{t}}^{\frac{1}{2}}
\\
&\lesssim 
\mathcal{M}_{p_0}^wf(x_0)\sum_{j\geq 2}2^{j\theta_1} \br{\int_{2^j}^{\infty}s^{2\theta_2}e^{-cs^2} \frac{ds}{s}}^{\frac{1}{2}}
\lesssim \mathcal{M}_{p_0}^wf(x_0).
\end{align*} 
Collecting the estimates obtained for $f_{\rm loc}$ and for $f_{\rm glob}$ we can conclude that
$$
\br{\frac{1}{w(B)}\int_{B}\br{\int_{0}^{r_B}\int_{B(x,t)} \abs{t\nabla_{y,t}e^{-t^2L_w}f(y)}^2 \frac{dw(y) \, dt}{tw(B(y,t))}}^{\frac{p_0}{2}} \, dw(x)}^{\frac{1}{p_0}}
\lesssim 
\mathcal{M}_{p_0}^wf(x_0).
$$
Taking the sup over all balls $B$ such that $x_0\in B$, we get \eqref{Cpo-Mp0-S-heat}.

To complete the proof we need to show the boundedness of $\Gcal_{\hh}^{L_w}$. To this end, take $F(y,t):=t\nabla_{y,t}e^{-t^2L_w}f(y)$, so that $\Gcal_{\hh}^{L_w}f(x)=\mathcal{A}_wF(x)$ and $\widetilde{\mathcal{C}}_{p_0}^{L_w}f(x)=\mathcal{C}_{w,p_0}F(x)$. Thus, \eqref{Cpo-Mp0-S-heat} and Proposition \ref{prop:maximal}, part ($a$),  imply that for every $p_-(L_w)<p_0\le 2$, $0<p<\infty$, $v\in A_{\infty}(w)$
\begin{equation}
 \norm{\Gcal_{\hh}^{L_w}f}_{L^p(vdw)}
 \lesssim \norm{\widetilde{\mathcal{C}}_{p_0}^{L_w}f}_{L^p(vdw)}
 \lesssim \norm{\mathcal{M}_{p_0}^wf}_{L^p(vdw)}, 
\qquad
\forall\,f\in  L^\infty_c(\R^n).
\label{eq:eyhthy}
\end{equation}
Note that the fact that $f\in  L^\infty_c(\R^n)$ guarantees that $t\nabla_{y,t}e^{-t^2L_w}f(y)\in L^2_{\rm loc}(\R_{+}^{n+1},dwdt)$,  since $t\nabla_{y,t}e^{-t^2L_w}$ is bounded on $L^2(w)$ uniformly in $t$. 

Next, fix $v\in A_\infty(w)$ and $p\in \mathcal{W}_v^w(p_-(L_w),\infty)$. In particular, we can find $p_0$ such that $p_-(L_w)<p_0\le\min\{p,2\}$ (close enough to $p_-(L_w)$) such that $v\in A_{\frac{p}{p_0}}(w)$. Therefore, $\mathcal{M}_{p_0}^w$ is bounded on $L^p(vdw)$. This and  \eqref{eq:eyhthy} yield
$$
\norm{\Gcal_{\hh}^{L_w}f}_{L^p(vdw)}
 \lesssim \norm{f}_{L^p(vdw)},
\qquad
 \forall\,f\in  L^\infty_c(\R^n).
$$
The proof is then complete.
\end{proof}


\subsection{Proof of Theorem \ref{thm:SF-Poisson}}\label{section:proof:SP-bounded}


The desired estimates follow very easily from Theorem \ref{theor:control-SF-Poisson} and Theorem \ref{thm:SF-Heat}. To prove $(a)$ we just use Theorem \ref{theor:control-SF-Poisson} part $(b)$ and Theorem \ref{thm:SF-Heat} part $(a)$. To obtain $(b)$, we only need to invoke Theorem \ref{theor:control-SF-Poisson} parts $(a)$, $(c)$, $(d)$ and Theorem \ref{thm:SF-Heat} (note that Theorem \ref{theor:control-SF-Poisson} part $(c)$ and Theorem \ref{thm:SF-Heat} part $(b)$ are used for the case $K=0$). Details are left to the interested reader.

\subsection{Proof of Theorem \ref{theor:control-SF-Heat}}\label{section:proof:control:Heat}
 
We first note that part $(a)$ is trivial. To prove $(b)$ and $(c)$ we fix $0<p<\infty$ and $v\in A_{\infty}(w)$. Pick $r> \max\{\frac{p}{2},\mathfrak{r}_v(w)\}$ so that $v\in A_{r}(w)$ and $0<p < 2r$. If $|x-y|<t$, then $B(x,t)\subset B(y,2t)$ and $B(y,t)\subset B(x,2t)$. Since $w$ is a doubling measure, one has
\begin{align}\label{doublingcondition}
\frac{w(B(x,t))}{w(B(y,t))}\leq\frac{w(B(y,2t))}{w(B(y,t))}\leq C
\quad\textrm{and}\quad
\frac{w(B(y,t))}{w(B(x,2^{j+1}t))}\leq\frac{w(B(x,2t))}{w(B(x,2^{j+1}t))}\leq 1,\textrm{ for all } j\geq 0.
\end{align}

We now prove $(c)$. Let $m\in \N$ and note that
\begin{equation}\label{eq:54wt5t54}
(t^2L_w)^{m+1} e^{-t^2 L_w} 
=
2^{m+1}
A_{\frac{t^2}{2}}
B_{\frac{t^2}{2},m},
\qquad
\text{with}\quad
A_{t}:=t L_w e^{-t L_w} 
\quad \textrm{and} \quad 
B_{t,m}:=\left(t L_w\right)^me^{-t L_w}.
\end{equation}
Using \eqref{doublingcondition}, the fact that  $A_t\in \mathcal{O}(L^{2}(w)- L^2(w))$, and
Proposition \ref{prop:alpha}, which can be applied by the choice of $r$, we obtain
\begin{align}\label{rfwerfer}
&\br{\int_{\mathbb{R}^n}\br{\int_0^{\infty}\int_{B(x,t)}\abs{A_{\frac{t^2}{2}}
B_{\frac{t^2}{2},m}f(y)}^2 \frac{dw(y) \, dt}{tw(B(y,t))}}^{\frac{p}{2}} \ v(x)dw(x)}^{\frac{1}{p}}
\\ 
\nonumber
&\qquad   \lesssim \sum_{j\geq 1} \br{\int_{\mathbb{R}^n}\br{\int_0^{\infty}\dashint_{B(x,t)}\abs{A_{\frac{t^2}{2}}
\br{\br{B_{\frac{t^2}{2},m}f}\chi_{C_j(B(x,t))}}(y)}^2 \frac{dw(y) \, dt}{t}}^{\frac{p}{2}} \ v(x)dw(x)}^{\frac{1}{p}}
\\ 
\nonumber
&\qquad
 \lesssim \sum_{j\geq 1} 2^{j(\theta_1+\theta_2)} e^{-c4^j} \br{\int_{\mathbb{R}^n}\br{\int_0^{\infty}\dashint_{B(x,2^{j+1}t)} \abs{B_{\frac{t^2}{2},m}f(y)}^2 \frac{dw(y) \, dt}{t}}^{\frac{p}{2}} \ v(x)dw(x)}^{\frac{1}{p}}
\\ \nonumber
& \qquad
\lesssim  \sum_{j\geq 1} e^{-c4^j} \br{\int_{\mathbb{R}^n}\br{\int_0^{\infty}\int_{B(x,2^{j+1}\sqrt{2}t)}\abs{
B_{t^2,m}f(y)}^2 \frac{dw(y) \, dt}{tw(B(y,t))}}^{\frac{p}{2}} \ v(x)dw(x)}^{\frac{1}{p}}
\\ \nonumber
&\qquad \lesssim  \sum_{j\geq 1} e^{-c4^j}  \norm{\Scal_{m,\hh}^{L_w}f}_{L^p(vdw)}
\\ \nonumber
&\qquad\lesssim  \norm{\Scal_{m,\hh}^{L_w}f}_{L^p(vdw)},
\end{align}
Note that in the fourth estimate we have changed the variable $t$ into $\sqrt{2}t$ and used that $w(B(y,t))\le w(B(x,2^{j+3}t))\lesssim w(B(x,2^{j+1}\sqrt{2}t))$ whenever $y\in B(x,2^{j+1}\sqrt{2}t)$. Collecting \eqref{eq:54wt5t54} and \eqref{rfwerfer} we conclude as desired $(c)$.

We finally prove part $(b)$. Take $m\in \N$ and note that elementary computations show 
$$
\Gcal_{m,\hh}^{L_w}f(y)
\le
G_{m,\hh}^{L_w}f(y)
+
2m\Scal_{m,\hh}^{L_w}f(y)
+
2 \Scal_{m+1,\hh}^{L_w}f(y),
\qquad y\in\mathbb{R}^n.
$$
Thus $(b)$ will follow from $(c)$ once we control the term involving $G_{m,\hh}^{L_w}f$. To that end we proceed as before and observe that
$$
t \nabla_{y} (t^2L_w)^m e^{-t^2 L_w} 
=
2^{m+\frac12}
A_{\frac{t^2}{2}}
B_{\frac{t^2}{2},m},
\qquad
\text{with}\quad
A_{t}:=\sqrt{t} \nabla_{y} e^{-t L_w} 
\quad \text{and} \quad B_{t,m}:=\left(t L_w\right)^me^{-t L_w}.
$$
We can now repeat the computations in \eqref{rfwerfer}, since again $A_t\in \mathcal{O}(L^2(w)-L^2(w))$, to conclude that
$$
\norm{G_{m,\hh}^{L_w}f}_{L^p(vdw)}
\lesssim
\norm{\Scal_{m,\hh}^{L_w}f}_{L^p(vdw)}.
$$
This, together with the previous considerations, allows us to complete the proof of $(b)$ and thus that of Theorem \ref{theor:control-SF-Heat}. 


\subsection{Proof of Theorem \ref{theor:control-SF-Poisson}}\label{section:proof:control:Poisson}

Recall that $w\in A_2$ has been fixed already. If $w\in A_1$ we set $\widehat r:=1$. Otherwise, let  $\widehat{r}$ be so that $r_w<\widehat r<2$ (eventually $\widehat{r}$ will be chosen very close to $r_w$) .   Note that in any scenario we have $w\in A_{\widehat r}$.


The proof of part $(a)$ is trivial. We start proving part $(b)$. We need to show that for every $K\in \N$ 
\begin{align}\label{SKP}
\norm{\Scal_{K,\pp}^{L_w}f}_{L^p(vdw)}\lesssim \norm{\Scal_{K,\hh}^{L_w}f}_{L^p(vdw)},
\end{align}
for all $v\in A_{\infty}(w)$ and $p\in \mathcal{W}_{v}^w(0,\br{p_+(L_w)}_w^{K,*})$, that is, for all $0<p<\br{p_+(L_w)}_w^{K,*}$ and 
$v\in RH_{(\br{p_+(L_w)}_w^{K,*}/p)'}(w)$. By Theorem \ref{theor:extrapol} parts $(b)$ and $(c)$ one can see that it suffices to prove such estimate for some fixed $p$ in the same range and all $v$ in the corresponding reverse Hölder class. In particular, as $\br{p_+(L_w)}_w^{K,*}>2$, we can take $p=2$
and hence we need to obtain that 
\begin{align}\label{SKP*}
\norm{\Scal_{K,\pp}^{L_w}f}_{L^2(vdw)}\lesssim \norm{\Scal_{K,\hh}^{L_w}f}_{L^2(vdw)},
\qquad
\forall\,v\in RH_{(\br{p_+(L_w)}_w^{K,*}/2)'}(w).
\end{align}

Fix then  $v\in RH_{(\br{p_+(L_w)}_w^{K,*}/2)'}(w)$ (notice that if $\br{p_+(L_w)}_w^{K,*}=\infty$ the condition on the weight $v$ becomes $v\in A_\infty(w)$)  and set
$$
B_{t,K}:=\left(t^2L_w\right)^{K}
e^{-t^2L_w},
$$
and recall the subordination formula \eqref{formula:SF}.
This and Minkowski's inequality imply
\begin{align*}
&\norm{\Scal_{K,\pp}^{L_w}f}_{L^2(vdw)}
 \lesssim 
\br{\int_{\mathbb{R}^n}\iint_{\Gamma(x)}\abs{(t^2L_w)^{K}\int_0^{\infty} e^{-u}u^{\frac{1}{2}}
e^{-\frac{t^2}{4u}L_w}f(y) \ \frac{du}{u}}^2  \frac{dw(y) \, dt}{tw(B(y,t))} \ v(x)dw(x)}^{\frac{1}{2}}
\\
&\qquad   \lesssim 
\int_0^{\infty}e^{-u}u^{\frac{1}{2}}\br{\int_{\mathbb{R}^n} \int_0^{\infty}\int_{B(x,t)}\abs{(t^2L_w)^{K}e^{-\frac{t^2}{4u}L_w}f(y)}^2 \ \frac{dw(y) \, dt}{tw(B(y,t))} v(x)dw(x)}^{\frac{1}{2}}\ \frac{du}{u}
\\
&\qquad \lesssim
\int_0^{\frac{1}{4}}e^{-u}u^{K+\frac{1}{2}}\br{\int_{\mathbb{R}^n} \int_0^{\infty}\int_{B(x,t)}\abs{B_{\frac{t}{2\sqrt{u}},K}f(y)}^2 \ \frac{dw(y) \, dt}{tw(B(y,t))}  \ v(x)dw(x)}^{\frac{1}{2}} \ \frac{du}{u}
\\
&\qquad\qquad +
\int_{\frac{1}{4}}^{\infty}e^{-u}u^{K+\frac{1}{2}}\br{\int_{\mathbb{R}^n} \int_0^{\infty}\int_{B(x,t)}\abs{B_{\frac{t}{2\sqrt{u}},K}f(y)}^2 \ \frac{dw(y) \, dt}{tw(B(y,t))}  \ v(x)dw(x)}^{\frac{1}{2}} \ \frac{du}{u}
\\
&\qquad =:I+II.
\end{align*}

To estimate $II$ we let $F(y,t):=B_{t,K}f(y)$ and pick $\widetilde{r}> \mathfrak{r}_v(w)\ge 1$ so that $v\in A_{\widetilde{r}}(w)$. Hence, changing the variable $t$ into $2\sqrt{u}t$, applying the fact that $w(B(y,t))\leq w(B(y,2\sqrt{u}t))$ when $u>1/4$, and Proposition \ref{prop:alpha}, we have
\begin{align*}
II
\lesssim
\int_{\frac{1}{4}}^{\infty}e^{-u}u^{K+\frac{1}{2}} \norm{\mathcal{A}_w^{2\sqrt{u}}F}_{L^2(vdw)} \frac{du}{u}
\lesssim
\int_{\frac{1}{4}}^{\infty}e^{-u}u^{K+\frac{1}{2}+\frac{n\,\widehat{r}\,\,\,\widetilde{r}}{4}} \norm{\mathcal{A}_wF}_{L^2(vdw)} \frac{du}{u}
\lesssim 
\norm{\Scal_{K,\hh}^{L_w}f}_{L^2(vdw)}.
\end{align*}

In order to estimate $I$ we start by distinguishing two cases. If $nr_w>(2K+1)p_+( L_w)$, the condition $v\in RH_{(\br{p_+(L_w)}_w^{K,*}/2)'}(w)$
implies that $0<\mathfrak{s}_v(w)<\frac{p_+(L_w)nr_w}{2\,(n r_w-(2K+1)p_+(L_w))}$. Therefore,
it is possible to pick $\varepsilon_1>0$ small enough, $\widehat r \in (r_w,2)$ close enough to $r_w$ (\,$\widehat r=1$ if $w\in A_1$) and $2<\widetilde{q}<p_+(L_w)$ so that
$$
 0<\mathfrak{s}_v(w)<\frac{\widetilde{q}\,n\,\widehat{r}}{2(1+\varepsilon_1)(n\,\widehat{r}-(2K+1)\,\widetilde{q}\,)}.
$$
Besides, there also exists $\varepsilon_2>0$ so that
$$
\widetilde{q}<\frac{\widetilde{q}\,n\,\widehat{r}}{(1+\varepsilon_2)(n\,\widehat{r}-(2K+1)\,\widetilde{q}\,)}.
$$
Take $\varepsilon_0:=\min\{\varepsilon_1,\varepsilon_2\}$ and $s:=\frac{\widetilde{q}\,n\widehat{r}}{2(1+\varepsilon_0)(n\widehat{r}-(2K+1)\,\widetilde{q}\,)}$. Then our choices guarantee that $2<\widetilde{q}<p_+(L_w)$, $\frac{\widetilde{q}}{2}\leq s<\infty$, $1\leq \mathfrak{s}_{v}(w)<s<\infty$, and hence $v\in RH_{s'}(w)$. Also,
\begin{align}\label{adewfdqwrf}
K+\frac{1}{2}+\frac{\widehat{r}\,n}{4s}-\frac{\widehat{r}\,n}{2\,\widetilde{q}}
= \varepsilon_0\br{\frac{\widehat{r}\,n}{2\,\widetilde{q}}-K-\frac{1}{2}}
> \varepsilon_0 \br{\frac{r_w n}{2p_+(L_w)}-K-\frac{1}{2}}>0.
\end{align}

In the other case, $nr_w \leq (2K+1)p_+(L_w)$ and then $\br{p_+(L_w)}_w^{K,*}=\infty$. Recall  then that our assumption on $v$ is simply  $v\in A_\infty(w)$. Fix now $s>\mathfrak{s}_v(w) $ so that $v\in RH_{s'}(w)$. If $w\notin A_1$ we  pick $\widehat r \in (r_w,2)$ (close enough to $r_w$) in such a  way that $1-\frac{r_w}{\widehat r}<\frac{p_+(L_w)}{2s}$, and if $w\in A_1$ we just take $\widehat r=1$. Let $\widetilde q$ satisfy $\max\left\{2,\frac{2sp_+(L_w)}{p_+(L_w)+2s\frac{r_w}{\widehat r}}\right\}<\widetilde{q}<\min
\left\{p_+(L_w),2s\right\}$ with the understanding that $\widetilde{q}=2s$ if $p_+(L_w)=\infty$. All these choices guarantee that  
$2<\widetilde{q}<p_+(L_w)$, $\frac{\widetilde{q}}{2}\leq s<\infty$, $1\leq \mathfrak{s}_v(w)<s<\infty$, and therefore $v\in RH_{s'}(w)$. Moreover, 
 note that from  the lower bound for $\widetilde q$ involving $s$, we have that if $p_+(L_w)<\infty$
\begin{align*}
K+\frac{1}{2}+\frac{\widehat{r}\,n}{4s}-\frac{\widehat{r}\,n}{2\widetilde{q}}>K+\frac{1}{2}-\frac{r_w\,n}{2p_+(L_w)}\geq 0.
\end{align*}
Additionally, if $p_+(L_w)=\infty$, then $K+\frac{1}{2}+\frac{\widehat{r}\,n}{4s}-\frac{\widehat{r}\,n}{2\,\widetilde{q}}=K+\frac{1}{2}>0$.

Putting all the possible cases together we have been able to find $\widetilde{q}$ and $s$ such that $2<\widetilde{q}<p_+(L_w)$, $\frac{\widetilde{q}}{2}\leq s<\infty$, $v\in RH_{s'}(w)$, and
\begin{align}\label{adewfdqwrf*}
K+\frac{1}{2}+\frac{\widehat{r}\,n}{4s}-\frac{\widehat{r}\,n}{2\,\widetilde{q}}
>0.
\end{align}
We can now proceed to estimate $I$. We first apply  H\"older's inequality and \eqref{doublingcondition}
\begin{align}\label{estI:1}
I
&\lesssim 
\int_0^{\frac{1}{4}}u^{K+\frac{1}{2}} \ \br{\int_{\mathbb{R}^n}
\int_0^{\infty}\br{\int_{B(x,t)}\abs{B_{\frac{t}{2\sqrt{u}},K}f(y)}^{\widetilde{q}}\frac{dw(y)}{w(B(y,t))} }^{\frac{2}{\widetilde{q}}} \ \frac{dt}{t}  \ v(x)dw(x)}^{\frac{1}{2}} \frac{du}{u}
\\ \nonumber
&=:
\int_0^{\frac{1}{4}}u^{K+\frac{1}{2}} \ \left(\int_{\mathbb{R}^n} \mathcal{J}(u,x)^2\,v(x)dw(x)\right)^{\frac{1}{2}} \frac{du}{u}.
\end{align}
Besides, note that since $1<\frac{\widetilde{q}}{2}\le s<\infty$, then for $\alpha:=2\sqrt{u}\in (0,1]$ and $q:=\frac{\widetilde{q}}{2}$,
we can apply Proposition \ref{prop:Q} to conclude that 
\begin{align}\label{KKS}
&\int_{\mathbb{R}^n}\mathcal{J}(u,x)^2 v(x) dw(x)
\\ \nonumber
&\qquad=
\int_0^{\infty}\int_{\mathbb{R}^n}\br{\int_{B(x,2\sqrt{u}\frac{t}{2\sqrt{u}})} \abs{B_{\frac{t}{2\sqrt{u}},K}f(y)}^{\widetilde{q}} \frac{dw(y)}{w\br{B\br{y,2\sqrt{u}\frac{t}{2\sqrt{u}}}}}}^{\frac{2}{\widetilde{q}}} \ v(x)dw(x)\frac{dt}{t}
\\  \nonumber
&\qquad\lesssim 
u^{\frac{\widehat{r}\,n}{2s}-\frac{\widehat{r}\,n}{\widetilde{q}}}
\int_{\mathbb{R}^n}\int_0^{\infty}\left(\int_{B(x,\frac{t}{2\sqrt{u}})}\left|B_{\frac{t}{2\sqrt{u}},K}f(y)\right|^{\widetilde{q}} \frac{dw(y)}{w\br{B\br{y,\frac{t}{2\sqrt{u}}}}} \right)^{\frac{2}{\widetilde{q}}}
\frac{dt}{t} \ v(x)dw(x)
\\  \nonumber
&\qquad\lesssim 
u^{\frac{\widehat{r}\,n}{2s}-\frac{\widehat{r}\,n}{\widetilde{q}}}
\int_{\mathbb{R}^n} \int_0^{\infty} \left(\int_{B(x,t)}\left|B_{t,K}f(y)\right|^{\widetilde{q}}\frac{w(y) dy}{w(B(y,t))} \right)^{\frac{2}{\widetilde{q}}} \frac{dt}{t}v(x)dw(x)
\\ \nonumber
&\qquad=: 
u^{\frac{\widehat{r}\,n}{2s}-\frac{\widehat{r}\,n}{\widetilde{q}}} \int_{\mathbb{R}^n} \mathcal{T}(x)^2\ v(x)dw(x),
\end{align}
where in the last inequality we have changed the variable $t$ into $2\sqrt{u}t$. 
By Lemma \ref{off-diag-sg}, $e^{-t L_w}\in \mathcal{O}(L^2(w)-L^{\widetilde{q}}(w))$. Applying this,  \eqref{doublingcondition} and Proposition \ref{prop:alpha}, we get
\begin{align*}
&\br{\int_{\mathbb{R}^n} \mathcal{T}(x)^2\ v(x)dw(x)}^{\frac12}
\lesssim
\br{\int_{\mathbb{R}^n} \int_0^{\infty} \left(\dashint_{B(x,t)}\left|e^{-\frac{t^2}{2}L_w}B_{\frac{t}{\sqrt{2}},K}f(y)\right|^{\widetilde{q}} dw(y) \right)^{\frac{2}{\widetilde{q}}} \frac{dt}{t}v(x)dw(x)}^{\frac12}\nonumber
\\
&\qquad\lesssim 
\sum_{j\geq 1} 2^{j(\theta_1+\theta_2)} e^{-c4^j}
\br{\int_{\mathbb{R}^n} \int_0^{\infty} \dashint_{B(x,2^{j+1}t)}\left|B_{\frac{t}{\sqrt{2}},K}f(y)\right|^2 \frac{dw(y) \, dt}{t}v(x)dw(x)}^{\frac12}
\\
&\qquad\lesssim \sum_{j\geq 1} e^{-c4^j}
\br{\int_{\mathbb{R}^n} \int_0^{\infty} \int_{B(x,2^{j+1}\sqrt{2}t)}\left|B_{t,K}f(x)\right|^2 \frac{dw(y) \, dt}{tw(B(y,t))}v(x)dw(x)}^{\frac12}
\\
&\qquad\nonumber
\lesssim
\norm{\Scal_{K,\hh}^{L_w}f}_{L^2(vdw)}.
\end{align*}
 Notice that in the third estimate we have changed the variable $t$ into $\sqrt{2}t$ and used that $w(B(y,t))\le w(B(x,2^{j+3}t))\lesssim w(B(x,2^{j+1}\sqrt{2}t))$. 
 This, \eqref{estI:1}, and \eqref{KKS} yield, 
\begin{align*}
I
\lesssim
\norm{\Scal_{K,\hh}^{L_w}f}_{L^2(vdw)}
\int_0^{\frac{1}{4}}u^{K+\frac{1}{2}+\frac{\widehat{r}\,n}{4s}-\frac{\widehat{r}\,n}{2\widetilde{q}}}
\frac{du}{u}
\lesssim
\norm{\Scal_{K,\hh}^{L_w}f}_{L^2(vdw)},
\end{align*}
where in the last inequality we have used \eqref{adewfdqwrf*}.
This completes the proof of part $(b)$.

Let us continue  by showing  parts $(c)$ and $(d)$. We need the following auxiliary result in the spirit of \cite[Lemma $3.5$]{AuscherHofmannMartell}  whose proof is given below:
\begin{lemma}\label{lemma:caccio}
For every $K\in \N_0$, $f\in L^2(w)$ and almost every $x\in \mathbb{R}^n$, there holds
\begin{align}\label{GKP}
\Gcal_{K,\pp}^{L_w}f(x)
&\lesssim 
K\left(\iint_{|x-y|<2t} \abs{(t^2L_w)^K e^{-t^2L_w}f(y)}^2 \frac{dw(y) \, dt}{tw(B(y,t))}\right)^{\frac{1}{2}}
\\
  \nonumber
&\qquad\quad + \left(\iint_{|x-y|<2t}\abs{t\nabla_{y,t}(t^2L_w)^K e^{-t^2L_w}f(y)}^2 \frac{dw(y) \, dt}{tw(B(y,t))}\right)^{\frac{1}{2}}
 \\
 \nonumber
 &\qquad\quad+\left(\iint_{|x-y|<2t} \abs{(t^2L_w)^K(e^{-t\sqrt{L_w}}-e^{-t^2L_w})f(y)}^2 \frac{dw(y) \, dt}{tw(B(y,t))}\right)^{\frac{1}{2}}.
\end{align}
Moreover, setting 
\begin{align*}
\mathfrak{G}_{K,\pp}^{L_w}f(x)
:=
\br{\iint_{|x-y|<2t}\abs{\br{t^2L_w}^K \br{e^{-t\sqrt{L_w}}-e^{-t^2L_w}}f(y)}^2 \frac{dw(y) \, dt}{tw(B(y,t))}}^{\frac{1}{2}},
\qquad K\in\N_0,
\end{align*}
the following estimate holds:
\begin{align}\label{est:gfrak}
\norm{\mathfrak{G}_{K,\pp}^{L_w}f}_{L^p(vdw)}\lesssim \norm{\Scal_{K+1,\hh}^{L_w}f}_{L^p(vdw)},
\end{align}
for all $K\in \N_0$, $v\in A_{\infty}(w)$, and $p\in \mathcal{W}_v^w(0,\br{p_+(L_w)}_w^{K,*})$. 

\end{lemma}

Assuming this lemma momentarily and  applying Proposition \ref{prop:alpha} to the first two terms in the right-hand side of \eqref{GKP},
we conclude, for all $K\in \N_0$, $v\in A_{\infty}(w)$, and $p\in \mathcal{W}_{v}^w(0,\br{p_+(L_w)}_w^{K,*})$,
\begin{align}\label{2GKP}
\norm{\Gcal_{K,\pp}^{L_w}f}_{L^p(vdw)}
\lesssim 
K\norm{\Scal_{K,\hh}^{L_w}f}_{L^p(vdw)}+\norm{\Gcal_{K,\hh}^{L_w}f}_{L^p(vdw)}+\norm{\Scal_{K+1,\hh}^{L_w}f}_{L^p(vdw)}.
\end{align}
For $K\in \N$, we just apply parts $(b)$ and $(c)$ of Theorem \ref{theor:control-SF-Heat} and $(d)$ follows at once. 
For $K=0$, we use part $(a)$ of Theorem \ref{theor:control-SF-Heat} to obtain
$$
\norm{\Gcal_{\pp}^{L_w}f}_{L^p(vdw)}
\lesssim \norm{\Gcal_{\hh}^{L_w}f}_{L^p(vdw)}+\norm{\Scal_{\hh}^{L_w}f}_{L^p(vdw)}
\lesssim \norm{\Gcal_{\hh}^{L_w}f}_{L^p(vdw)},
$$
and this shows part $(c)$.

\begin{proof}[Proof of Lemma \ref{lemma:caccio}]For a fixed $K\in \N_0$,
we start proving \eqref{est:gfrak}. Much as before, it suffices to obtain \eqref{est:gfrak} for $p=2$ and for every 
$v\in RH_{(\br{p_+(L_w)}_w^{K,*}/2)'}(w)$. Fixing such a weight, by the subordination formula \eqref{formula:SF} and Minkowski's inequality, we get
\begin{multline*}
\norm{\mathfrak{G}_{K,\pp}^{L_w}f}_{L^2(vdw)}
\\ \lesssim 
\int_0^{\infty}e^{-u}u^{\frac{1}{2}}\left(\int_{\mathbb{R}^n} \int_0^{\infty}\!\!\int_{B(x,2t)} \abs{(t^2L_w)^K \br{e^{-\frac{t^2}{4u}L_w} -e^{-t^2L_w}}f(y)}^2 \frac{dw(y) \, dt}{tw(B(y,t))}v(x)dw(x)\right)^{\frac{1}{2}}\frac{du}{u}
\\ =: 
\int_0^{\infty}e^{-u}u^{\frac{1}{2}} F(u)\frac{du}{u}
\le \int_0^{\frac{1}{4}}u^{\frac{1}{2}}F(u)\frac{du}{u} +\int_{\frac{1}{4}}^{\infty}e^{-u}u^{\frac{1}{2}}F(u)\frac{du}{u}
=:I+II.
\end{multline*}

We start dealing with $I$. We proceed as in the proof of the corresponding estimate of $I$ for $\Scal_{K,\pp}^{L_w}$. Recall that after considering some cases we ended up finding $\widetilde{q}$ and $s$ such that $2<\widetilde{q}<p_+(L_w)$, $\frac{\widetilde{q}}{2}\leq s<\infty$, $v\in RH_{s'}(w)$, and
\begin{align}\label{adewfdqwrf***}
\theta:=K+\frac{1}{2}+\frac{\widehat{r}\,n}{4s}-\frac{\widehat{r}\,n}{2\,\widetilde{q}}
>0.
\end{align}
For later use choose $\widetilde{\theta}$ so that $0<\widetilde{\theta}<\min\{4\,\theta,1\}$. Then, for every $0<a<1$
\begin{equation}
\int_{a}^{1}t^{4\,\theta-1}\frac{dt}{t}
\leq
\int_{a}^{1}t^{\widetilde{\theta}-1}\frac{dt}{t}
 \le \frac{1}{1-\widetilde \theta}\,a^{\widetilde{\theta}-1}.
\label{eq:THETA-a}
\end{equation}

Fix now $0<u<\frac{1}{4}$, and note that
\begin{align*}
\abs{\br{e^{-\frac{t^2}{4u}L_w} -e^{-t^2L_w}}f} 
=
\abs{\int_{t}^{\frac{t}{2\sqrt{u}}} \partial_r e^{-r^2L_w}f dr}
\lesssim
 \int_{t}^{\frac{t}{2\sqrt{u}}}\big|r^2L_we^{-r^2L_w}f\big|\frac{dr}{r}.
\end{align*}
We set $H_K(y,r):=(r^2L_w)^{K+1}  e^{-r^2L_w}f(y)$. Using the above estimate and applying Minkowski's and H\"older's inequalities, it follows that
\begin{align*}
F(u)&
\lesssim 
\br{\int_{\mathbb{R}^n} \int_0^{\infty}\br{\int_{t}^{\frac{t}{2\sqrt{u}}}\br{\int_{B(x,2t)} \abs{(t^2L_w)^K
 r^2L_w e^{-r^2L_w}f(y)}^{2}\frac{dw(y)}{w(B(y,t))}}^{\frac{1}{2}} \frac{dr}{r}}^2\frac{dt}{t}v(x)dw(x)}^{\frac{1}{2}}
\\ &\lesssim  
u^{-\frac{1}{4}}\br{\int_{\mathbb{R}^n} \int_0^{\infty}\int_{t}^{\frac{t}{2\sqrt{u}}}\int_{B(x,2t)}
|H_K(y,r)|^2 \br{\frac{t}{r}}^{4\,K}\frac{dw(y)}{w(B(y,t))} \ \frac{dr}{r^2}dt\, v(x)dw(x)}^{\frac{1}{2}} 
\\ &\lesssim 
u^{-\frac{1}{4}}\br{\int_{\mathbb{R}^n} \int_0^{\infty}\int_{2\sqrt{u}r}^{r}\int_{B(x,2t)}
|H_K(y,r)|^2 \br{\frac{t}{r}}^{4\,K}\frac{dw(y)}{w(B(y,t))} dt\frac{dr}{r^2}v(x)dw(x)}^{\frac{1}{2}}.
\end{align*}
By \eqref{doublingcondition}, applying H\"older's inequality to the integral in $y$, and changing the variable $t$ into $rt$, we obtain
\begin{align}\label{hatH}
F(u) & \lesssim 
u^{-\frac{1}{4}} \br{\int_{\mathbb{R}^n} \int_0^{\infty}\int_{2\sqrt{u}r}^r \br{\dashint_{B(x,2t)} | H_K(y,r)|^{\widetilde{q}} dw(y)}^{\frac{2}{\widetilde{q}}}  \br{\frac{t}{r}}^{4\,K} \frac{dt\,dr}{r^{2}} v(x)dw(x)}^{\frac{1}{2}}
\\ \nonumber
& \lesssim 
u^{-\frac{1}{4}} \br{\int_{\mathbb{R}^n} \int_0^{\infty}\int_{2\sqrt{u}r}^r \br{\int_{B(x,2t)} |H_K(y,r)|^{\widetilde{q}} \frac{dw(y)}{w(B(y,t))}}^{\frac{2}{\widetilde{q}}} \br{\frac{t}{r}}^{4\,K} \frac{dt\,dr}{r^{2}} v(x)dw(x)}^{\frac{1}{2}}
\\ \nonumber
&\lesssim 
u^{-\frac{1}{4}} \br{\int_{\mathbb{R}^n} \int_0^{\infty}\int_{2\sqrt{u}}^{1} \br{\int_{B(x,2rt)}|H_K(y,r)|^{\widetilde{q}}\frac{dw(y)}{w(B(y,rt))}}^{\frac{2}{\widetilde{q}}} t^{4K} \frac{dt\, dr}{r} v(x)dw(x)}^{\frac{1}{2}}
\\ \nonumber
&=:
u^{-\frac{1}{4}} \br{\int_{\mathbb{R}^n} \widehat{H}(x,u)^2v(x)dw(x)}^{\frac{1}{2}}.
\end{align}
Note that $1<\frac{\widetilde{q}}{2} \leq s<\infty$ and recall that $w\in A_{\widehat r}$, with $\widehat r$ fixed before. Then, for  $\alpha:=t\in (0,1)$ and $q:=\frac{\widetilde{q}}{2}$, we can apply Proposition \ref{prop:Q} and \eqref{pesosineqw:Ap} to obtain
\begin{align}\label{HG}
&\int_{\mathbb{R}^n}\widehat{H}(x,u)^2 v(x)dw(x) 
\lesssim
\int_{2\sqrt{u}}^1\int_{0}^{\infty} \int_{\mathbb{R}^n}\br{\int_{B(x,2rt)}|H_K(y,r)|^{\widetilde{q}} \frac{dw(y)}{w(B(y,2rt))}}^{\frac{2}{\widetilde{q}}}v(x)dw(x)\frac{dr}{r}t^{4K} dt
\\ \nonumber 
&\qquad
\lesssim 
\left(\int_{2\sqrt{u}}^1t^{4K+\frac{\widehat{r}\,n}{s}-\frac{2\,\widehat{r}\,n}{\widetilde{q}}+1}\frac{dt}{t}\right) \int_{0}^{\infty}\int_{\mathbb{R}^n}\br{\int_{B(x,2r)}|H_K(y,r)|^{\widetilde{q}} \frac{dw(y)}{w(B(y,r))}}^{\frac{2}{\widetilde{q}}}v(x)dw(x)\frac{dr}{r}
\\ 
\nonumber 
&
\lesssim
u^{\frac{\widetilde{\theta}-1}{2}}
\int
_{\mathbb{R}^n}\widetilde{H}_K(x)^2 v(x) dw(x),
\end{align}
where we have used \eqref{eq:THETA-a} and where 
$$
\widetilde{H}_K(x):= 
\br{\int_0^{\infty} \br{\int_{B(x,2r)}|H_K(y,r)|^{\widetilde{q}}\frac{dw(y)}{w(B(y,r))}}^{\frac{2}{\widetilde{q}}} \frac{dr}{r}}^{\frac{1}{2}}.
$$
Using that $e^{-t L_w} \in\mathcal{O}(L^2(w)\rightarrow L^{\widetilde{q}}(w))$ by Lemma \ref{off-diag-sg}, and since $H_K(y,r)=2^{K+1}e^{-\frac{r^2}2 L_w}\,H_K\big(y,\tfrac{r}{\sqrt{2}}\big)$, it follows from \eqref{doublingcondition} and Proposition \ref{prop:alpha} that
\begin{align*}
\br{\int_{\mathbb{R}^n}\widetilde{H}_K(x)^2v(x) dw(x)}^{\frac{1}{2}}
&\lesssim 
\sum_{j\geq 1} e^{-c4^j} \br{\int_{\mathbb{R}^n}\int_0^{\infty}\dashint_{B(x,2^{j+2}r)}\big|H_K\big(y,\tfrac{r}{\sqrt{2}}\big)\big|^2
\frac{dw(y) \, dr}{r}v(x)dw(x)}^{\frac{1}{2}}
\\ &\lesssim 
\sum_{j\geq 1} e^{-c4^j} \br{\int_{\mathbb{R}^n}\int_0^{\infty}\int_{B(x,2^{j+2}\sqrt{2}r)}\big|H_K(y,r)\big|^2
\frac{dw(y) \, dr}{rw(B(y,r))}v(x)dw(x)}^{\frac{1}{2}}
\\ &\lesssim 
\sum_{j\geq 1} e^{-c4^j} \br{\int_{\mathbb{R}^n}\Scal_{K+1,\hh}^{L_w}f(x)^2v(x)dw(x)}^{\frac{1}{2}}
\\ &\lesssim 
\norm{\Scal_{K+1,\hh}^{L_w}f}_{L^2(vdw)}.
\end{align*}
 Note that in the second estimate we have changed the variable $r$ into $\sqrt{2}r$ and used that $w(B(y,r))\le w(B(x,2^{j+4}r))\lesssim w(B(x,2^{j+2}\sqrt{2}r))$.
 This, \eqref{hatH}, and \eqref{HG} give
\begin{align*}
F(u)
\lesssim u^{\frac{\widetilde{\theta}-2}{4}} \,\norm{\Scal_{K+1,\hh}^{L_w}f}_{L^2(vdw)},
\end{align*}
which in turn yields
\begin{align*}
I=
\int_0^{\frac{1}{4}}u^{\frac{1}{2}}F(u)\frac{du}{u}
\lesssim
\left(\int_0^{\frac{1}{4}}u^{\frac{\widetilde{\theta}}{4}}\frac{du}{u}\right)\,\norm{\Scal_{K+1,\hh}^{L_w}f}_{L^2(vdw)}
\lesssim
\,\norm{\Scal_{K+1,\hh}^{L_w}f}_{L^2(vdw)},
\end{align*}
since $\widetilde{\theta}>0$.

\medskip
To estimate $II$, we fix $\frac{1}{4}< u<\infty$ and observe that
\begin{align*}
\abs{(e^{-\frac{t^2}{4u}L_w} -e^{-t^2L_w})f}
= \abs{\int_{\frac{t}{2\sqrt{u}}}^{t} \partial_r e^{-r^2L_w}f dr}
\lesssim \int_{\frac{t}{2\sqrt{u}}}^{t}\abs{r^2L_we^{-r^2L_w}f} \frac{dr}{r}.
\end{align*}
Set $T_{r^2,K}:=(r^2L_w)^{K+1}e^{-r^2L_w}$ and pick $\widetilde{r}> \mathfrak{r}_v(w)\ge 1$ so that $v\in A_{\widetilde{r}}(w)$. Then, applying Minkowski's integral inequality, H\"older's inequality,  Fubini's Theorem, 
and Proposition \ref{prop:alpha}, we have 
\begin{align*}
F(u) &\lesssim
\br{\int_{\mathbb{R}^n} \int_{0}^{\infty}\br{\int_{\frac{t}{2\sqrt{u}}}^{t}\br{\int_{B(x,2t)} \abs{(t^2L_w)^KT_{r^2,0}
f(y)}^2 \frac{dw(y)}{w(B(y,t))}}^{\frac{1}{2}} \frac{dr}{r}}^2\frac{dt}{t}v(x)dw(x)}^{\frac{1}{2}}
\\ &\lesssim
\br{\int_{\mathbb{R}^n} \int_{0}^{\infty} \int_{\frac{t}{2\sqrt{u}}}^{t}\int_{B(x,2t)} \abs{(t^2L_w)^KT_{r^2,0} f(y)}^2 \frac{dw(y)}{w(B(y,t))} \frac{dr}{r^2}dt\,v(x)dw(x)}^{\frac{1}{2}}
\\ &=
\br{\int_{\mathbb{R}^n} \int_{0}^{\infty} \int_r^{2\sqrt{u}r}\int_{B(x,2t)} \abs{(t^2L_w)^KT_{r^2,0}f(y)}^2 \frac{dw(y)\,dt}{w(B(y,t))}\frac{dr}{r^2}v(x)dw(x)}^{\frac{1}{2}}
\\ & \lesssim 
u^K \br{\int_{\mathbb{R}^n} \int_{0}^{\infty} \int_r^{2\sqrt{u}r}\int_{B(x,4\sqrt{u}r)} \abs{T_{r^2,K}f(y)}^2 \frac{dw(y)}{w(B(y,r))} \, dt\frac{dr}{r^2}v(x)dw(x)}^{\frac{1}{2}}
\\ &\lesssim
u^{K+\frac{1}{4}}\br{\int_{\mathbb{R}^n} \int_{0}^{\infty}\int_{B(x,4\sqrt{u}r)}\abs{T_{r^2,K}f(y)}^2 \frac{dw(y)\,dr}{rw(B(y,r))}v(x)dw(x)}^{\frac{1}{2}}
\\ &\lesssim 
u^{K+\frac{1}{4}+\frac{n\,\widehat{r}\,\,\,\widetilde{r}}{4}}\br{\int_{\mathbb{R}^n} \int_{0}^{\infty}\int_{B(x,r)} \abs{T_{r^2,K}f(y)}^2 \frac{dw(y)\,dr}{rw(B(y,r))}v(x)dw(x)}^{\frac{1}{2}}
\\ &=
u^{K+\frac{1}{4}+\frac{n\,\widehat{r}\,\,\,\widetilde{r}}{4}}\,\norm{\Scal_{K+1,\hh}^{L_w}f}_{L^2(vdw)}.
\end{align*}
Hence,
$$
II
=
\int^{\infty}_{{\frac{1}{4}}} e^{-u}u^{\frac{1}{2}}F(u)\frac{du}{u}
\lesssim
\left(\int_{\frac14}^{\infty} e^{-u} u^{K+\frac{3}{4}+\frac{n\,\widehat{r}\,\,\,\widetilde{r}}{4}} \frac{du}{u}\right)\,\norm{\Scal_{K+1,\hh}^{L_w}f}_{L^p(vdw)}
\lesssim
\norm{\Scal_{K+1,\hh}^{L_w}f}_{L^p(vdw)}.
$$
This completes the proof of \eqref{est:gfrak}.

We finally show \eqref{GKP}. The proof of this inequality follows the lines of that of \cite[Lemma 3.5]{AuscherHofmannMartell}.  
If $K=0$ take $f\in L^2(w)$, and $f_0:=f$. If $K\geq 1$, we assume that $f$ is in the domain of $L_w^K$  (we explain at the end of the proof how to pass to general functions in $L^2(w)$),  and define $f_K:=L_w^Kf$. Besides, consider
$u_K:=L_w^Ke^{-t\sqrt{L_w}}f=e^{-t\sqrt{L_w}}f_K$, and $v_K:=L_w^Ke^{-t^2L_w}f=e^{-t^2L_w}f_K$. 
Notice that
$$
t\nabla_{y,t}(t^{2K}u_{K})
=
2K t^{2K}v_{K}  \vec{e} +2K t^{2K}(u_{K}-v_{K}) \vec{e}
+
t^{2K}(t\nabla_{y,t}u_{K})
$$
with $\vec{e}= (0, \ldots, 0, 1)$. The first and second terms give the first and third terms on the right hand side of 
 \eqref{GKP}, respectively. Then we need to control the third term which in turn is controlled by 
$$
I(x):= \iint  |\nabla_{y,t}u_K (y,t)|^2 \varphi^2\bigg( \frac{x-y}t\bigg) \frac{t^{4\,K+1}dw(y)\,dt}{w(B(y,t))},
$$
where $\varphi$ is a smooth function such that $0\leq \varphi\leq 1$, $\varphi\equiv 1$ on $B(0,1)$ and $\varphi\equiv 0$ in $\re^n\setminus B(0,2)$. 
Write $\varphi_w(x,t):=\int_{\R^n}\varphi^2\left(\frac{x-z}{t}\right)w(z)\,dz$ and note that since $w\in A_2$ (and hence is doubling), we have
\begin{equation}
\frac{w(B(x,t))}{w(B(y,t))}
 \approx 1
\qquad\mbox{and}\qquad
1
\le
\frac{\varphi_w(x,t)}{w(B(x,t))}
\le \frac{w(B(x,2t))}{w(B(x,t))}
 \lesssim 1
\label{eq:jrjuju}
\end{equation}
whenever $|x-y|<2t$. For $0<r<R/10<\infty$, take $\psi_{r,R}(t)=\zeta(t/r)(1-\zeta(t/R))$, where $\zeta(t)$ is a smooth function that satisfies:
$\zeta(t)\equiv 0$ if $t\leq 1/2$ and $\zeta(t)\equiv 1$ if $t\geq 2$. Using all this and the monotone convergence theorem, it suffices to estimate
$$
I_{r,R}(x):= \iint  |\nabla_{y,t}u_K (y,t)|^2 \varphi^2\bigg( \frac{x-y}t\bigg) \psi_{r,R}^2(t) \frac{t^{4\,K+1}dw(y)\,dt}{\varphi_w(x,t)}.
$$
Let $B$ be the $(n+1)\times (n+1)$ block matrix with $A$ being one block and 1 the other one,  i.e., $B=\left(\begin{array}{cc}A & 0 \\0 & 1\end{array}\right)$.  Since the matrix $B$ is uniformly elliptic, $I_{r,R}(x) \le C(\lambda) \Re \calI_{r,R}(x)$ with
$$
\calI_{r,R}(x):= \iint  B(y)\nabla_{y,t}u_K(y,t) \cdot \overline{ \nabla_{y,t}u_K(y,t)}\,  \varphi^2\bigg( \frac{x-y}t\bigg)  \psi_{r,R}^2(t) \frac{t^{4\,K+1}dw(y)\,dt}{\varphi_w(x,t)}.
$$
Next, we write
\begin{align*}
\calI_{r,R}(x)
&=
\iint  B(y)\nabla_{y,t}u_K(y,t)\cdot  \overline{ \nabla_{y,t}(u_K-v_K)(y,t)} \, \varphi^2\bigg( \frac{x-y}t\bigg) \psi_{r,R}^2(t)  \frac{t^{4\,K+1}dw(y)\,dt}{\varphi_w(x,t)}
 \\
&\qquad\qquad + \iint  B(y)\nabla_{y,t}u_K(y,t) \cdot  \overline{ \nabla_{y,t}v_K(y,t)}\,  \varphi^2\bigg( \frac{x-y}t\bigg) \psi_{r,R}^2(t) \frac{t^{4\,K+1}dw(y)\,dt}{\varphi_w(x,t)}
\\
&=: \calI_{r,R}^1(x) + \calI_{r,R}^2(x).
\end{align*}
In the last integral, distribute  the product $\varphi\psi_{r,R}$ on each gradient term  and  use Young's inequality with  $\epsilon$ to obtain a bound
$$
|\calI_{r,R}^2(x)|
\le
 \epsilon I_{r,R}(x)
 +
 C\|B\|_\infty^2 \epsilon^{-1} \iint_{|x-y| <  2t}  |\nabla_{y,t}v_K(y,t)|^2  \frac{t^{4\,K+1}dw(y)\,dt}{\varphi_w(x,t)}.
$$
Using that
$$
t^{2K}(t\nabla_{y,t}v_{K}(y,t))
=
t\nabla_{y,t}(t^{2K}v_{K}(y,t))-2K t^{2K}v_{K}(y,t)  \vec{e},
$$
and \eqref{eq:jrjuju}
we can  obtain
\begin{multline*}
|\calI_{r,R}^2(x)|
\le
 \epsilon I_{r,R}(x)
 +
 C\|B\|_\infty^2\epsilon^{-1} K \iint_{|x-y| <  2t}  |t^{2K} v_K(y,t)|^2  \frac{dw(y)\,dt}{tw(B(y,t))}
\\
  +
 C\|B\|_\infty^2 \epsilon^{-1} \iint_{|x-y| <  2t}  |t\nabla_{y,t}(t^{2K }v_K(y,t))|^2  \frac{dw(y)\,dt}{tw(B(y,t))}.
\end{multline*}

To estimate $\mathcal{I}^1_{r,R}$ we first observe that $w^{-1}\,\div_{y,t}\big(w(y) B(y)\nabla_{y,t}u_K(y,t)\big)=0$ in the weak sense in $\re^{n+1}_+$ with respect to the inner product in $L^2(\re^{n+1}_+,dw\,dt)$. This and Leibniz's rule give
$$
\mathcal{I}_{r,R}^1(x)
=
-\iint B(y)\nabla_{y,t}u_K(y,t)\cdot\nabla_{y,t}
\left\{
\,\varphi^2\left(\frac{x-y}{t}\right)\psi_{r,R}^2(t)\frac{t^{4K+1}}{\varphi_w(x,t)}\right\}\,\overline{(u_K-v_K)(y,t)}\,dw(y)\,dt.
$$
To estimate this we first observe that easy calculations lead to
\begin{align*}
F(y,t)
&:=
\left|\nabla_{y,t}
\left\{
\,\varphi^2\left(\frac{x-y}{t}\right)\psi_{r,R}^2(t)\frac{t^{4K+1}}{\varphi_w(x,t)}\right\}\right|
\\
&\lesssim
\frac{t^{4K}\psi_{r,R}(t)}{\varphi_w(x,t)}
\varphi\left(\frac{x-y}{t}\right)\,
\Bigg\{
\psi_{r,R}(t)
\left|(\nabla \varphi)\left(\frac{x-y}{t}\right)\right|
+
\psi_{r,R}(t)
\varphi\left(\frac{x-y}{t}\right)
+ t|\psi_{r,R}'(t)|\varphi\left(\frac{x-y}{t}\right)
\\
&
\qquad+
\psi_{r,R}(t)\frac{|x-y|}{t}\left|(\nabla \varphi)\left(\frac{x-y}{t}\right)\right|
+
t\,\psi_{r,R}(t)\varphi\left(\frac{x-y}{t}\right)\frac{|\partial_t(\varphi_w(x,t))|}{\varphi_w(x,t)}
\Bigg\}.
\end{align*}
Note that $|\psi_{r,R}'(t)|\lesssim t^{-1}$ uniformly  in  $r$ and $R$. Also, using \eqref{eq:jrjuju} and the properties of $\varphi$ it follows that $|\partial_t(\varphi_w(x,t))|\lesssim t^{-1}\varphi_w(x,t)$. These and the way that $\varphi$ and $\psi_{r,R}$ have been chosen easily lead to
$$
F(y,t)\lesssim
\frac{t^{4K}\psi_{r,R}(t)}{\varphi_w(x,t)}
\varphi\left(\frac{x-y}{t}\right)\,\Theta\left(\frac{x-y}{t}\right),
$$
where $\Theta=\varphi+|\nabla\varphi|$ is a bounded function supported in $B(0,2)$. We can use this, 
Young's  inequality with $\epsilon>0$, and \eqref{eq:jrjuju} to estimate $\mathcal{I}_{r,R}^1(x)$:
\begin{align*}
&|\mathcal{I}_{r,R}^1(x)|
\\
&\ \le
\iint \left\{|\nabla_{y,t}u_K(y,t)|\varphi\left(\frac{x-y}{t}\right)\psi_{r,R}(t)\right\}
\left\{C\,\|B\|_\infty\,t^{-1}\Theta\left(\frac{x-y}{t}\right)|(u_K-v_K)(y,t)|\right\}
\,\frac{t^{4\,K+1}dw(y)\,dt}{\varphi_w(x,t)}
\\
&\ \le
 \epsilon I_{r,R}(x)
+
C\|B\|_\infty^2\epsilon^{-1}
\iint_{|x-y|<2t} \big|t^{2K}(u_K-v_K)(y,t)\big|^2
\,\frac{dw(y)\,dt}{tw(B(y,t))}.
\end{align*} 
Collecting the estimates that we have obtained  and  recalling the definitions of $u_K$, $v_K$,  we conclude
\begin{align*}
|\calI_{r,R}(x)|
&\le
2\,\epsilon I_{r,R}(x)
+
 C\|B\|_\infty^2\epsilon^{-1} K 
\iint_{|x-y| <  2t} \abs{(t^2L_w)^K e^{-t^2L_w}f(y)}^2 \frac{dw(y) \, dt}{tw(B(y,t))}
\\
 &
\qquad
+
 C\|B\|_\infty^2 \epsilon^{-1} 
\iint_{|x-y| <  2t} \abs{t\nabla_{y,t}(t^2L_w)^K e^{-t^2L_w}f(y)}^2\frac{dw(y)\,dt}{tw(B(y,t))}
\\
&\qquad
+
C\|B\|_\infty^2\epsilon^{-1}
\iint_{|x-y|<2t}
\abs{(t^2L_w)^K(e^{-t\sqrt{L_w}}-e^{-t^2L_w})f(y)}^2
\,\frac{dw(y)\,dt}{tw(B(y,t))},
\end{align*}
where all the constants are uniform  in  $r$, $R$, and $x$. Recalling that  $I_{r,R}(x) \le C(\lambda) \Re \calI_{r,R}(x)$ we can hide the first term in the right-hand side of the previous estimate  (which is finite thanks to the different cut-off functions)  by choosing $\epsilon$ small enough. Letting then $r\searrow 0$ and $R\nearrow \infty$,   one derives \eqref{GKP} for functions $f\in L^2(w)$ when $K=0$ and for functions $f$ is in the domain of $L_w^K$ when $K\ge 1$. 

 To complete the proof we explain how to extend \eqref{GKP} to arbitrary functions in $L^2(w)$. Let us fix $K\ge 1$ and write $\mathcal{T}$  to denote the sublinear operator
defined from the right-hand  side of the inequality \eqref{GKP}. Note that combining Proposition \ref{prop:alpha}, \eqref{est:gfrak}, Theorem \ref{theor:control-SF-Heat}, and the trivial case $p=2$ of Proposition \ref{thm:boundednessGH} we conclude that for all $f\in L^2(w)$
\begin{equation}\label{eq:frqweaf}
\|\mathcal{T}f\|_{L^2(w)}
\lesssim
\|\Gcal_{\hh}^{L_w}f\|_{L^2(w)}
\lesssim
\|f\|_{L^2(w)}.
\end{equation}
We fix $f\in L^2(w)$ and our goal is to show that $\Gcal_{K,\pp}^{L_w}f(x)\lesssim \mathcal{T}f(x)$ for almost every $x\in\re^n$. To that end, we use  that the domain of $L_w^K$ is dense in $L^2(w)$ and find a sequence $\{f_j\}_j$ contained in the domain of $L_w^K$ such that $f_j\to f$ in $L^2(w)$ as $j\to\infty$. 
Without loss of generality we may also assume that $\mathcal{T}(f-f_j)(x)\to 0$ for a.e. $x\in\re^n$ as $j\to\infty$. Indeed,  from \eqref{eq:frqweaf} we know that $\mathcal{T}(f-f_j)\to 0$ in $L^2(w)$ as $j\to\infty$ and therefore, passing to a subsequence, the convergence occurs almost everywhere. 
On the other hand,  $t\nabla_{y,t}(t\sqrt{L_w}\,)^{2K} e^{-t\sqrt{L_w}}$ is uniformly bounded on $L^2(w)$ and it follows from \eqref{doublingcondition} that 
for every $N,j\ge 1$ and every $x\in\re^n$
\begin{align*}
&\left(\int_{N^{-1}}^N \int_{|x-y|<t}|t\nabla_{y,t}(t\sqrt{L_w}\,)^{2K} e^{-t\sqrt{L_w}}f(y)|^2 \frac{dw(y) \, dt}{tw(B(y,t))}\right)^{\frac{1}{2}}
\\
&
 \qquad\lesssim
\frac{\log N}{w(B(x,N^{-1}))}\|f-f_j\|_{L^2(w)}
+
\left(\iint_{|x-y|<t}|t\nabla_{y,t}(t\sqrt{L_w}\,)^{2K} e^{-t\sqrt{L_w}}f_j(y)|^2 \frac{dw(y) \, dt}{tw(B(y,t))}\right)^{\frac{1}{2}}
\\
&
\qquad\lesssim
\frac{\log N}{w(B(x,N^{-1}))}\|f-f_j\|_{L^2(w)}
+
\mathcal{T}f_j(x)
\\
&
\qquad\lesssim
\frac{\log N}{w(B(x,N^{-1}))}\|f-f_j\|_{L^2(w)}
+
\mathcal{T}(f-f_j)(x)
+
\mathcal{T}f(x),
\end{align*}
where in the second inequality we have used  \eqref{GKP} for $f_j$, which by construction is in the domain of $L_w^K$. Next, we first let  $j\to\infty$ and then $N\to\infty$ to conclude as desired that \eqref{GKP} holds for $f\in L^2(w)$. The proof of Lemma \ref{lemma:caccio} is now complete. 
\end{proof}


\section{Unweighted boundedness for square functions}\label{section:unweighted}

In this section we prove unweighted estimates, that is,  on $L^p(\re^n)$, for the conical square functions associated with the heat or Poisson semigroups associated with $L_w$. These will be obtained as a consequence of their weighted boundedness on $L^p(vdw)$ by simply taking $v=w^{-1}$ on Theorems \ref{thm:SF-Heat} and \ref{thm:SF-Poisson}. In order to check that the corresponding result can be applied we will need to make additional assumptions on $w\in A_2$. In particular, we are interested in specific examples of power weights $|x|^{\alpha}$, $-n<\alpha<n$, and their associated family of degenerate operators $L_{|x|^\alpha}=-|x|^{-\alpha}\div(|x|^{\alpha}\,A\,\nabla)$.

Before stating our results we need to recall some definitions. Given $w\in A_\infty$, the ``critical'' exponents $r_w$ and $s_w$ were defined in \eqref{eq:defi:rw}. By ``self-improving'' properties of the $A_p$ and $RH_s$ classes it follows that if $w\in A_r$ with $r>1$ then $r_w<r$ and, analogously, if $w\in RH_{q'}$ with $q>1$ then $s_w<q$.

We also note that as observed above there is a
``duality'' relationship between the weighted and unweighted
$A_p$ and $RH_s$ conditions:  $v=w^{-1} \in A_p(w)$ if and only if $w \in
RH_{p'}$ and $v=w^{-1}\in RH_{s'}(w)$ if and only if $w\in A_{s}$. We also recall that $2_w^*=\frac{2\,n\,r_w}{n\,r_w-2}$ if $2<n\,r_w$ and $2_w^*=\infty$ otherwise.

We start considering the conical square functions associated with the heat semigroup.

\begin{corollary}\label{cor-unweightedHeat}
Let $L_w$ be a degenerate elliptic operator with $w\in A_2$. If $p>(2^*_w)'\,s_w$ then the conical square functions $\Scal_{m,\hh}^{L_w}$ for $m\in\N$, and $\Gcal_{m,\hh}^{L_w}$ and $\Grm_{m,\hh}^{L_w}$ for $m\in\mathbb{N}_0$,  are all bounded on $L^p(\re^n)$. In particular, this is the case in the following situations:
\begin{list}{$(\theenumi)$}{\usecounter{enumi}\leftmargin=1cm \labelwidth=1cm\itemsep=0.2cm\topsep=.2cm \renewcommand{\theenumi}{\alph{enumi}}}

\item If $\frac{2n}{n+2}<p<\infty$ and $w\in A_1\cap  RH_{\big(\frac{p(n+2)}{2n}\big)'}$. 

\item If $1<r\le 2$, $\frac{2nr}{nr+2}\le p<\infty$ and $w\in A_r\cap  RH_{\big(\frac{p(nr+2)}{2nr}\big)'}$. 
\end{list}
Hence, all the previous square functions are bounded on $L^2(\re^n)$ if $w\in A_r\cap RH_{\frac{n}2 r+1}$ for $1\le r\le 2$.

\end{corollary}

\begin{proof}
Fix $p>(2^*_w)'\,s_w$ and note that by \eqref{p-p+} we have $p>(2^*_w)'\ge p_-(L_w)$. Also, $s_w <p/(2^*_w)'\le p/p_-(L_w)$ and hence $w\in RH_{\left(\frac{p}{p_-(L_w)}\right)'}$ or, equivalently, $v:=w^{-1}\in A_{\frac{p}{p_-(L_w)}}(w)$. These facts imply that $p\in\mathcal{W}_v^w(p_-(L_w),\infty)$, and then, Theorem  \ref{thm:SF-Heat} gives immediately the boundedness on $L^p(v\,dw)=L^p(\re^n)$ of all the conical square functions in the statement. 

To see that the situation in $(a)$ falls within the conditions stated above, we first consider the case $n=2$.  Our current assumptions give $r_w=1$ and $w\in RH_{\big(\frac{p(n+2)}{2n}\big)'}=RH_{p'}$, hence  $2^*_w=\infty$ and $p>s_w$ as desired.   On the other hand, if $n\ge 3$, using again that $r_w=1$ it follows that $(2^*_w)'=(2n/(n-2))'=2n/(n+2)$. In turn, $w\in RH_{\big(\frac{p(n+2)}{2n}\big)'}$ implies that 
$p>(2^*_w)'\,s_w$.

We now examine the conditions in $(b)$. Note that in that case since $w\in A_r$ with $r>1$, then $r_w<r$. This implies that 
$(2_w^*)'<\frac{2\,n\,r}{n\,r+2}$. Then, the assumptions on $p$ and $w$ easily yield that $p>(2^*_w)'\,s_w$.

Concerning the estimates on $L^2(\re^n)$ we just need to combine $(a)$ and $(b)$ with $p=2$.
\end{proof}

Next, we consider the conical square functions associated with the Poisson semigroup. Let us make some comments first. Given $w\in A_\infty$, $1\le q<\infty$, and $K\ge 1$ set
\begin{equation}\label{SobConj}
q_w^{K,\star}
:=\left\{
\begin{array}{ll}
\dfrac{qn r_w}{n r_w-Kq}, &\quad\mbox{ if}\quad Kq<nr_w,
\\[10pt]
\infty, &\quad\mbox{ if}\quad Kq\ge nr_w.
\end{array}
\right.
\end{equation}
When $K=1$ we write $q_w^\star:=q_w^{1,\star}$.  Notice that according to the notation introduced above $2_w^*=2_w^\star$ and also $\br{p_+(L_w)}_w^{K,*}=\br{p_+(L_w)}_w^{2K+1,\star}$ for every $K\in\N_0$ (cf. \eqref{indexPoisson}). On the other hand, it is easy to see that $q_w^{K+1,\star}=\br{q_w^{K,\star}}_w^{\star}=\br{q_w^{\star}}_w^{K,\star}$ for every $K\in\N$. 

It might be convenient to observe that 
$$
q_w^{K,\star}=\left[\br{\frac1q-\frac K{nr_w}}^{+}\right]^{-1},
$$ 
where $r^{+}=\max\{r,0\}$. Using this, it is easy to see that $q_w^{K+1,\star}\ge q_w^{K,\star}$ for every $K\in\N$. Moreover, if $K\in\N$ is fixed, then $q_w^{K,\star}$ is an increasing function of  $q$. Hence, since $p_+(L_w)\ge 2_w^{*}=2_w^{\star}$ (see \eqref{p-p+}), it follows that for all $K\in \N$, one has 
$(p_+(L_w))_w^{K,*}\ge  (p_+(L_w))_w^* \ge 2_w^{**}$, where, as defined above,
\begin{equation}
2_w^{**}
:=
\br{2_w^{\star}}_w^{\star}
=
2^{2,\star}_w
=
\begin{cases}
\frac{2nr_w}{nr_w-4}&\textrm{if}\quad 4<nr_w,
\\
\infty&\textrm{if}\quad 4\geq nr_w.
\end{cases}
\label{eq:defi-2-Sob-Sob}
\end{equation}

In the following result we present some $L^p(\re^n)$ estimates for  the conical square functions associated with the Poisson semigroup. As seen from Theorem \ref{thm:SF-Poisson}, these square functions  are ``better'' as the parameter $K$ increases ---since $\mathcal{W}_v^w(p_-(L_w),\br{p_+(L_w)}_w^{K,*})\subset \mathcal{W}_v^w(p_-(L_w),\br{p_+(L_w)}_w^{K+1,*})$, for all $K\in \N_0$. To simplify the statements we will compute the conditions that arise from the case $K=0$, and, in particular, all the square functions in Theorem \ref{thm:SF-Poisson} will be bounded under the same conditions. Having said that, if one targets a particular square functions with a given parameter $K$, the following result and its proof can be sharpened to provide both better ranges of $p$'s where the $L^p(\re^n)$-boundedness happens and also bigger classes of  weights, see Remark \ref{rem:sharpen} below.

\begin{corollary}\label{cor-unweightedPoisson}
Let $L_w$ be a degenerate elliptic operator with $w\in A_2$. If $ (2^*_w)'\,s_w<p<\frac{2^{**}_w}{r_w}$ then the conical  square   functions $\Scal_{K,\pp}^{L_w}$ for $K\in\N$, and $\Gcal_{K,\pp}^{L_w}$ and $\Grm_{K,\pp}^{L_w}$ for $K\in\mathbb{N}_0$,
are all bounded on $L^p(\re^n)$. In particular, this is the case in the following situations:
\begin{list}{$(\theenumi)$}{\usecounter{enumi}\leftmargin=1cm \labelwidth=1cm\itemsep=0.2cm\topsep=.2cm \renewcommand{\theenumi}{\alph{enumi}}}

\item If $\frac{2n}{n+2}<p<\infty$ for $n\le 4$ or if $\frac{2n}{n+2}<p<\frac{2n}{n-4}$ for $n>4$, and $w\in A_1\cap  RH_{\big(\frac{p(n+2)}{2n}\big)'}$.

\item If $1<r\le 2$, $\frac{2nr}{nr+2}\le p<\infty$ for $nr\le 4$ or $\frac{2nr}{nr+2} \le p \le \frac{2n}{nr-4}$ for $nr>4$, and $w\in A_r\cap  RH_{\big(\frac{p(nr+2)}{2nr}\big)'}$. 
\end{list}
Hence, all the previous square functions are bounded on $L^2(\re^n)$ if $w\in A_r\cap RH_{\frac{n}2r+1}$, for $1\le r\le \min\left\{2,1+\frac{4}{n}\right\}$.
\end{corollary}

\begin{remark}\label{rem:sharpen}
 Let us mention that in the scenario $(b)$ when $nr>4$ it could happen that there is no value of $p$ satisfying the required conditions unless $r$ is sufficiently close to $1$ depending on dimension. This happens because in the previous result we allow small values of $K$. Indeed, if we just fix $K\in\mathbb{N}$, the same argument, with appropriate changes, will give the range $ (2^*_w)'\,s_w<p<\frac{2^{2(K+1),\star}_w}{r_w}$. In particular, in $(a)$ we will have
$\frac{2n}{n+2}<p<\infty$ for $n\le 4(K+1)$ and $\frac{2n}{n+2}<p<\frac{2n}{n-4(K+1)}$ if $n>4(K+1)$. Analogously, in the context of $(b)$ we would obtain 
$\frac{2nr}{nr+2}\le p<\infty$ for $nr\le 4(K+1)$ or $\frac{2nr}{nr+2} \le p \le \frac{2n}{nr-4(K+1)}$ for $nr>4(K+1)$. In particular, taking $K\ge \max\{\frac{n-5}4,0\}$ the latter range will be non-empty regardless of $r$. Further details and the precise statements are left to the interested reader. 
\end{remark}

\begin{proof}
Using the ideas in the proof of Corollary \ref{cor-unweightedHeat} and the previous comments we need to see that if we set $v:=w^{-1}$ then $p\in  \mathcal{W}_v^w(p_-(L_w),\br{p_+(L_w)}_w^{*})$. This amounts to checking that $p_-(L_w)<p<\br{p_+(L_w)}_w^{*}$ and also that
$$
w^{-1}\in A_{\frac{p}{p_-(L_w)}}(w)\cap RH_{\left(\frac{\br{p_+(L_w)}_w^{*}}{p}\right)'}(w)
\quad\Longleftrightarrow\quad
w\in A_{\frac{\br{p_+(L_w)}_w^{*}}{p}}\cap RH_{\left(\frac{p}{p_-(L_w)}\right)'}.
$$
Let us observe that in the proof of Corollary \ref{cor-unweightedHeat} we took care on the lower bound for $p$ and the membership of $w$ to a reverse Hölder class by assuming that $p>(2^*_w)'\,s_w$. Now we need to look at the upper bound and the membership to a Muckenhoupt class. That is, we need to see that our assumption $p<2^{**}_w/r_w$ guarantees that $p<\br{p_+(L_w)}_w^{*}$ and that $w\in A_{\frac{\br{p_+(L_w)}_w^{*}}{p}}$. But, this follows at once from the estimates $p<2^{**}_w/r_w\le \br{p_+(L_w)}_w^{*}/r_w\le \br{p_+(L_w)}_w^{*}$.

Much as before, we now see that the situations described in $(a)$ and $(b)$ give the desired restrictions on $p$ depending on $s_w$ and $r_w$. We start with $(a)$. We showed in the proof of Corollary \ref{cor-unweightedHeat} that $p>2n/(n+2)$ and $w\in A_1\cap  RH_{\big(\frac{p(n+2)}{2n}\big)'}$ imply that $p>(2^*_w)'\,s_w$. On the other hand, taking into account that $r_w=1$ in the present scenario, we see that the upper bound assumed in $p$ can be rewritten as $p<2^{**}_w$ as desired.

Turning our attention to the scenario in $(b)$, and looking again at the proof of Corollary \ref{cor-unweightedHeat}, we know that 
if $1<r\le 2$, $\frac{2nr}{nr+2}\le p<\infty$ and $w\in A_r\cap  RH_{\big(\frac{p(nr+2)}{2nr}\big)'}$ then $p>(2^*_w)'\,s_w$. Again, since $w\in A_r$ with $r>1$ it follows that $r_w<r$. Note that if $nr\le 4$ then $nr_w<nr<4$ and hence $2^{**}_w=\infty$ in which case we can take $p$ as larger as we wish. On the other hand, if $nr>4$ then one can easily see that $p\le 2n/(nr-4)<2^{**}_w/r_w$ as desired. 

Concerning the estimates on $L^2(\re^n)$ we just need to combine $(a)$ and $(b)$ with $p=2$.
\end{proof}

\bigskip

Finally, we consider the case of power weights.
Define now 
$w_\alpha(x):=|x|^{\alpha}$, $\alpha>-n$; this restriction guarantees that
$w_\alpha$ is locally integrable.   We can exactly determine the
Muckenhoupt $A_p$ and reverse H\"older $RH_s$ classes of these weights
in terms of $\alpha$:   if $-n< \alpha\le 0$, then $w_{\alpha}\in A_1$; for
$1<p<\infty$,  $w_{\alpha}\in A_p$ if $-n<\alpha <n\,(p-1)$.  Furthermore,  if $0\le
\alpha<\infty$, $w_{\alpha}\in RH_\infty$;  for  $1<q<\infty$, $w_{\alpha}\in RH_q$,  if
$-n/q<\alpha<\infty$. Hence, we easily see that
\begin{equation}\label{sw-rw:w-alpha}
r_{w_\alpha}=\max\bbr{1,1+\frac{\alpha}{n}} \quad \text{and} \quad s_{w_\alpha}=\max\bbr{1,\br{1+\frac{\alpha}{n}}^{-1}}.
\end{equation}  
Using all these and Corollaries \ref{cor-unweightedHeat} and \ref{cor-unweightedPoisson}  we obtain the following result whose proof is left to the interested reader.

\begin{corollary}\label{cor-PW}
Let $A$ be  an $n\times n$ complex-valued matrix that satisfies the uniform ellipticity condition \eqref{eq:elliptic-intro} and consider $L_{w_\alpha}=-w_{\alpha}^{-1}\div(w_{\alpha}\,A\,\nabla)$ with $-n<\alpha<n$.

\begin{list}{$(\theenumi)$}{\usecounter{enumi}\leftmargin=1cm \labelwidth=1cm\itemsep=0.2cm\topsep=.2cm \renewcommand{\theenumi}{\alph{enumi}}}

\item For $-\frac{2n}{n+2}<\alpha<n$, all the square functions in Theorem \ref{thm:SF-Heat} (the ones associated with the heat semigroup) are bounded on $L^2(\re^n)$;

\item For $-\frac{2n}{n+2}<\alpha<\min\{n,4\}$, all the square functions in Theorem \ref{thm:SF-Poisson} (the ones associated with the Poisson semigroup) are bounded on $L^2(\re^n)$.
\end{list}
\end{corollary}

\appendix

\section{Extrapolation on weighted measure spaces}\label{appendix}

In this section we present some extrapolation results where the underlying measure space is $(\R^n,w)$ with $w\in  A_{\infty}$. The statements and proofs are quite similar to the euclidean setting with the Lebesgue measure. As in  \cite{CruzMartellPerez}, we write the extrapolation theorem in terms of pairs of functions. To set the stage consider $\mathcal{F}$ a family of pairs $(f,g)$ of non-negative, measurable functions that are not identically zero. Given such a family $\mathcal{F}$, $0<p<\infty$, and a weight $v\in A_\infty(w)$, when we write
$$
\int_{\R^n}f(x)^{p}v(x)dw(x)\leq  C_{w,v,p}  \int_{\R^n}g(x)^{p}v(x)dw(x), \quad  (f,g)\in \mathcal{F},
$$ 
we mean that this inequality holds for all pairs $(f,g)\in \mathcal{F}$ and that the constant $  C_{w,v,p}  $ depends  
only on $p$, the $A_\infty(w)$ constant of $v$ (and the $A_\infty$ character of $w$ which is ultimately fixed). Note that in \cite{CruzMartellPerez} such inequalities appear both in the hypotheses and in the conclusion of the extrapolation results and hold for all pairs $(f,g)\in\mathcal{F}$ for which the left hand sides are finite. Here we do not make such assumptions and, in particular, we do have that the infiniteness of the left-hand side will imply that of the right-hand one. This formulation is more convenient for our purposes and, as pointed out in \cite[Section 3.1]{MartellPrisuelos}, it follows from the formulation
where the inequalities hold for pairs for which the left hand sides are finite.

The following result for $w=1$ can be found in \cite[Chapter 2]{CruzMartellPerez} and \cite[Section 3.1]{MartellPrisuelos}. The proof
can be easily obtained by adapting the arguments there replacing everywhere the Lebesgue measure by the measure $w$ and the Hardy-Littlewood maximal function by its ``weighted'' version $\mathcal{M}^w$ introduced in \eqref{weightedHLM}. Further details are left to the interested reader.

\begin{theorem}\label{theor:extrapol}
Let $\mathcal{F}$ be a given family of pairs $(f,g)$ of non-negative and not identically zero measurable functions.
\begin{list}{$(\theenumi)$}{\usecounter{enumi}\leftmargin=1cm
\labelwidth=1cm\itemsep=0.2cm\topsep=.2cm
\renewcommand{\theenumi}{\alph{enumi}}}

\item Suppose that for some fixed exponent $p_0$, $1\le p_0<\infty$, and every weight $v\in A_{p_0}(w)$,
\begin{equation*}
\int_{\mathbb{R}^{n}}f(x)^{p_0}\,v(x)dw(x)
\leq
 C_{w,v,p_0}
\int_{\mathbb{R}^{n}}g(x)^{p_0}\,v(x)dw(x),
\qquad\forall\,(f,g)\in{\mathcal{F}}.
\end{equation*}
Then, for all $1<p<\infty$, and for all $v\in A_p(w)$,
\begin{equation*}
\int_{\mathbb{R}^{n}}f(x)^{p}\,v(x)dw(x)
\leq
C_{w,v,p}
\int_{\mathbb{R}^{n}}g(x)^{p}\,v(x)dw(x),
\qquad\forall\,(f,g)\in{\mathcal{F}}.
\end{equation*}

\item Suppose that for some fixed exponent $q_0$, $1\le q_0<\infty$, and every weight $v\in RH_{q_0'}(w)$,
\begin{equation*}
\int_{\mathbb{R}^{n}}f(x)^{\frac1{q_0}}\,v(x)dw(x)
\leq
 C_{w,v,q_0}
\int_{\mathbb{R}^{n}}g(x)^{\frac1{q_0}}\,v(x)dw(x),
\qquad\forall\,(f,g)\in{\mathcal{F}}.
\end{equation*}
Then, for all $1<q<\infty$ and for all $v\in RH_{q'}(w)$,
\begin{equation*}
\int_{\mathbb{R}^{n}}f(x)^{\frac1q}\,v(x)dw(x)
\leq
C_{w,v,q}
\int_{\mathbb{R}^{n}}g(x)^{\frac1q}\,v(x)dw(x),
\qquad\forall\,(f,g)\in{\mathcal{F}}.
\end{equation*}

\item Suppose that for some fixed exponent $r_0$, $0< r_0<\infty$, and every weight $v\in A_\infty(w)$,
\begin{equation*}
\int_{\mathbb{R}^{n}}f(x)^{r_0}\,v(x)dw(x)
\leq
 C_{w,v,r_0}
\int_{\mathbb{R}^{n}}g(x)^{r_0}\,v(x)dw(x),
\qquad\forall\,(f,g)\in{\mathcal{F}}.
\end{equation*}
Then, for all $0<r<\infty$ and for all $v\in A_\infty(w)$,
\begin{equation*}
\int_{\mathbb{R}^{n}}f(x)^{r}\,v(x)dw(x)
\leq
C_{w,v,r}
\int_{\mathbb{R}^{n}}g(x)^{r}\,v(x)dw(x),
\qquad\forall\,(f,g)\in{\mathcal{F}}.
\end{equation*}
 \end{list}
\end{theorem}

The following result is a version of \cite[Proposition 3.30]{MartellPrisuelos}  in  our current setting.

\begin{proposition}\label{prop:Q} 
Let $w\in A_{r}$ and $v\in RH_{s'}(w)$ with $1\le r,s<\infty$. For every $1\le q\le s$, $0<\alpha\leq 1$ and $t>0$, there holds
\begin{multline}\label{G-alpha}
\int_{\mathbb{R}^n}\left(\int_{B(x,\alpha t)}|h(y,t)| \, \frac{dw(y)}{w(B(y,\alpha t))} \right)^{\frac{1}{q}}v(x)dw(x)
\\ \lesssim 
\alpha^{nr\left(\frac{1}{s}-\frac{1}{q}\right)} 
\int_{\mathbb{R}^n}\left(\int_{B(x,t)}|h(y,t)| \, \frac{dw(y)}{w(B(y,t))} \right)^{\frac{1}{q}}v(x)dw(x).
\end{multline}
\end{proposition}

\begin{proof}
We fix $t>0$, $0< \alpha \le 1$ and set
$$
G^{\alpha}(x,t):=\int_{B(x,\alpha t)}|h(y,t)| \,\frac{dw(y)}{w(B(y,\alpha t))}.
$$
For simplicity, $G(x,t):=G^1(x,t)$. 
Then, for any $1\le s_0<\infty$ and $v_0\in RH_{s_0'}(w)$, we have
\begin{align}\label{G-alpha:extr}
\int_{\mathbb{R}^n}G^{\alpha}(x,t)v_0(x)dw(x)
&=
\int_{\mathbb{R}^n} |h(y,t)|\, \frac{v_0w(B(y,\alpha t))}{w(B(y,\alpha t))} dw(y)
\\  \nonumber
&\lesssim 
\int_{\mathbb{R}^n} |h(y,t)|\,v_0w(B(y,t))\br{\frac{w(B(y,\alpha t))}{w(B(y,t))}}^{\frac{1}{s_0}}\,\frac{dw(y)}{w(B(y,\alpha t))}
\\  \nonumber
&=
\int_{\mathbb{R}^n} \int_{B(x,t)} |h(y,t)|\,\br{\frac{w(B(y,\alpha t))}{w(B(y,t))}}^{\frac{1}{s_0}-1}\,\frac{dw(y)}{w(B(y,t))} v_0(x)dw(x)
\\  \nonumber
&\lesssim 
\alpha^{n\,r\,\br{\frac{1}{s_0}-1}} \int_{\mathbb{R}^n} \int_{B(x,t)} |h(y,t)|\,\,\frac{dw(y)}{w(B(y,t))} v_0(x)dw(x)
\\  \nonumber
&=
\alpha^{n\,r\,\br{\frac{1}{s_0}-1}} \int_{\mathbb{R}^n} G(x,t) v_0(x)dw(x).
\end{align}
Note that the two inequalities follow from \eqref{pesosineq:RHq} and \eqref{pesosineqw:Ap}, respectively, and the second equality is obtained by using  Fubini's theorem. Let us observe that \eqref{G-alpha:extr} is the desired estimate when $q=1$.

To prove the case $q>1$ we next extrapolate from \eqref{G-alpha:extr}. Consider $\mathcal{F}$ the family of pairs
$$
(f,g)=\br{G^\alpha(\cdot,t)^{s_0}, \alpha^{n\,r\,(1-s_0)}\,G(\cdot,t) ^{s_0}},
$$
and notice that \eqref{G-alpha:extr} immediately gives that, for every $v_0\in RH_{s_0'}(w)$, $1\le s_0<\infty$,
$$
\int_{\R^n} f(x)^\frac1{s_0}v_0(x)\,dw(x) \leq C\,\int_{\R^n} g(x)^\frac1{s_0}v_0(x)\,dw(x),
$$
where $C$ does not depend on $\alpha$ or $t$. Next, apply Theorem \ref{theor:extrapol}, part $(b)$, to conclude that, for every $1<s<\infty$ and for every $v\in RH_{s'}(w)$,
$$
\int_{\R^n} G^{\alpha}(x,t)^{\frac{s_0}{s}}v(x)\,dw(x) \leq
C\,\alpha^{n\,r\,\left(\frac{1}{s}-\frac{s_0}{s}\right)}\,\int_{\R^n} G(x,t)^{\frac{s_0}{s}}v(x)\,dw(x),
$$
where $C$ does not depend on $\alpha$ or $t$ and where $1\le s_0<\infty$ is arbitrary.  From this, given $v\in RH_{s'}(w)$,  $1<q\le s<\infty$, we can conclude \eqref{G-alpha} by taking $s_0=s/q$.
\end{proof}

\end{document}